\documentclass{amsart}
\usepackage{amsmath,amssymb, amscd, stmaryrd,bm}
\usepackage{fullpage}
\usepackage{enumerate}
\usepackage{cite}
\usepackage{framed}
\usepackage[dvips,dvipdf]{graphicx}
\usepackage[usenames,dvipsnames]{color}
\usepackage{color}
\usepackage{tabularx}
\usepackage{array}
\usepackage{multirow}
\usepackage{changes}
\usepackage{subfig}
\usepackage{algorithm,algpseudocode}

\newtheorem{theorem}{Theorem}[section]
\newtheorem{proposition}[theorem]{Proposition}
\newtheorem{corollary}[theorem]{Corollary}

\theoremstyle{definition}

\theoremstyle{remark}
\newtheorem{remark}[theorem]{Remark}
\newtheorem{example}[theorem]{Example}

\newcommand{\mb}[1]{\ensuremath{\mathbf{#1}}}

\newcommand{\CC}{\mathbb C}
\newcommand{\NN}{\mathbb N}

\newcommand{\RR}{\mathbb R}

\newcommand{\ii}{\mathfrak i}
\newcommand{\bb}{\mathfrak b}

\newcommand{\cT}{\mathcal T}

\newcommand{\cL}{\mathcal L}
\newcommand{\cM}{\mathcal M}
\newcommand{\cA}{\mathcal A}
\newcommand{\cB}{\mathcal B}

\newcommand{\argmin}{\mathrm{arg\,min}}
\newcommand{\Dom}{\mathrm{Dom}}
\newcommand{\Spec}{\mathrm{Spec}}
\newcommand{\dist}{\mathrm{dist}}

\author[J. Ovall]{Jeffrey S. Ovall} \address{Jeffrey S. Ovall,
  Fariborz Maseeh Department of Mathematics and Statistics, Portland
  State University, Portland, OR 97201}
\email{jovall@pdx.edu}

\author[R. Reid]{Robyn Reid} \address{Robyn Reid, Fariborz Maseeh Department of Mathematics and Statistics, Portland
  State University, Portland, OR 97201}
\email{reid3@pdx.edu}

\begin{document}

\title[Eigenvector localization]{An algorithm for identifying eigenvectors exhibiting strong spatial localization} \date{\today}

\begin{abstract}
  We introduce an approach for exploring eigenvector localization
  phenomena for a class of (unbounded) selfadjoint
  operators.  More specifically, given a target region and a
  tolerance, the algorithm identifies candidate eigenpairs for which
  the eigenvector is expected to be localized in the target region to
  within that tolerance.  Theoretical results, together with detailed
  numerical illustrations of them, are provided that support our algorithm.
  A partial realization of the algorithm is described and tested,
  providing a proof of concept for the approach. 
\end{abstract}

\maketitle

\section{Introduction}\label{Intro}
This paper concerns the development of new computational methods for exploring
eigenvector localization phenomena
for selfadjoint elliptic eigenvalue problems,
\begin{align}\label{ModelProblem}
  \cL\psi\doteq -\nabla\cdot(A\nabla \psi)+V\psi=\lambda\psi\mbox{ in }\Omega\quad,\quad \psi=0\mbox{ on }\partial\Omega~,
\psi\not\equiv 0\mbox{ in }\Omega,
\end{align}
where $\Omega\subset\RR^d$ is a bounded, connected open set,
$V\in L^\infty(\Omega)$ is non-negative, and there are constants $c,C>0$ such that the
symmetric matrix
$A:\Omega\to\RR^{d\times d}$ satisfies
\begin{align*}
c\mb{v}^t\mb{v}\leq \mb{v}^tA(x)\mb{v}\leq C\mb{v}^t\mb{v}\mbox{ for
  all }\mb{v}\in\RR^d\mbox{ and a.e. } x\in\Omega~.
\end{align*}
When $d=2$, we require $A\in [L^{\infty}(\Omega)]^{d\times d}$; and
when $d>2$, we require that $A$ is uniformly Lipschitz in each of its components.
The operator
$\cL$ is viewed as an unbounded operator on $L^2(\Omega)$, with
domain $\Dom(\cL)=\{v\in H^1_0(\Omega):\, \cL v\in L^2(\Omega)\}$.  We denote the (real) spectrum
of $\cL$ by $\Spec(\cL)$, and recall that it consists of a sequence,
$\inf_\Omega V<\lambda_1<\lambda_2\leq\lambda_3\leq\cdots$, that has no finite accumulation points.  Furthermore, the eigenspace
$E(\lambda,\cL)=\{v\in\Dom(\cL):\cL v=\lambda v\}$ is finite dimensional for
each $\lambda\in\Spec(\cL)$.

The assumptions on $\cL$ guarantee that it has the \textit{unique
  continuation property}
(cf.\cite{Alessandrini1998,Garofalo1987,Hormander1983}), i.e. any
function $v\in H^1(\Omega)$ satisfying $\cL v=0$ in $\Omega$ that
vanishes on a non-empty open subset of $\Omega$ must vanish
identically on $\Omega$.  A simple consequence of this is that no
eigenvector $\psi$ of~\eqref{ModelProblem} may vanish identically on
any open subset of $\Omega$.  However, it may be the case that nearly
all of the ``mass'' of an eigenfunction $\psi$ is concentrated in a
non-empty, open, proper subset $R$ of $\Omega$.  In this case, we say
that $\psi$ is \textit{localized} in $R$.  For convenience, we will
refer to a non-empty, open, proper subset $R$ of $\Omega$ as a
\textit{subdomain} of $\Omega$. We now quantify
what we mean by localization in $R$.  Given a function $v\in L^2(\Omega)$ and
a subdomain $R$, the complementary quantities
\begin{align}\label{LocalizationQuantities}
\delta(v,R)=\|v\|_{L^2(\Omega\setminus R)}/\|v\|_{L^2(\Omega)}\quad,\quad \tau(v,R)=\|v\|_{L^2(R)}/\|v\|_{L^2(\Omega)}~,
\end{align}
provide measures of localization/concentration of $v$ within the
subdomain $R$.  Clearly,
\begin{align*}
\delta^2(v,R)+\tau^2(v,R)=1\quad,\quad
  \delta(v,R)=\tau(v,\Omega\setminus R)\quad,\quad  \tau(v,R)=\delta(v,\Omega\setminus R)~,
\end{align*}
and we have
$\delta(\psi,R),\tau(\psi,R)\in (0,1)$ for eigenvectors $\psi$ and any
subdomain $R$.  Given a tolerance $\delta^*\in(0,1/2)$,
we say that \textit{$v$ is localized in $R$ with tolerance $\delta^*$}
if $\delta(v,R)\leq\delta^*$ or, equivalently,
$\tau(v,R)\geq\sqrt{1-(\delta^*)^2}$.
Note that we do not require that $R$ is
connected, although it will be in many applications. 

Localization of eigenvectors may occur due to properties of the
coefficients $A$ and $V$, the shape of the domain, and/or the boundary
conditions.  A 2013 overview of the geometric structure of
eigenvectors for $\cL=-\Delta$ that highlights eigenvector localization due to domain
geometry is provided in~\cite{Grebenkov2013} (see
also~\cite{Nguyen2013,Hassell2009,Burq2005,Marklof2012}).  In a series of recent articles
\cite{Arnold2016,Arnold2019,Arnold2019a,Filoche2012}, starting in
2012, the authors investigate localization due to highly discontinuous
potentials $V$ (``disordered media''), providing some theoretical
insight into the mechanisms driving localization, a novel numerical
method for approximating likely subdomains in which localization may
occur, and a simple estimate of the smallest eigenvalue whose
eigenvector is localized in each such subdomain.
In~\cite{Filoche2012}, the authors state that
\begin{quote}
  ``...there has been no general theory able to directly determine for any domain and any type of inhomogeneity the precise relationship between the geometry of the domain, the nature of the disorder, and the localization of vibrations, to predict in which subregions one can expect localized standing waves to appear, and in which frequency range.''
\end{quote}
Although progress has been made by these authors and others
(cf.~\cite{Altmann2019,Altmann2020,Jones2019,Steinerberger2017,Lu2018})
in the meantime, there is still room for improvement, particularly on
the algorithmic front.  We describe the computational
approaches of~\cite{Arnold2019a} and~\cite{Altmann2019,Altmann2020},
which are presented for Schr\"odinger operators $\cL=-\Delta+V$,
before giving a summary description of our own.

Central to the work in~\cite{Arnold2016,Arnold2019,Arnold2019a,Filoche2012}
is the so-called \textit{landscape function} $u$ for $\cL$.
The following basic result, which is stated in each of these works, and
can be proved by applying the Maximum
Principle to $w_{\pm}=\lambda \|\psi\|_{L^{\infty}(\Omega)}u\pm \psi$,
provides pointwise bounds on an eigenvector in terms of its eigenvalue
and the landscape function. 
\begin{proposition}\label{LandscapeProposition}  
Let $(\lambda,\psi)$ be an eigenpair of~\eqref{ModelProblem}.  It holds that 
\begin{align}\label{LandscapeBound}
|\psi(x)|\leq \lambda\, u(x) \|\psi\|_{L^\infty(\Omega)}
\end{align}
pointwise  in  $\Omega$, where $u\in\Dom(\cL)$,  satisfies $\cL u=1$ in $\Omega$.
\end{proposition}
\noindent
An outline of the computational approach from~\cite{Arnold2019a} is:
\begin{enumerate}
\item Compute the landscape function $u$.
\item Determine several or all local minima of the associated
  \textit{effective potential} $W=1/u$, $W_k=W(x_k)$ for $\,1\leq
  k\leq N$, with $W_k\leq W_{k+1}$.
\item Estimate $N$ eigenvalues as $\tilde\lambda_k=(1+d/4)W_k$.  The
  factor $(1+d/4)$ is supported empirically and heuristically.
\item Choose the set $R_k$ to be the connected component of
  $\{x\in\Omega:\, W(x)\leq E\}$ that contains $x_k$, where $E\geq\tilde\lambda_k$ is a parameter to be set by the user.
  It is expected that $\delta(\psi,R_k)$ is small, where $\psi$ is an
  eigenvector associated with the eigenvalue of $\cL$ estimated by $\tilde\lambda_k$.
\end{enumerate}
We highlight a few features of this approach.  The most obvious is
that it does not involve the (approximate) solution of any eigenvalue
problems; only the solution of a single source problem is required.
From this, estimates (as opposed to convergent approximations) of
several eigenvalues are computed, together with localization regions
for eigenvectors whose eigenvalues are near the given estimates.
\textit{Eigenvector approximations are not provided}, though the
authors mention solving for the smallest eigenpair of $\cL$ with
homogeneous Dirichlet conditions on $R_k$, or some slightly larger
region, as an option.  We note that $\tilde\lambda_k$,
$1\leq k\leq N$, are not necessarily estimates of the first $N$
eigenvalues of $\cL$.  For example, there may be multiple eigenvectors
that are localized in one or more of the regions $R_1,\ldots,R_k$
whose eigenvalues are smaller than the smallest eigenvalue having an
eigenvector localized in $R_{k+1}$.  We emphasize that the method
of~\cite{Arnold2019a} estimates only the smallest eigenvalue having an
eigenvector localized in a each of its determined subdomains $R$; we
will refer to this eigenvalue (or eigenpair) as the \textit{ground
  state} for $R$.

In contrast, the approach of~\cite{Altmann2019,Altmann2020}, aims to
compute localized eigenvectors (and their eigenvalues), together with their localization
regions, by a two-phase iterative process.
Given a highly disordered, but structured, potential $V$, a fine mesh
$\cT^\varepsilon$ is generated that is deemed suitable for resolving
the lowermost part of the spectrum of $\cL$ via a finite element
method---the user determines the number $N$ of eigenpairs sought.  A (much)
coarser mesh $\cT^H$, of which $\cT^\varepsilon$ is a refinement, is
used in the first phase of the algorithm, whose aim is to provide
regions of localization, together with a basis for a rough
approximation of the space in which approximate eigenpairs will be
computed during the second phase.
Starting with finite element hat functions associated with the
vertices of $\cT^H$ (there should be at least $3N$ of them), a few approximate inverse iterations, using one
preconditioned conjugate gradient (PCG) step per iteration, are performed
during phase one,
with a mechanism involving Rayleigh quotients and a parameter
$\eta\in(0,1)$ used pair down the set of functions after each
iteration.  At the end of phase one, a basis for a coarse subspace of 
dimension at least $3N$ is obtained. Since $\cT^\varepsilon$ is a
refinement of $\cT^H$ the functions obtained in phase one are
already finite element functions (of the same degree) on $\cT^\varepsilon$.
In phase two, a few steps of approximate inverse iteration are again
used on the finer discretization, starting with the functions obtained
from phase one; this time, however, three PCG steps are used per
iteration.  After each inverse iteration, approximate
eigenpairs are obtain by a Rayleigh-Ritz procedure on the remaining
set of functions, and a similar mechanism is used to pair down the set
of functions (if needed) for the next iteration.  At the end of phase
two, a set of at least $2N$ approximate eigenpairs is obtained, and
the first $N$ of them are kept.  Variations on this basic algorithm
are presented in~\cite{Altmann2019}, and the description above comes
from Algorithm 1 in that paper.  This algorithm is clearly
more sophisticated and costly than that described above from~\cite{Arnold2019a}, but
it does provide approximations of eigenpairs, not just eigenvalues,
and the quality of these approximations can be controlled by parameters
in the discretization.  Additionally, this algorithm can find more than
one eigenvector that is localized in a given region, which was not the
case for the approach of~\cite{Arnold2019a}.  However, the authors
note that their algorithm assumes that the first $N$ eigenvectors are
localized, which is a reasonable assumption for the highly disordered
potentials they consider, but might be problematic for problems in
which other factors, such as domain geometry, are dominant drivers of
localization.
The approach of~\cite{Arnold2019a} does not assume that the first $N$
eigenfunctions are localized, but is currently limited to computing
estimates of as many eigenvalues and localization subdomains as there
are local minima of $W=1/u$.

Both approaches discussed above are aimed at the lower part of the
spectrum, and are best suited to situations in which this part of the
spectrum contains many localized eigenvectors.  Additionally, these
approaches offer no a priori control of how strongly localized an
eigenvector $\psi$ should be in a localization subdomain $R$
determined by their algorithms, i.e. how small $\delta(\psi,R)$ should
be, in order to consider it ``localized enough''.  Our approach puts
this consideration at the forefront, and targets with the following fundamental task:
\begin{equation}
  \tag{T}\label{KeyTask}
  \parbox{\dimexpr\linewidth-4em}{%
    \strut
    Given a subdomain $R$, an interval $[a,b]$ and a (small) tolerance $\delta^*>0$, find all
eigenpairs $(\lambda,\psi)$ for which $\delta(\psi,R)\leq \delta^*$ and $\lambda\in[a,b]$, or
determine that there are not any.%
    \strut
  }
\end{equation}
One might obtain reasonable candidates for such an $R$ using a
landscape function approach, as in~\cite{Arnold2019a}, or the first
phase of Algorithm 1 in~\cite{Altmann2019}, but for our purposes we
will just assume an $R$ is given. 
If $[a,b]$ contains relatively few eigenvalues of $\cL$ (counted by
multiplicity), this task can be accomplished reasonably efficiently
using existing technology: just compute (approximate) all eigenpairs
$(\lambda,\psi)$ for $\lambda\in[a,b]$ by your favorite method, and check $\delta(\psi,R)$ for
each.  However, if $[a,b]$ contains many eigenvalues of $\cL$, or we
have no a priori sense of how many eigenvalues it contains, an
approach that automatically filters out eigenvectors that are not
localized in $R$ is desirable.  Drawing inspiration from the work of
Marletta~\cite{Marletta2009,Marletta2012} on combating the effects of
spectral pollution in computing eigenvalues for operators having
essential spectrum, we also consider a complex-shifted version of the
operator (though our operator $\cL$ has no essential spectrum).
More specifically, given a subdomain $R$ and a number $s>0$,  we
define the normal operator $\cL_s$ by
\begin{align}\label{Ls}
\cL_s=\cL+\ii s\,\chi_R\quad,\quad \Dom(\cL_s)=\Dom(\cL)~.
\end{align}
The intuition behind our approach is that, if $(\lambda,\psi)$ is an
eigenpair of $\cL$ with $\psi$ highly localized in $R$, then their
ought to be an eigenpair $(\mu,\phi)$ of $\cL_s$ with $\mu$ near
$\lambda+\ii\,s$ and $\phi$ near $\psi$.  This intuition will be
justified theoretically in Section~\ref{Theory}.  With this in mind,
our algorithm template, Algorithm~\ref{ELAT}, begins by finding 
eigenpairs $(\mu,\phi)$ of $\cL_s$ for which
$\Im \mu$ is near $s$.  These eigenpairs of $\cL_s$ are then
``post-processed'' to find eigenpairs $(\lambda,\psi)$ of $\cL$
for which $\psi$ is likely to be localized in $R$.
The parameter $\delta^*$ governs both whether $\Im(\mu)$ is ``near enough'' to
$s$,
and whether $\delta(\psi,R)$ is ``small enough''.

The rest of the paper is outlined as follows.  In
Section~\ref{Theory}, we establish the key theoretical results that
naturally lead to an algorithm template for~\eqref{KeyTask}, and
illustrate several of the results and ideas of this section via two 1D examples for which
discretization is not needed for computing eigenpairs of $\cL$ and
$\cL_s$.  The algorithmic template itself is given in
Section~\ref{Template}, together with a description of one reasonable
choice for computing eigenpairs of $\cL_s$ that lie in a target
region---the FEAST method.  Section~\ref{Experiments} contains several
experiments illustrating the performance a practical realization of
the algorithm using finite element discretizations to approximate
eigenpairs.  Is Section~\ref{Conclusions} we offer a few concluding remarks.

\section{Theoretical Results and Illustrations}\label{Theory}
Given a proper subdomain $R$ and an $s>0$, we now explain and justify what we mean
by asserting that,
\begin{quote}
If $(\lambda,\psi)$ is an eigenpair of $\cL$ with
$\psi$ highly localized in $R$, then there is an eigenpair
$(\mu,\phi)$ of $\cL_s$, defined in~\eqref{Ls}, that is close to the eigenpair
$(\lambda+\ii\,s,\psi)$ of $\cL+\ii\,s$.
\end{quote}
A natural analogue of this assertion, with the roles of $\cL$ and $\cL_s$ reversed, also holds,
and both will be considered below, after first establishing some
simple bounds on the real and imaginary parts of eigenvalues of $\cL_s$.
Suppose that $(\mu,\phi)$ is an eigenpair of $\cL_s$. 
We have 
\begin{align}\label{RayleighQuotient}
  \mu=\frac{(\cL \phi,\phi)}{\|\phi\|_{L^2(\Omega)}^2}+\ii\,s[\tau(\phi,R)]^2~,
\end{align}
where $(\cdot,\cdot)$ is the (complex) inner-product on $L^2(\Omega)$.
It follows by the unique continuation principle and the variational
characterization of the eigenvalues of $\cL$ that
\begin{proposition}\label{BoundsOnImaginaryPart}
For any $\mu\in\Spec(\cL_s)$, $0<\Im \mu<s$ and $\Re\mu>\lambda_1(\cL)=\min\Spec(\cL)$.
\end{proposition}

We now present the first of two key results concerning the eigenpairs of $\cL$,
$\cL_s$ and $\cL+\ii\,s$. 
\begin{theorem}\label{KeyTheorem}
Let $(\lambda,\psi)$ be an eigenpair of $\cL$.  Then 
\begin{align}\label{EigenvalueCloseness}
\dist(\lambda+\ii\,s,\Spec(\cL_s))\leq s\, \delta(\psi,R)~.
\end{align}
Let $\mu=\argmin\{|\lambda+\ii\,s-\sigma|:\,\sigma\in\Spec(\cL_s)\}$.
If $M\subset\Spec(\cL_s)$ contains $\mu$, then
\begin{align}\label{EigenvectorCloseness}
\inf_{v\in
  E(M,\cL_s)}\frac{\|\psi-v\|_{L^2(\Omega)}}{\|\psi\|_{L^2(\Omega)}}\leq
  \frac{s\,\delta(\psi,R)}{\dist(\lambda+\ii\,s,\Spec(\cL_s)\setminus M)}~,
\end{align}
where $E(M,\cL_s)=\bigoplus\{E(\sigma,\cL_s):\,\sigma\in M\}$ is the
invariant subspace for $M$.
\end{theorem}
\begin{proof}
Since $\lambda+\ii\,s\notin \Spec(\cL_s)$, $(\lambda+\ii\,s-\cL_s)^{-1}$ is bounded and normal, and we have (cf.\cite[Chapter 5, Section 3.5]{Kato1995})
\begin{align*}
\|(\lambda+\ii\,s-\cL_s)^{-1}\|^{-1}=\dist(\lambda+\ii\,s,\Spec(\cL_s))~.
\end{align*}
We see by direct computation that
$(\lambda+\ii\,s-\cL_s)\psi=\ii\,s\,\chi_{\Omega\setminus R}\psi$,
from which we obtain
\begin{align*}
  \|\psi\|_{L^2(\Omega)}\leq \|(\lambda+\ii\,s-\cL_s)^{-1}\|\,s\|\psi\|_{L^2(\Omega\setminus R)}~.
\end{align*}
Rearranging terms, and using the definition of $\delta(\psi,R)$, 
completes the proof of~\eqref{EigenvalueCloseness}.   

Now let $P_s=\frac{1}{2\pi\ii}\int_\gamma(z-\cL_s)^{-1}\,dz$ denote
the (orthogonal) spectral projector for $E(M,\cL_s)$.  Here, $\gamma\subset\CC$ is
any simple closed contour enclosing $M$ and excluding
$\Spec(\cL_s)\setminus M$. 
Noting that $P_s$ commutes with $\cL_s$, direct algebraic manipulation
reveals that
\begin{align}\label{ComplementaryProjectorIdentity}
I-P_s=(\lambda+\ii\,s-\cL_s (I-P_s))^{-1}(I-P_s)(\lambda+\ii\,s-\cL_s)~.
\end{align}
It follows that
\begin{align*}
(I-P_s)\psi=(\lambda+\ii\,s-\cL_s (I-P_s))^{-1}(I-P_s)\,\ii\,s \chi_{\Omega\setminus R}\psi~,
\end{align*}
Therefore,
\begin{align*}
\|(I-P_s)\psi\|_{L^2(\Omega)}\leq \frac{s\|\psi\|_{L^2(\Omega\setminus
  R)}}{\dist(\lambda+\ii\,s,\Spec(\cL_s)\setminus M)}~,
\end{align*}
which establishes~\eqref{EigenvectorCloseness}.  We note that it was not necessary that $M$ contain
$\mu$ for~\eqref{EigenvectorCloseness} to hold, but that including
$\mu$ in $M$ makes the demoninator larger.  Of course, for~\eqref{EigenvectorCloseness}
to be a meaningful bound, the term on the right must be less than $1$,
so including (at least) $\mu$ in $M$ is prudent.
\end{proof}

A simple consequence of Proposition~\ref{BoundsOnImaginaryPart} and
Theorem~\ref{KeyTheorem}, put in the context of our key task~\eqref{KeyTask},
is that
\begin{corollary}\label{SearchRegionCor}
  If $(\lambda,\psi)$ is an eigenpair of $\cL$ with $\lambda\in[a,b]$
  and $\delta(\psi,R)\leq \delta^*$, then there is an eigenpair
  $(\mu,\phi)$ of $\cL_s$ in the region
  $U=U(a,b,s,\delta^*)=\{z\in\CC:\,\dist(z,L)\leq s\delta^*\,,\,\Im
  z<s\}$ pictured in Figure~\ref{SearchRegionFig}, where $L=[a,b]+\ii\,s$.
\end{corollary}
\begin{figure}
  \centering
  \includegraphics[width=3.0in]{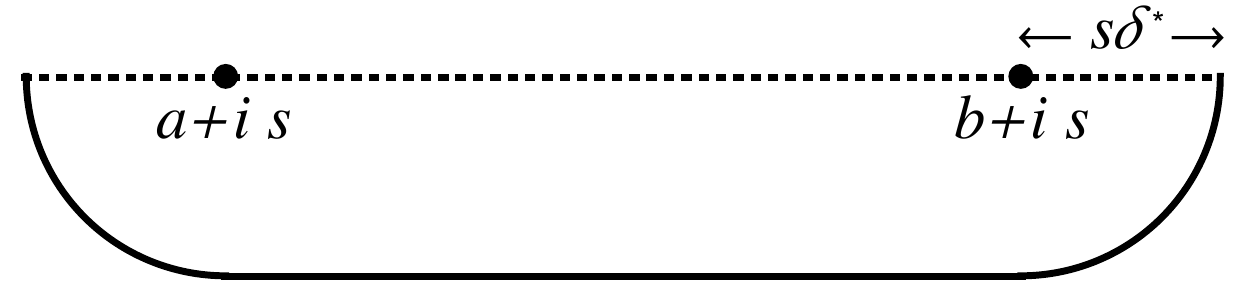}
  \caption{\label{SearchRegionFig} The region $U=U(a,b,s,\delta^*)$
    from Corollary~\ref{SearchRegionCor}.}
\end{figure}

We also have the complementary result to Theorem~\ref{KeyTheorem}, by
similar reasoning.
\begin{theorem}\label{KeyTheorem2}
Let $(\mu,\phi)$ be an eigenpair of $\cL_s$.  Then 
\begin{align}\label{EigenvalueCloseness2}
s\, [\delta(\phi,R)]^2\leq \dist(\mu,\Spec(\cL+\ii\,s))\leq s\,
  \delta(\phi,R)~,\\
  \label{EigenvalueCloseness3}
  \dist(\Re\mu,\Spec(\cL))\leq s\, \delta(\phi,R)\, \tau(\phi,R)~.
\end{align}
Let $\lambda=\argmin\{|\sigma-\Re\mu|:\,\sigma\in\Spec(\cL)\}$.
If $\Lambda\subset\Spec(\cL)$ contains $\lambda$, then
\begin{align}\label{EigenvectorCloseness2}
\inf_{v\in
  E(\Lambda,\cL)}\frac{\|\phi-v\|_{L^2(\Omega)}}{\|\phi\|_{L^2(\Omega)}}\leq
  \frac{s\,\delta(\phi,R)\,\tau(\phi,R)}{\dist(\Re\mu,\Spec(\cL)\setminus(\Lambda\cup\{\Re\mu\}))}~,
\end{align}
where $E(\Lambda,\cL)$ is the corresponding invariant subspace.
\end{theorem}
\begin{proof}
Since $\mu\notin \Spec(\cL+\ii\,s)$, and $(\cL+\ii\,s-\mu)\phi=\ii\,s\,\chi_{\Omega\setminus R}\phi$,
the upper bound in~\eqref{EigenvalueCloseness2} now follows in the same way as
its counterpart~\eqref{EigenvalueCloseness}.
Since 
$\mu=\mu_1+\ii\,s[\tau(\phi,R)]^2$,
$|d+\ii\,s-\mu|=|d-\mu_1+\ii\,s[\delta(\phi,R)]^2|\geq
s[\delta(\phi,R)]^2$ for any $d\in\RR$,
from which the lower bound  follows immediately.

Letting $\delta=\delta(\phi,R)$ and $\tau=\tau(\phi,R)$, one
finds that $(\cL-\mu_1)\phi=\ii\,s(\tau^2\chi_{\Omega\setminus
  R}\phi-\delta^2\chi_R\phi)$ by direct computation.
If $\mu_1\in\Spec(\cL)$, then~\eqref{EigenvalueCloseness3} holds
trivially, so we assume that $\mu_1\not\in\Spec(\cL)$. 
It follows that
\begin{align*}
\|\phi\|_{L^2(\Omega)}^2&\leq\|(\cL -
                          \mu_1)^{-1}\|^{2}\|\ii\,s(\tau^2\chi_{\Omega\setminus
                          R}\phi-\delta^2\chi_R\phi)\|_{L^2(\Omega)}^2\\
  &=\|(\cL - \mu_1)^{-1}\|^{2}s^2(\tau^4\|\phi\|_{L^2(\Omega\setminus R)}^2
    +\delta^4\|\phi\|_{L^2(R)}^2)~.
\end{align*}
Shifting around terms, and recalling that $\delta^2 + \tau^2 = 1$, we obtain gives
\begin{align}
\|(\cL - \mu_1)^{-1}\|^{-2} \leq s^2(\delta^4 \tau^2 + \tau^4 \delta^2) = s^2\delta^2\tau^2~,
\end{align}
which yields~\eqref{EigenvalueCloseness3}.

Now let $P=\frac{1}{2\pi\,\ii}\int_\gamma(z-\cL)^{-1}\,dz$ denote the
(orthogonal) spectral projector for $E(\Lambda,\cL)$, where $\gamma$
is a simple closed contour that encloses $\Lambda\cup\{\mu_1\}$, and
excludes $\Spec(\cL)\setminus\Lambda$.  Noting that $P$ commutes with $\cL$,
the identity
\begin{align}\label{ComplementaryProjectorIdentity2}
I-P=(\cL(I-P)-\mu_1)^{-1}(I-P)(\cL-\mu_1)~,
\end{align}
follows by direct algebraic manipulation.  From
this,~\eqref{EigenvectorCloseness2} is achieved as before.  
\end{proof}

\begin{remark}[Vector Normalization]\label{EigenvectorRotation}
  Given an eigenpair $(\mu,\phi)$ of $\cL_s$, with
  $\|\phi\|_{L^2(\Omega)}=1$, we will further normalize $\phi$ as
  follows:
  \begin{align}\label{VectorNormalization}
    \phi\longleftarrow c\phi \mbox{ where } c = \argmin\{\|\Im(d\phi)\|_{L^2(\Omega)}:\,|d|=1\}~.
  \end{align}
  Our rationale for minimizing the imaginary part in this sense is
  that if some scaling of $\phi$, $c\phi$, is close to a real(!)
  eigenvector $\psi$ of $\cL$, which is the case if the upper-bound
  in~\eqref{EigenvectorCloseness2} is small, then the imaginary part of $c\phi$
  should be small.  Proposition~\ref{ResidualBound} will provide
  further motivation for this kind of normalization.  For a given non-zero
  function $\phi=\phi_1+\ii\phi_2$, not necessarily an eigenvector,
  one can recast the minimization problem
  $\alpha=\min\{\|\Im(d\phi)\|_{L^2(\Omega)}:\,|d|=1\}$ as a
  $2\times 2$ eigenvalue problem, with $\alpha^2$ as the smaller of
  the two (real) eigenvalues.  The corresponding real eigenvector
  $\mb{c}=(c_1,c_2)$, with $c_1^2+c_2^2=1$, is related to the optimal
  scalar $c$ by $c=c_1+\ii\,c_2$.  The matrix for this eigenvalue
  problem is
\begin{align}\label{NormalizationMatrix}
  \begin{pmatrix}\|\phi_2\|_{L^2(\Omega)}^2&\int_\Omega \phi_1\phi_2\,dx\\\int_\Omega \phi_1\phi_2\,dx&\|\phi_1\|_{L^2(\Omega)}^2\end{pmatrix}~.
\end{align}
It can be seen, using the Cauchy-Schwarz inequality, that this matrix
is positive semidefinite, and that $\alpha=0$ iff
$\{\phi_1,\phi_2\}$ is a linearly dependent set.
\end{remark}

\begin{proposition}\label{ResidualBound}
  Let $(\mu,\phi)$ be an eigenpair of $\cL_s$ with $\|\phi\|_{L^2(\Omega)}=1$, and set $\mu_1=\Re\mu$,
  $\phi_1=\Re\phi$, $\phi_2=\Im\phi$, $\tau=\tau(\phi,R)$ and
  $\delta=\delta(\phi,R)$.  It holds that
  \begin{align}\label{ResidualBound2}
    \|(\cL-\mu_1)\phi_1\|_{L^2(\Omega)}^2=s^2(\tau^4
    \|\phi_2\|^2_{L^2(\Omega\setminus R)}+\delta^4 \|\phi_2\|^2_{L^2(R)})~.
  \end{align}
\end{proposition}
\begin{proof}
  As was seen in the proof of Theorem~\ref{KeyTheorem2},
  $(\cL-\mu_1)\phi=\ii\,s(\tau^2\chi_{\Omega\setminus
  R}\phi-\delta^2\chi_R\phi)$.  Comparing the real and imaginary parts
of both sides, we determine that
\begin{align*}
  (\cL-\mu_1)\phi_1=-s(\tau^2\chi_{\Omega\setminus R}\phi_2-\delta^2\chi_{R}\phi_2)~.
\end{align*}
This identity immediately yields~\eqref{ResidualBound2}.  
\end{proof}

The following 1D examples illustrate several
of the ideas and results discussed so far.
\begin{example}\label{1DExample}
  For $\Omega=(0,1)$, we consider the operator
  \begin{align*}
    \cL=-\frac{d^2}{dx^2}+\sum_{k=1}^4V_k\chi_{R_k}\quad,\quad R_k=\frac{1}{4}(k-1,k)~,
  \end{align*}
  with homogeneous Dirichlet boundary conditions, for constants
  $V_k\geq 0$.  The 
  landscape function $u$ can be determined analytically in this case.
The eigenfunctions can also be determined analytically, up to the
solutions of a non-linear equation for the eigenvalues (cf.~\cite{Canosa1970}).
As a concrete illustration, we consider the case
$(V_1,V_2,V_3,V_4)=(0,80^2,0,400^2)$.
In this case, $u$ has precisely two local maxima, $u_1=0.008652$ and
$u_3=0.008819$, at $x_1=0.13155$ and $x_3=0.61972$, respectively.
The approach of~\cite{Arnold2019a} estimates the two ground state eigenvalues as
$\lambda\approx 1.25/u_3=141.74280$ and $\lambda\approx
1.25/u_1=144.46879$; the actual ground state eigenvalues in this case are
$\lambda=140.49323$ and $\lambda=143.18099$.
Using the factor
1.875 employed in~\cite{Arnold2019a}  for a (more complicated) 1D
Schr\"odinger problem, the localization interval for the first ground
state is $[0.528993, 0.710441]\subset R_3$, and for the second ground
state, it is $[0.0416835, 0.221412]\subset R_1$.

In Figure~\ref{1DExampleFig}, plots are given of the maximum localization measures
$\max\tau_k=\max\tau(\psi,R_k)$, for the eigenvectors associated with the
smallest sixteen eigenvalues, and for the eigenvectors associated with
first fifteen eigenvalues larger than $V_4$ and the one immediately
preceding them.
The plot concerning eigenvalues near or larger than $V_4$
was specifically chosen to
demonstrate localization in the region $R_4$, which would not have been
predicted in the approaches of~\cite{Arnold2019a} or~\cite{Altmann2019,Altmann2020}. 
Plots of three eigenvectors are also given, together with their
eigenvalues and the largest of their localization measures, to
illustrate what ``highly localized'' may or may not look like in
practice.  With respect to the standard ordering of eigenvalues
$0<\lambda_1< \lambda_2<\lambda_3<\cdots $ (all eigenvalues of $\cL$ are
simple), the eigenpairs pictured in Figure~\ref{1DExampleFig}(B)
correspond to $\lambda_{11}$, $\lambda_{13}$ and $\lambda_{96}$.
It is clear that $\tau_k\leq 0.84$ does not correspond to a
natural understanding of being highly localized in $R_k$, but that
$\tau_k\geq 0.96$ does, in these cases.
The first twelve eigenvectors for this example, those for which
$\lambda<V_2=80^2$, are all strongly localized in either $R_1$ or
$R_3$, alternating between these subdomains, with
$\max\{\tau_1,\tau_3\}>0.96$; and the remaining four eigenvectors,
though each most localized in $R_2$, are not highly localized in any of 
the four subdomains.
Among the sixteen eigenvectors much higher in the spectrum, four of
them are strongly localized in $R_4$, with $\tau_4>0.96$, the second,
third, fourth and seventh; none of
the rest are highly localized in any of the four subdomains.

A slight modification of the approach in~\cite{Canosa1970} allows for
the computation of complex eigenpairs for $\cL_s$, and we use it below.
In Table~\ref{1DExampleTab}, we see the five eigenvalues $\lambda$ of
$\cL$ for which $\lambda\in[0,220000]$ and $\delta(\psi,R_3)\leq
\delta^*=1/5$ for the corresponding eigenvector.  The twelfth
eigenvalue of $\cL$,
$\lambda=4954.5303$, just fails to make the cut, with
$\delta(\psi,R_3)=0.27074936$.  
For convenience in comparison, this eigenvalue and its localization measure are
included in the table \textit{in italics}.
Also given in this table are all
eigenvalues $\mu$ of $\cL_s=\cL+\ii\,s\chi_{R_3}$ within the region
$U(0,220000,s,\delta^*)$, for $s=1$ and $s=100$, together with the
localization measures $\delta(\phi,R_3)$ for their corresponding
eigenvectors.  In both cases for $\cL_s$, six eigenpairs make the cut,
with the final eigenvalue approximating $\lambda=4954.5303$.
Recalling~\eqref{RayleighQuotient}, we note that $\delta(\psi,R_3)$
can be obtained directly from $\Im\mu$,
$\delta(\psi,R_3)=\sqrt{1-\Im\mu/s}$.
When $s=1$, $\lambda$ and $\Re\mu$, and $\delta(\psi,R_3)$ and
$\delta(\phi,R_3)$, agree in all digits shown, for eigenpairs
$(\lambda,\psi)$ and $(\mu,\phi)$ that are matched.
We note that there are $130$
eigenvalues of $\cL$ in $[0,220000]$, so the reduction to six
candidates for localization in $R$ is significant.

In Table~\ref{1DExampleTabB}, we see the two eigenvalues $\lambda$ of
$\cL$ for which $\lambda\in [0,220000]$ and $\delta(\psi,R_4)\leq
\delta^*=1/5$ for the corresponding eigenvector.  Also given in this table are all
eigenvalues $\mu$ of $\cL_s=\cL+\ii\,s\chi_{R_4}$ within the region
$U(0,220000,s,\delta^*)$, for $s=1$ and $s=100$, together with the
localization measures $\delta(\phi,R_4)$ for their corresponding
eigenvectors.  In both cases for $\cL_s$, six eigenpairs make the cut,
and the eigenvalues of $\cL$ that best match the remaining four of
$\cL_s$, together with their localization measures, are also given
\textit{in italics} in this table.

Finally, in Figure~\ref{1DExampleFigB}, analogues of the plots in
Figure~\ref{1DExampleFig}(B) are given for $\cL_s$, with $s=100$, and
$R=R_3$ for the first pair of plots, $R=R_2$ for the second pair,
and $R=R_4$ for the third
pair.  Each pair of plots shows the real and imaginary parts of an
eigenfunction $\phi$, normalized so that $\|\phi\|_{L^2(\Omega)}=1$
and $\alpha=\|\Im\phi\|_{L^2(\Omega)}$ is minimized---see
Remark~\ref{EigenvectorRotation}.  Up to scaling, the
corresponding $\psi$ and $\phi_1=\Re\phi$ show strong resemblances.
Localization values $\delta(\phi,R)$ and relative residuals
$\|(\cL-\mu_1)\phi_1\|_{L^2(\Omega)}/\|\phi_1\|_{L^2(\Omega)}$ (see
Proposition~\ref{ResidualBound}) are also included.
The second eigenvector, shown in Figure~\ref{1DExampleFigB}(B), has
the largest (worst) $\delta$-value, $\alpha$-value and relative
residual among the three, whereas the third eigenvector is best in
each of these categories.
The $\alpha$-values in
Table~\ref{1DExampleTab} for $s=100$ increased from
$7.189\times10^{-4}$ to $5.337\times 10^{-3}$ for the first five, and
$\alpha=1.1015\times10^{-2}$ for the sixth.  For the eigenvectors in
Table~\ref{1DExampleTabB} with $s=100$, $\alpha=4.152\times10^{-3}$
and $\alpha=8.870\times 10^{-3}$ for the first two, and the
$\alpha$-values ranged between $1.515\times10^{-2}$ and
$2.613\times10^{-2}$ for the final four.

Based on the apparent correlation
between a normalized eigenvector $\phi$ of $\cL_s$ having a small
$\alpha$-value and it being close to an eigenvector $\psi$ of $\cL$
that is localized in $R$, one might be tempted think that
$\alpha$-values for eigenvectors of $\cL_s$ are, \textit{by
  themselves}, decent indicators of localization of eigenvectors of
$\cL$ in $R$.  However, this is not the case.  For example, the second
eigenvalue of $\cL_s$, with $s=100$ and $R=R_3$, is
$\mu=143.18098+6.2629546\times 10^{-17}\,\ii$, with the real part of
the (normalized) eigenfunction $\phi$
highly localized in $R_1$, and its significantly smaller imaginary
part highly localized in $R_3$: $\delta(\Re\phi,R_1)=0.032815726$,
$\delta(\Im\phi,R_3)=0.042859575$ and $\alpha=6.9038783\times
10^{-10}$.  The eigenvector $\phi$ of $\cL$ that is closest to $\phi$
has $\lambda=143.18098$, $\delta(\psi,R_1)=0.032815726$ and
$\delta(\psi,R_3)=1.0000000$.  So a small $\alpha$-value for $\phi$ in
this case corresponds to a nearby eigenvector of $\cL$ that is highly
localized in the complement of $R_3$!  Of course, we never would have
considered this eigenpair of $\cL_s$ if we were searching in the
region $U$ indicated in Corollary~\ref{SearchRegionCor}.

{\sc Mathematica}
was used for all of these computations and plots, employing very high precision
  arithmetic.
\begin{figure}
	\centering
	\subfloat[$\max\tau_k$ for $\lambda\in(140,7155)$ (left) and
        $\lambda\in(158923,176357)$ (right)]
	{
          \includegraphics[width=0.45\textwidth]{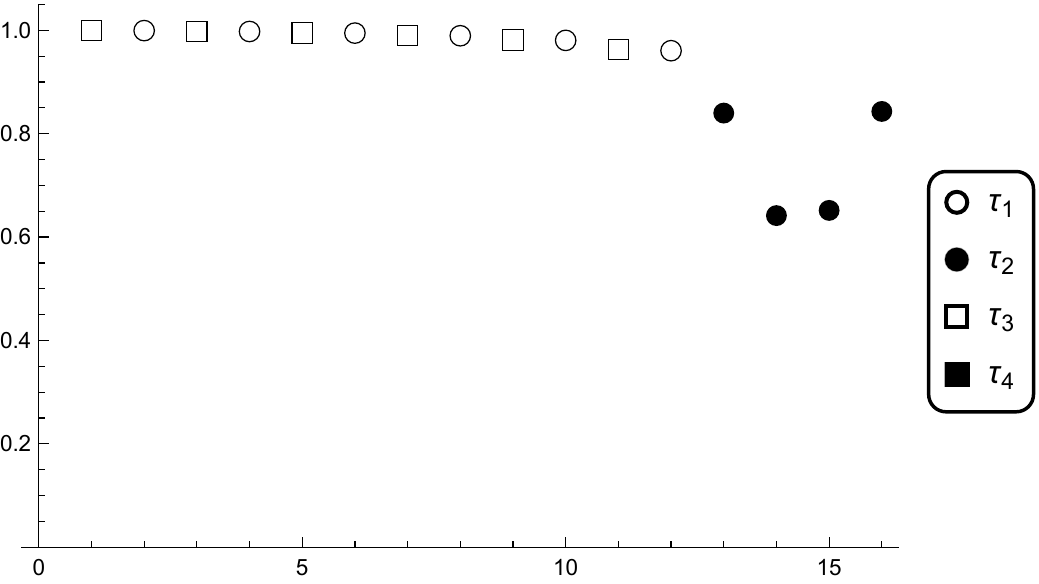}\quad
          \includegraphics[width=0.45\textwidth]{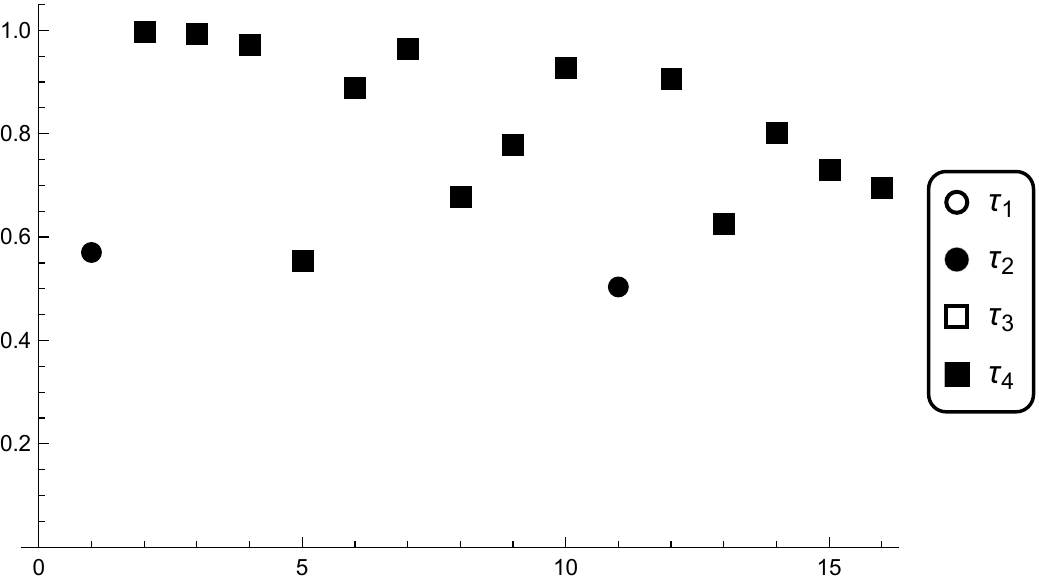}
	}
\quad
	\subfloat[Three eigenvectors, with $\lambda$ and $\max \tau_k$.]
	{
          \includegraphics[width=0.3\textwidth]{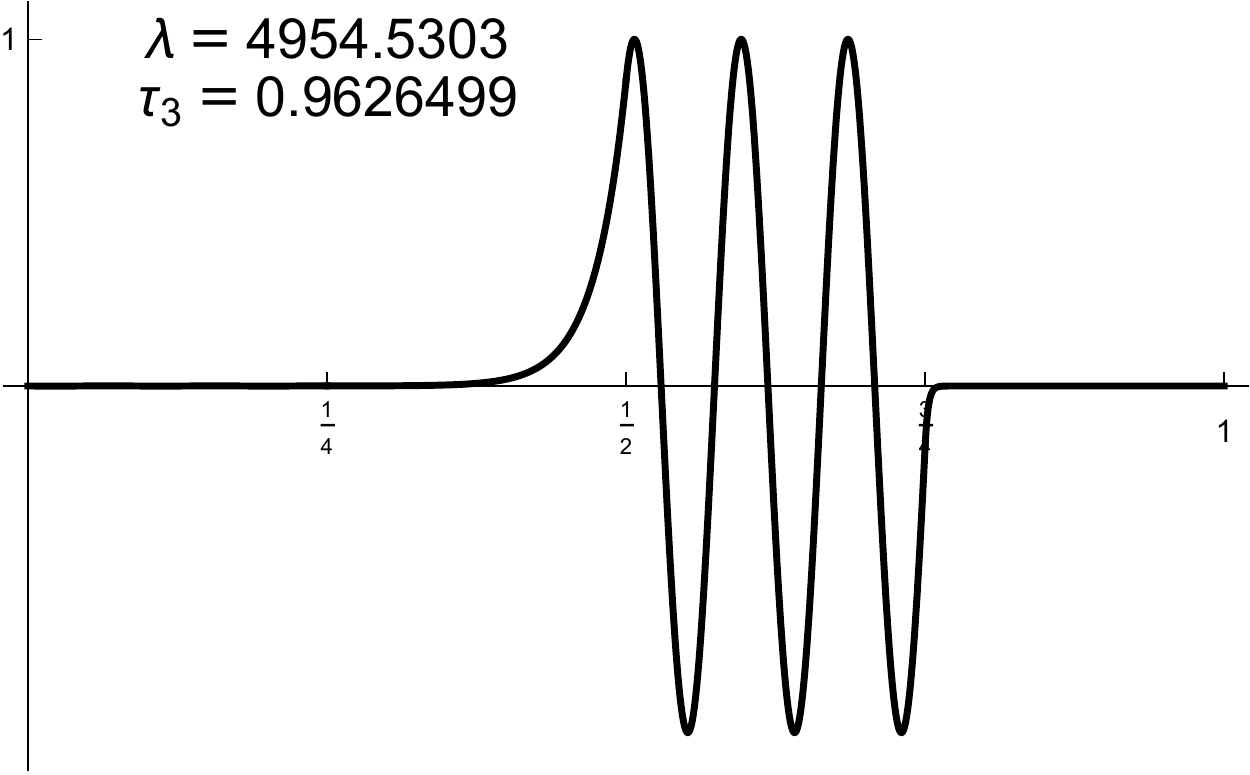}\quad
          \includegraphics[width=0.3\textwidth]{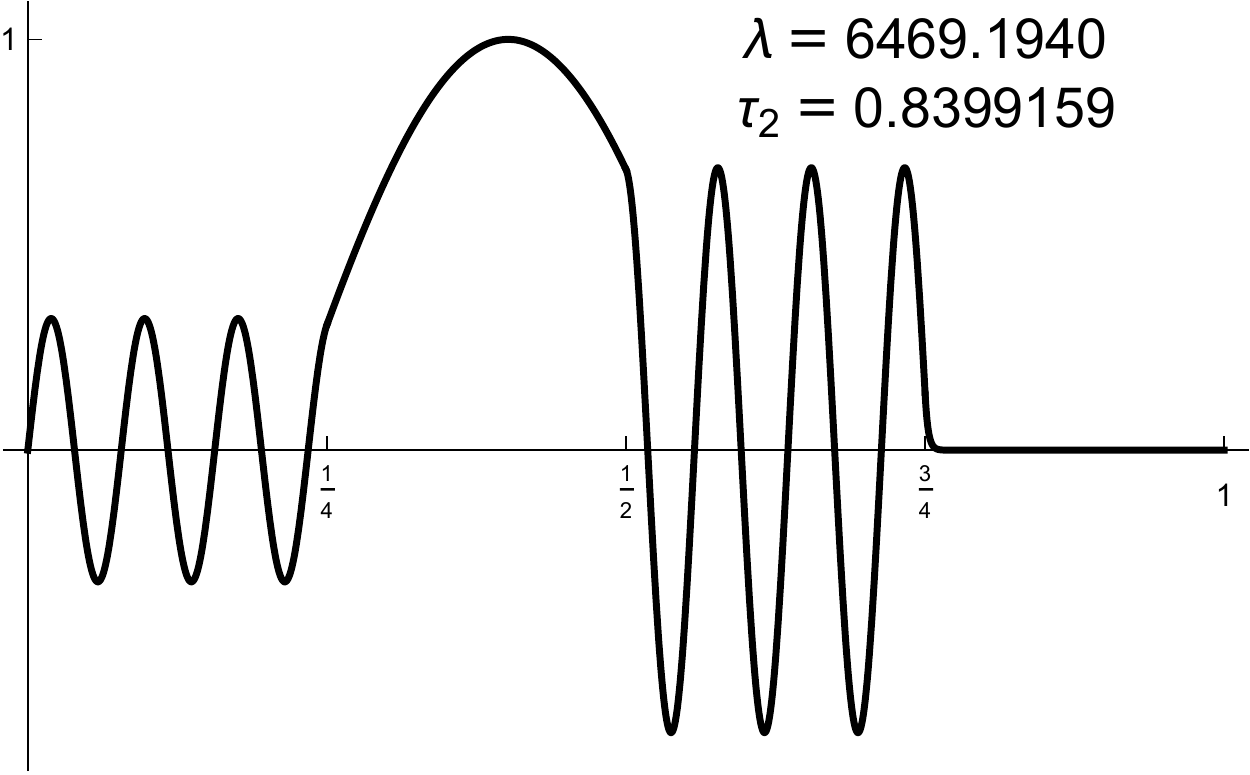}\quad
          \includegraphics[width=0.3\textwidth]{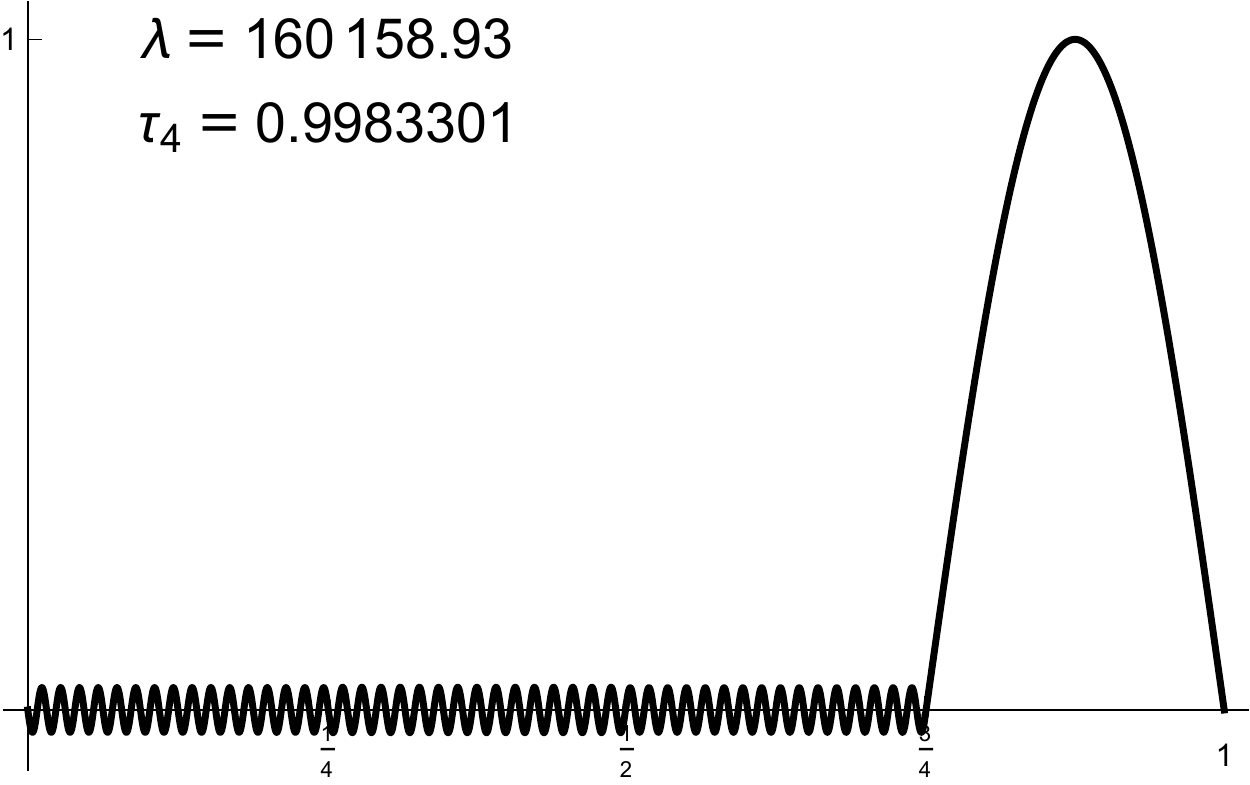}
	}
\caption{\label{1DExampleFig} Localization measures
  $\max\tau_k=\max\tau(\psi,R_k)$ for 32 eigenvectors, and plots of three of
  these eigenvectors, for Example~\ref{1DExample}.}
\end{figure}

\begin{table}
  \centering
  \caption{\label{1DExampleTab} Eigenvalues and localization measures
    in $R_3$
    for eigenpairs $(\lambda,\psi)$ of $\cL$ with
    $\lambda\in[0,220000]$ and $\delta(\psi,R_3)\leq \delta^*=1/5$,
    for Example~\ref{1DExample}.  Eigenvalues and localization
    measures for the eigenpairs $(\mu,\phi)$ of $\cL_s$ with $\mu\in
    U(0,220000,s,1/5)$, for $s=1$ and $s=100$.  In italics are
    the eigenvalue and localization measure for $\cL$ that fail to
    satisfy $\delta(\psi,R_3)\leq \delta^*$, but which have obvious
    counterparts among the  given eigenpairs of $\cL_s$.}
\begin{tabular}{|cc|ccc|ccc|}\hline
  \multicolumn{2}{|c|}{}&\multicolumn{3}{c|}{$s=1$}&\multicolumn{3}{c|}{$s=100$}\\
  $\lambda$&$\delta(\psi,R_3)$&$\Re\mu$&$\Im\mu$&$\delta(\phi,R_3)$&$\Re\mu$&$\Im\mu$&$\delta(\phi,R_3)$\\\hline
  140.49323&0.0324770&140.49323&0.99894524&0.0324770&140.49441&99.894539&0.0324748\\
  561.35749&0.0659934&561.35749&0.99564487&0.0659934&561.36244&99.564551&0.0659886\\
  1260.5517&0.1019069&1260.5517&0.98961499&0.1019069&1260.5640&98.961665&0.1018987\\
  2233.8447&0.1425062&2233.8447&0.97969199&0.1425062&2233.8708&97.969580&0.1424928\\
  3472.5421&0.1928871&3472.5421&0.96279456&0.1928871&3472.5974&96.280425&0.1928620\\
  \textit{4954.5303}&\textit{0.2707494}&4954.5303&0.92669479&0.2707494&4954.6877&92.674005&0.2706658\\\hline
\end{tabular}
\end{table}

\begin{table}
  \centering
  \caption{\label{1DExampleTabB} Eigenvalues and localization measures
    in $R_4$
    for eigenpairs $(\lambda,\psi)$ of $\cL$ with
    $\lambda\in[0,220000]$ and $\delta(\psi,R_4)\leq \delta^*=1/5$,
    for Example~\ref{1DExample}.  Eigenvalues and localization
    measures for the eigenpairs $(\mu,\phi)$ of $\cL_s$ with $\mu\in
    U(0,220000,s,1/5)$, for $s=1$ and $s=100$.  In italics are
    the eigenvalues and localization measures for $\cL$ that fail to
    satisfy $\delta(\psi,R_4)\leq \delta^*$, but which have obvious
    counterparts among the  given eigenpairs of $\cL_s$.
  }
\begin{tabular}{|cc|ccc|ccc|}\hline
  \multicolumn{2}{|c|}{}&\multicolumn{3}{c|}{$s=1$}&\multicolumn{3}{c|}{$s=100$}\\
  $\lambda$&$\delta(\psi,R_4)$&$\Re\mu$&$\Im\mu$&$\delta(\phi,R_4)$&$\Re\mu$&$\Im\mu$&$\delta(\phi,R_4)$\\\hline
  160158.93&0.0577667&160158.93&0.99666301&0.0577667&160158.92&99.667632&0.0576513\\
  160629.92&0.1092483&160629.92&0.98806481&0.1092483&160629.94&98.809982&0.1090879\\
  \textit{161389.73}&\textit{0.2332336}&161389.73&0.94560213&0.2332335&161390.21&94.613050&0.2320981\\
  \textit{163942.68}&\textit{0.2610104}&163942.68&0.93187361&0.2610103&163942.71&93.203093&0.2607088\\
  \textit{167673.57}&\textit{0.3716197}&167673.57&0.86189881&0.3716197&167673.93&86.221341&0.3711962\\
  \textit{170192.99}&\textit{0.4219290}&170192.99&0.82197599&0.4219289&170192.51&82.233071&0.4215084\\\hline
\end{tabular}
\end{table}

\begin{figure}
	\centering
	\subfloat[$s=100$ for $R_3$,  $\delta(\phi,R_3)=0.27066575$, $\|(\cL-\mu_1)\phi_1\|_{L^2(\Omega)}/\|\phi_1\|_{L^2(\Omega)}=0.67472464$.]
	{
          \includegraphics[width=0.45\textwidth]{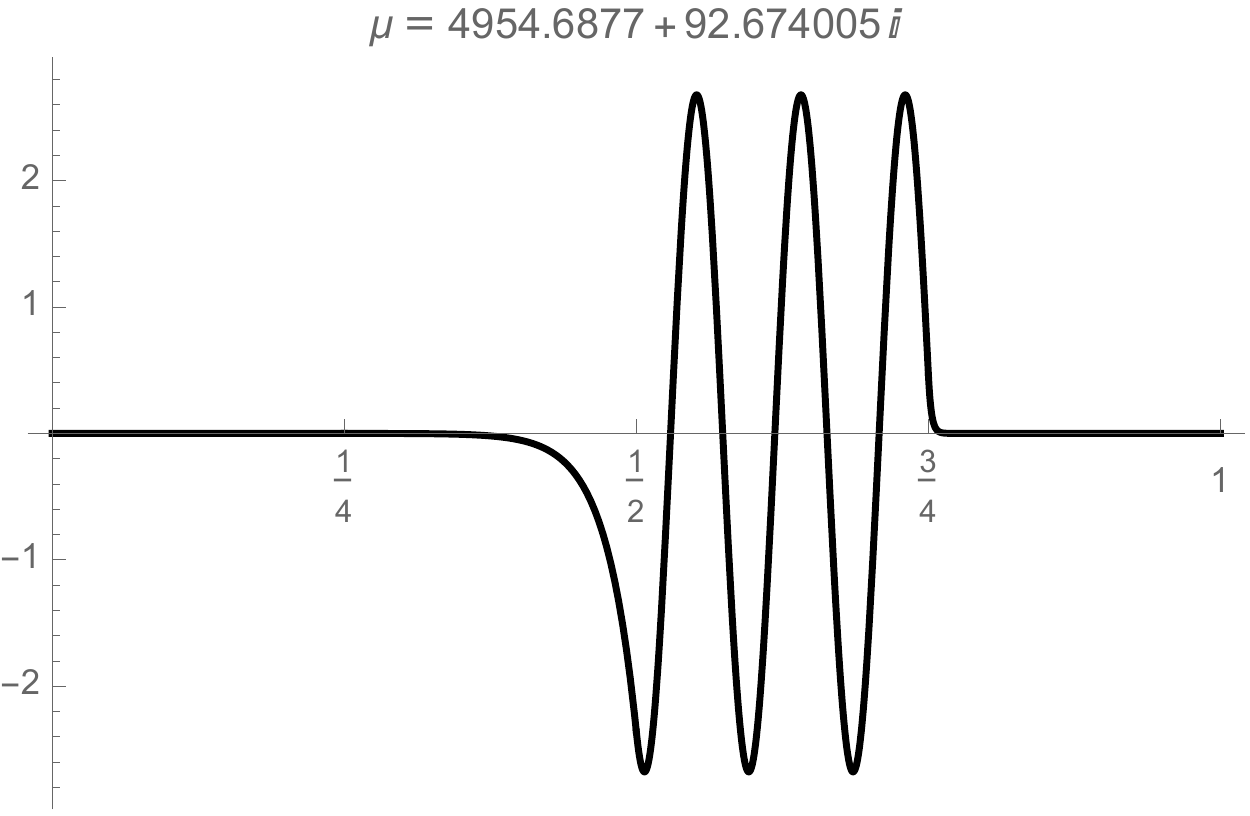}\quad
          \includegraphics[width=0.45\textwidth]{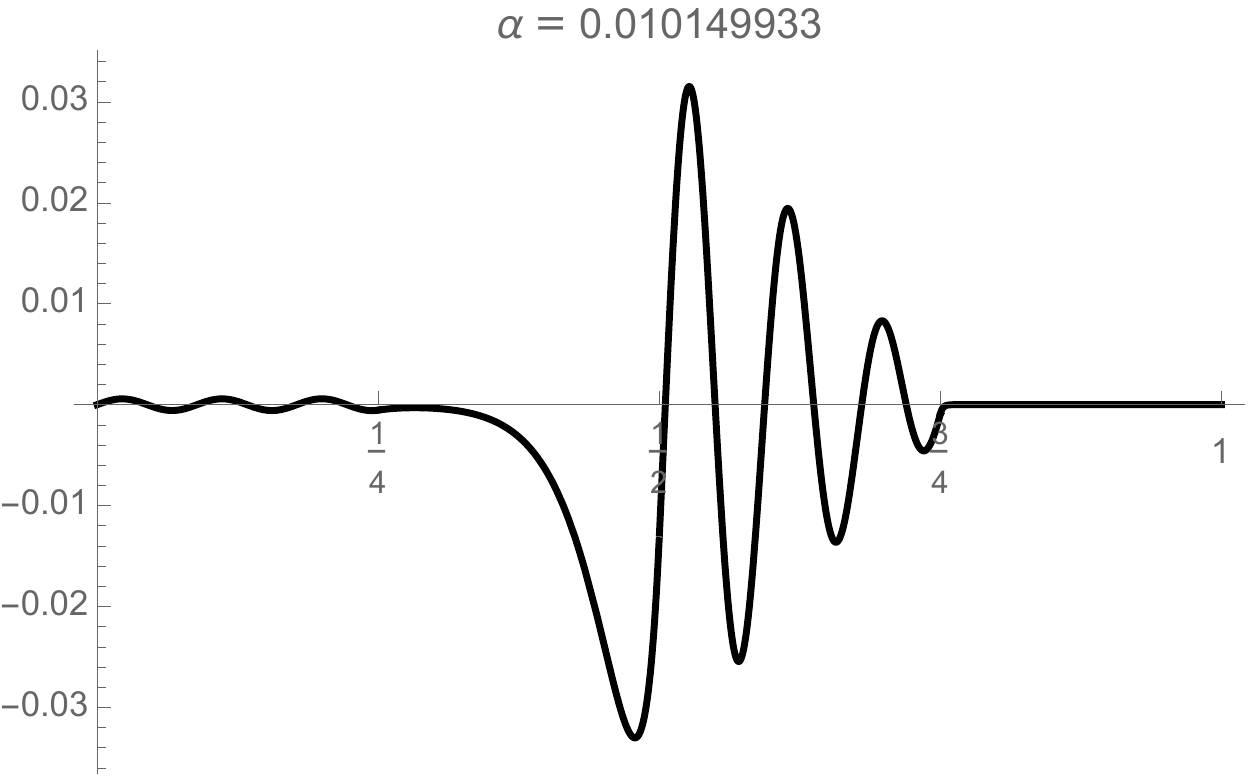}
	}
\\
	\subfloat[$s=100$ for $R_2$,  $\delta(\phi,R_2)=0.53309756$, $\|(\cL-\mu_1)\phi_1\|_{L^2(\Omega)}/\|\phi_1\|_{L^2(\Omega)}=10.9088368$.]
	{
          \includegraphics[width=0.45\textwidth]{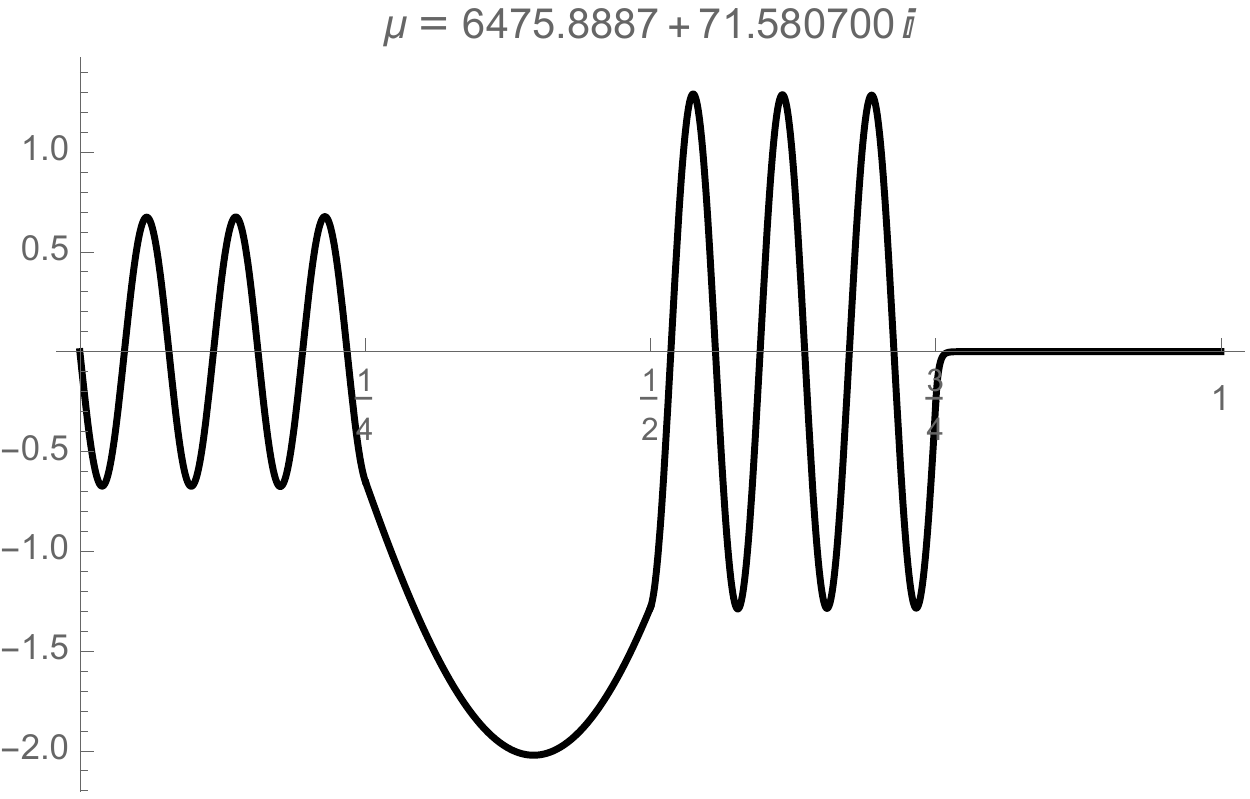}\quad
          \includegraphics[width=0.45\textwidth]{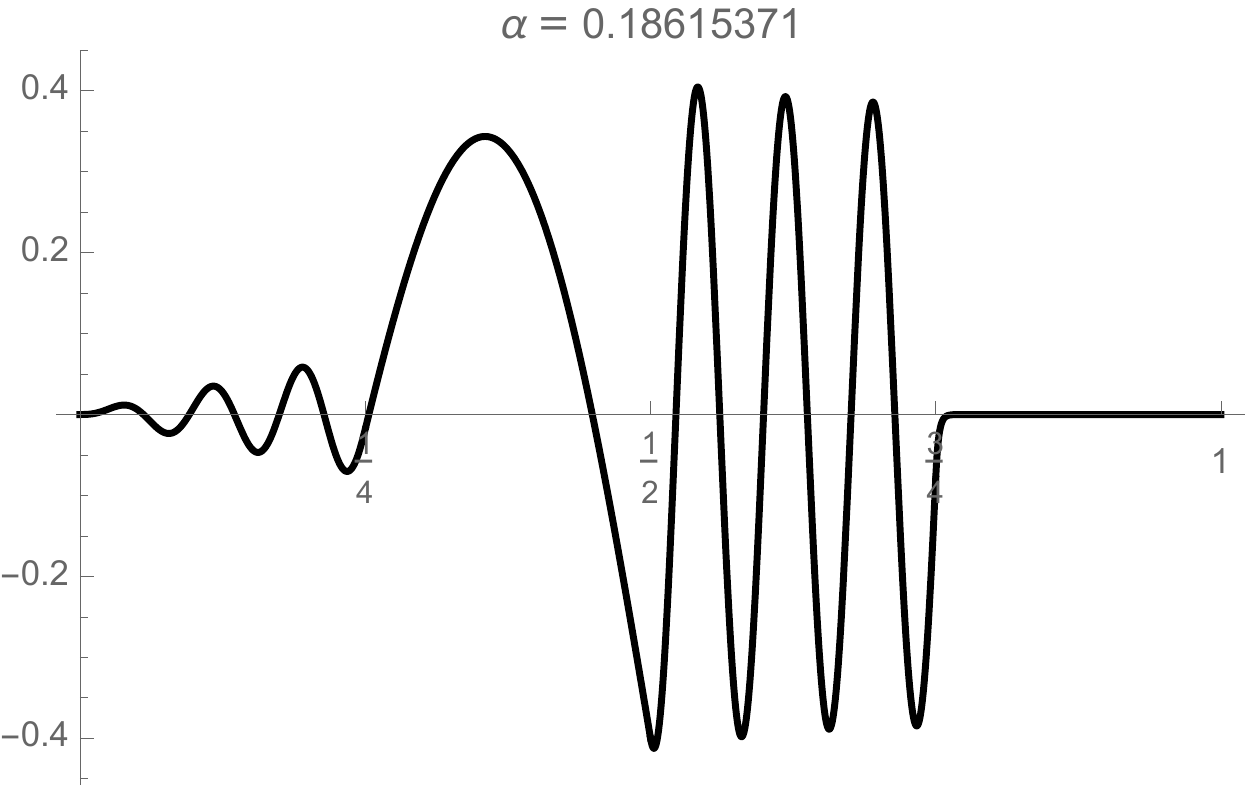}
	}
\\
	\subfloat[$s=100$ for $R_4$,  $\delta(\phi,R_4)=0.05765133$, $\|(\cL-\mu_1)\phi_1\|_{L^2(\Omega)}/\|\phi_1\|_{L^2(\Omega)}=0.36354044$.]
	{
          \includegraphics[width=0.45\textwidth]{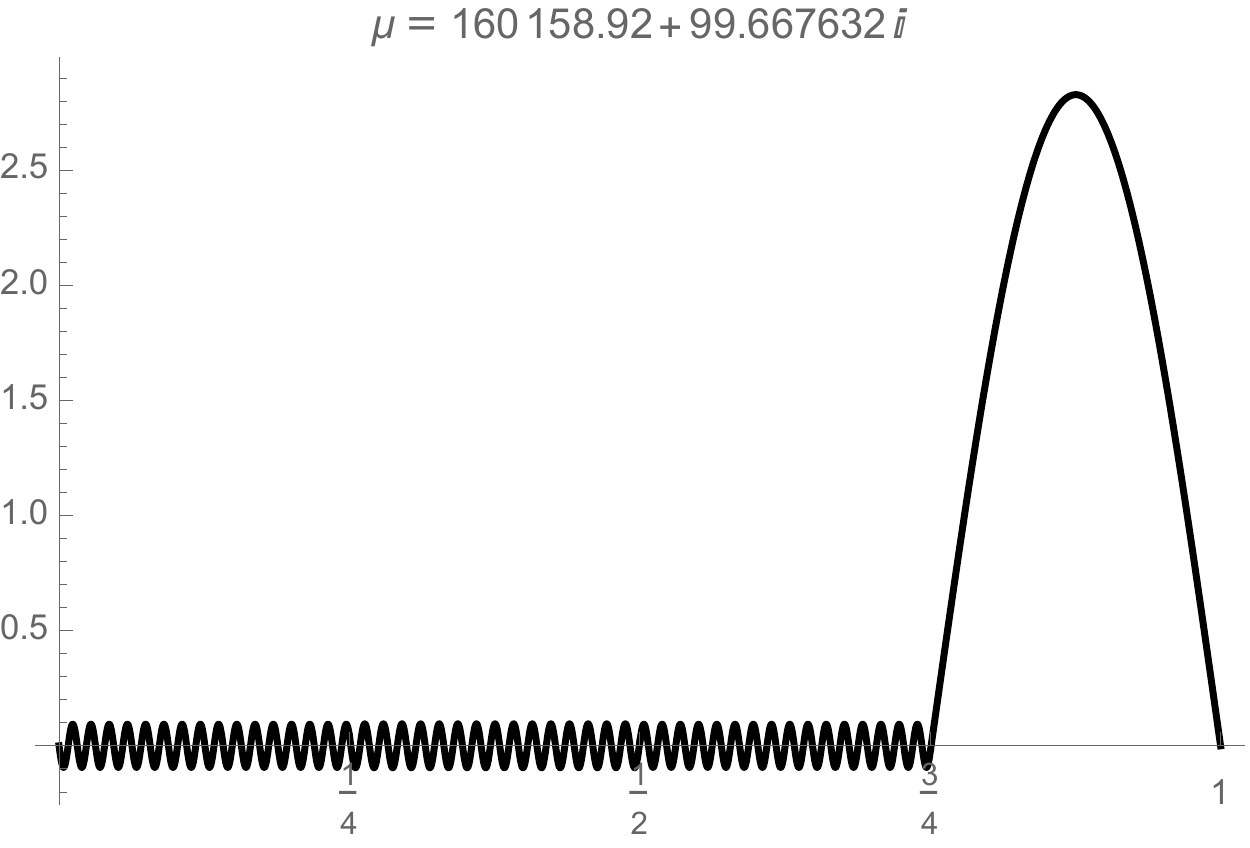}\quad
          \includegraphics[width=0.45\textwidth]{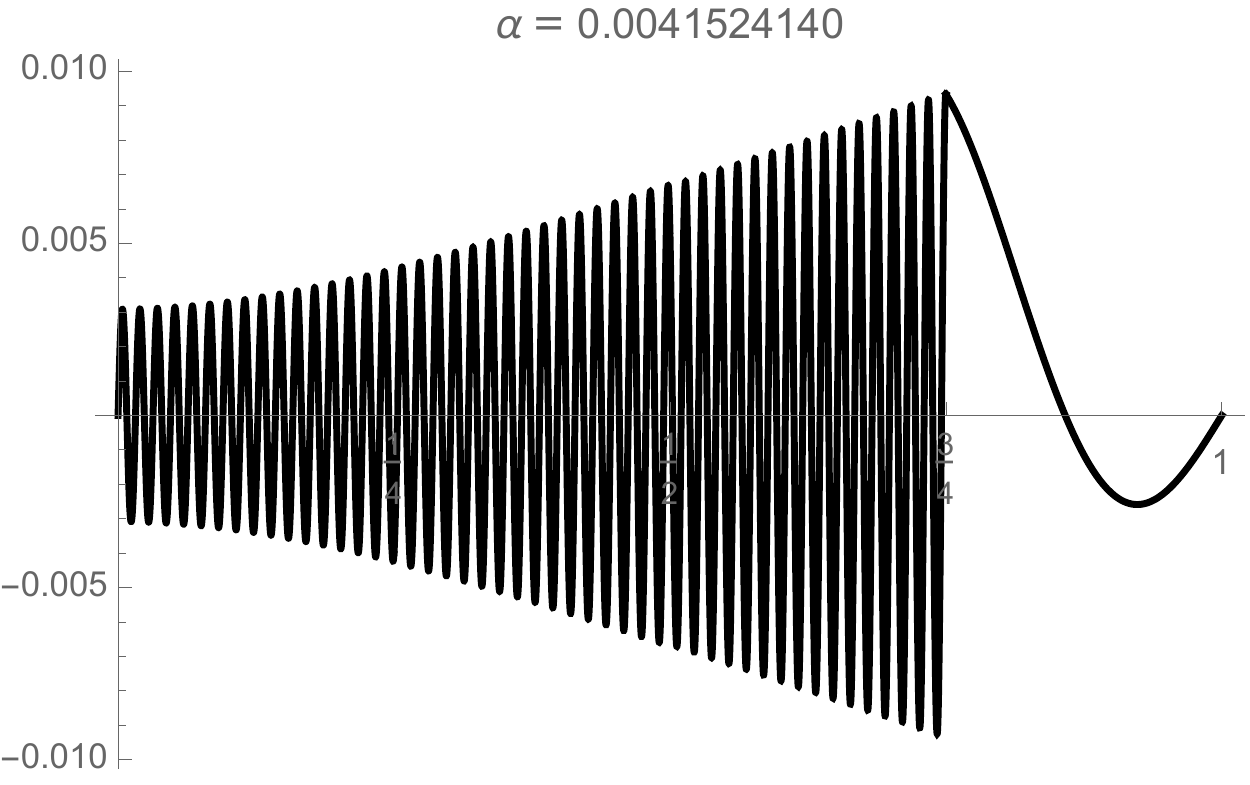}
	}        
\caption{\label{1DExampleFigB} Plots of $\phi_1=\Re \phi$ (left) and
  $\phi_2=\Im \phi$ for  eigenvectors $\phi$ of $\cL_s$, normalized so that
  $\|\phi\|_{L^2(\Omega)}=1$ and $\alpha=\|\phi_2\|_{L^2(\Omega)}$ is
  minimized, for Example~\ref{1DExample}.  The eigenvalues
  $\mu=\mu_1+\ii\,\mu_2$ are given above the plots of $\phi_1$,
  and the $\alpha$-values are given above the plots of $\phi_2$.  
  Compare with Figure~\ref{1DExampleFig}(B).}
\end{figure}
\end{example}

  \begin{example}[False Indication of Localization, Large $s$]\label{FalseLocalizationExample1}
    It can be the case that $\cL_s=\cL+\ii\,s\,\chi_R$ has localized
    eigenvectors even when $\cL$ has none.  For example, if
    $\cL=-\frac{d^2}{dx^2}$ on $(0,1)$, with homogeneous Dirichlet
    conditions, the eigenpairs are
    $(\lambda_n,\psi_n)=((n\pi)^2,\sin(n\pi x))$.  Taking $R=(0,1/4)$,
    we have
    \begin{align*}
    \frac{1}{4}\left(1-\frac{2}{\pi}\right)\leq
      [\tau(\psi_n,R)]^2=\frac{1}{4}\left(1+\frac{\ii^{n+1}+(-\ii)^{n+1}}{n\pi}\right)
      \leq \frac{1}{4}\left(1+\frac{2}{3\pi}\right)\quad,\quad
       [\tau(\psi_n,R)]^2\to \frac{1}{4}~.
    \end{align*}
    Clearly, none of the eigenvectors are localized in $R$.

    However, if we take $s=10^4$, we find an eigenpair $(\mu,\phi)$ of
    $\cL_s$ with $\mu=149.02494+9991.7736\,\ii$ and
    $\delta(\phi,R)=0.02868$.  In this case, $\phi$ is highly localized
    in $R$, and $\mu$ will lie in some $U(a,b,s,\delta^*)$ for any
    $\delta^*\geq 0.000822$, which could lead to the false indicator
    that there is an eigenpair $(\lambda,\psi)$ of $\cL$ with $\psi$
    highly localized in $R$.  The nearest eigenvalue of $\cL+\ii\,s$
    to $\mu$ is $\lambda=(4\pi)^2+\ii\,s$, with $\psi=\psi_4$, for
    which $R$ is a nodal domain. 
  \end{example}

   \begin{example}[False Indication of Localization, Small $s$]\label{FalseLocalizationExample2}
     One might wonder if only allowing smaller $s$, say $s\approx 1$,
     would eliminate such false indicators of localization of
     eigenvectors of $\cL$, but that is not necessarily the case.
     Suppose that $\cL=-\frac{d^2}{dx^2}+V\chi_{(1/4,3/4)}$.  It is
     straight-forward to show that eigenvectors of $\cL$ must have
     either even or odd symmetry about $x=1/2$, so eigenvectors cannot
     be highly localized in either $R=(0,1/4)$ or $R=(3/4,1)$.  We
     also note that, being a Sturm-Liouville problem, the eigenvalues
     are known to be simple, but choosing large $V$ can result in
     distinct eigenvalues that are very close. 
    For our
    illustrations, we choose $V=80^2$.
  The eigenvectors corresponding to the smallest two eigenvalues of
  $\cL$ are given in Figure~\ref{FalseEx2}.  These eigenvalues are
  extremely close to each other, $\lambda_2-\lambda_1\approx
  3.59\times 10^{-16}$.  Increasing $V$ reduces this gap even further.

When $R=(0,1/4)$ and $s=1$ are chosen for $\cL_s$, the eigenpair
$(\mu,\phi)$ having smallest real part $\mu_1$ is given in
Figure~\ref{FalseEx2}.  For this eigenpair, we have
$\delta(\phi,R)=0.032815726$,
$\|(\cL-\mu_1)\phi_1\|_{L^2(\Omega)}/\|\phi_1\|_{L^2(\Omega)}=7.27553\times
10^{-6}$, which is a very strong false(!) indicator of localization of
an eigenvector $\psi$ of $\cL$ having eigenvalue $\lambda$ near $\mu_1$.
There is a counterpart of $\lambda_2$ among the eigenvalues of $\cL_s$
as well, having real part $143.180985607932844609$ (slightly larger
than $\mu_1=\Re\mu$) and imaginary part $3.2278925\times10^{-32}$.
The corresponding eigenvector is highly localized in $(3/4,1)$.  Motivated
by this observation, we mention a simple analogue of~\eqref{EigenvalueCloseness}
whose proof follows precisely the
same pattern as that for~\eqref{EigenvalueCloseness} in Theorem~\ref{KeyTheorem}.  For an eigenpair
$(\lambda,\psi)$ of $\cL$, it holds that
    \begin{align}\label{AltEigenvalueEstimate}
    \dist(\lambda\,,\,\Spec(\cL_s))&\leq
                                     s\tau(\psi,R)=s\delta(\psi,\Omega\setminus
                                     R)~.
    \end{align}
An obvious analogue for~\eqref{EigenvectorCloseness} holds as well,
though we do not state it here.

Changing to $R=(3/4,1)$ yields essentially identical results as were
obtained for $R=(0,1/4)$, up to a flip in the graphs of the real and
imaginary parts of $\phi$ across $x=1/2$ and a change in signs.  These
are also given in Figure~\ref{FalseEx2}.  If one had the inspired idea
of taking $R=(0,1/4)\cup(3/4,1)$, the first two eigenpairs of $\cL_s$
(ordered by their real parts) are very close to those of $\cL$.  These
are given in Figure~\ref{FalseEx2B}.

Because of the tight clustering of eigenvalues of $\cL$, we use this
example to discuss the eigenvector
result~\eqref{EigenvectorCloseness2} in Theorem~\ref{KeyTheorem2}.
With $R=(0,1/4)$, the pair $(\mu_1,\phi_1)=(\Re\mu,\Re\phi)$
associated corresponding to Figure~\ref{FalseEx2} (B) has $\mu_1$ very
close to two eigenvalues of $\cL$, those we called $\lambda_1$ and
$\lambda_2$ above, but $\phi_1$ is not close any
eigenvector of $\cL$.  If we were to take $\Lambda=\{\lambda_1\}$ or
$\Lambda=\{\lambda_2\}$ in~\eqref{EigenvectorCloseness2}, then the
denominator in the bound,
$\dist(\mu_1,\Spec(\cL)\setminus\Lambda)$,
is extremely small, which permits the poor approximation of $\phi_1$
in $E(\Lambda,\cL)$.  Note that $\mu_1\notin\Spec(\cL)$, so
$\Spec(\cL)\setminus(\Lambda\cup\{\mu_1\})=\Spec(\cL)\setminus
\Lambda$, and so we use the more concise expression $\dist(\mu_1,\Spec(\cL)\setminus\Lambda)$.
However, if we take
$\Lambda=\{\lambda_1,\lambda_2\}$, then
$\dist(\mu_1,\Spec(\cL)\setminus\Lambda)$ is much
larger, and a good approximation of $\phi_1$
in $E(\Lambda,\cL)$ is ensured.  More specifically, the next closest
eigenvalues of $\cL$ to $\mu_1$ are $\lambda_3=572.082899256658465449$ 
and $\lambda_4=572.08289925665847099$, so
$\dist(\mu_1,\Spec(\cL)\setminus\Lambda)=\lambda_3-\mu_1=428.90191$.
In this case, it is easy to see how
$\phi_1$ is very close to a linear combination of the eigenvectors of $\cL$
pictured in Figure~\ref{FalseEx2}(A); calling these $\psi_1$ and
$\psi_2$, we see that $\phi_1\approx -(\psi_1+\psi_2)/\sqrt{2}$.

As a matter of interest, we mention that the landscape function $u$ for $\cL$ has a
single local maximum value $u_{\max}=0.0086523880$, which is achieved at
$x=x_{\max}=0.13154762$ and $x=1-x_{\max}$.  Using the approach
of~\cite{Arnold2019a}, the corresponding estimate of the smallest
eigenvalue is $\tilde\lambda =1.25/u_{\max}=144.46879$.  The summary
description for finding a localization region $R$ that was given in
Section~\ref{Intro} is inadequate in this case, because both connected
components of $\{x\in\Omega:\,W(x)\geq E\}$ (for
$E\geq \tilde\lambda$) contain a minimizer of $W=1/u$, and both
components are of the same size, so such a consideration cannot be
used as a tie-breaker.  The authors of~\cite{Arnold2019a} note that
tightly clustered minima and/or minimizers of $W$ can make the choice
of corresponding localization regions much more challenging, and
indicate that they may pursue a more nuanced approach in the future.

  \begin{figure}
	\centering
	\subfloat[The first two eigenvectors of $\cL$.]
	{
          \includegraphics[width=0.45\textwidth]{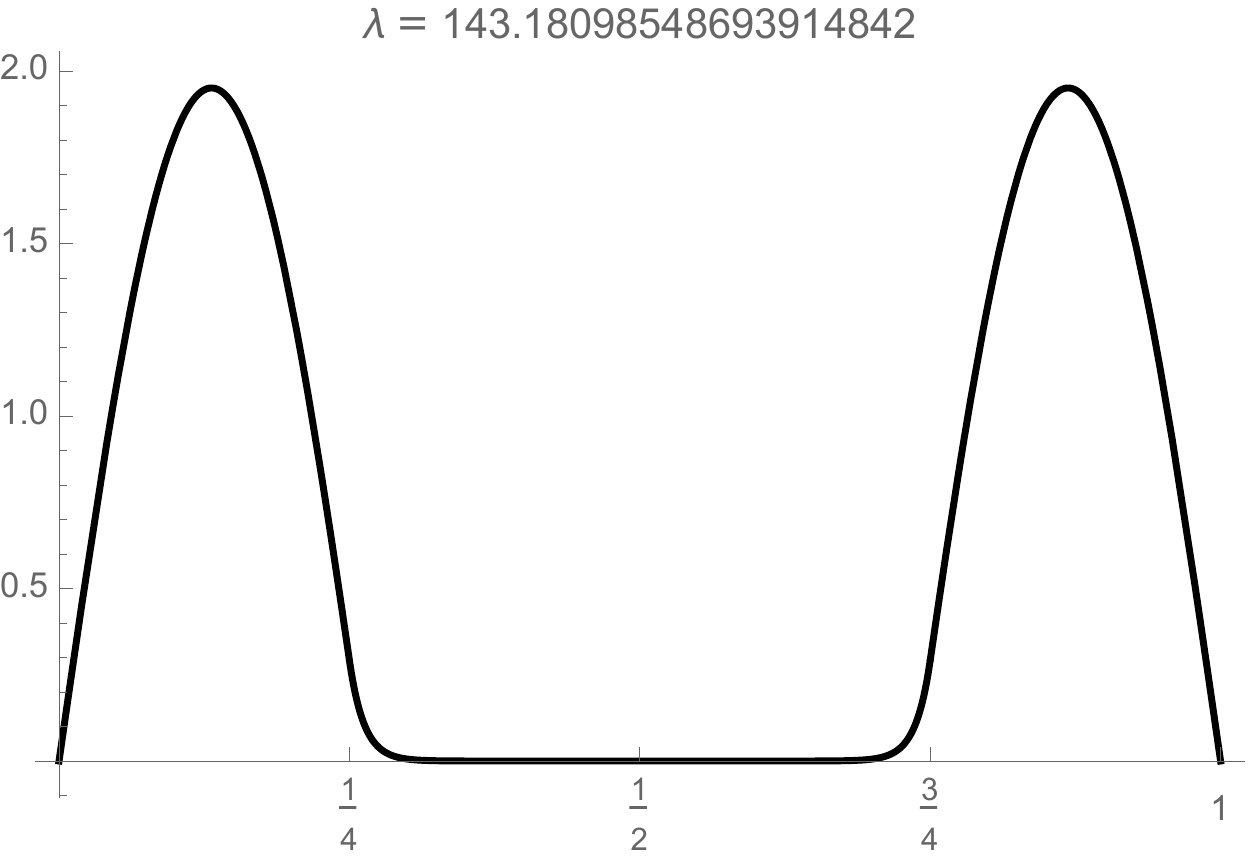}\quad
          \includegraphics[width=0.45\textwidth]{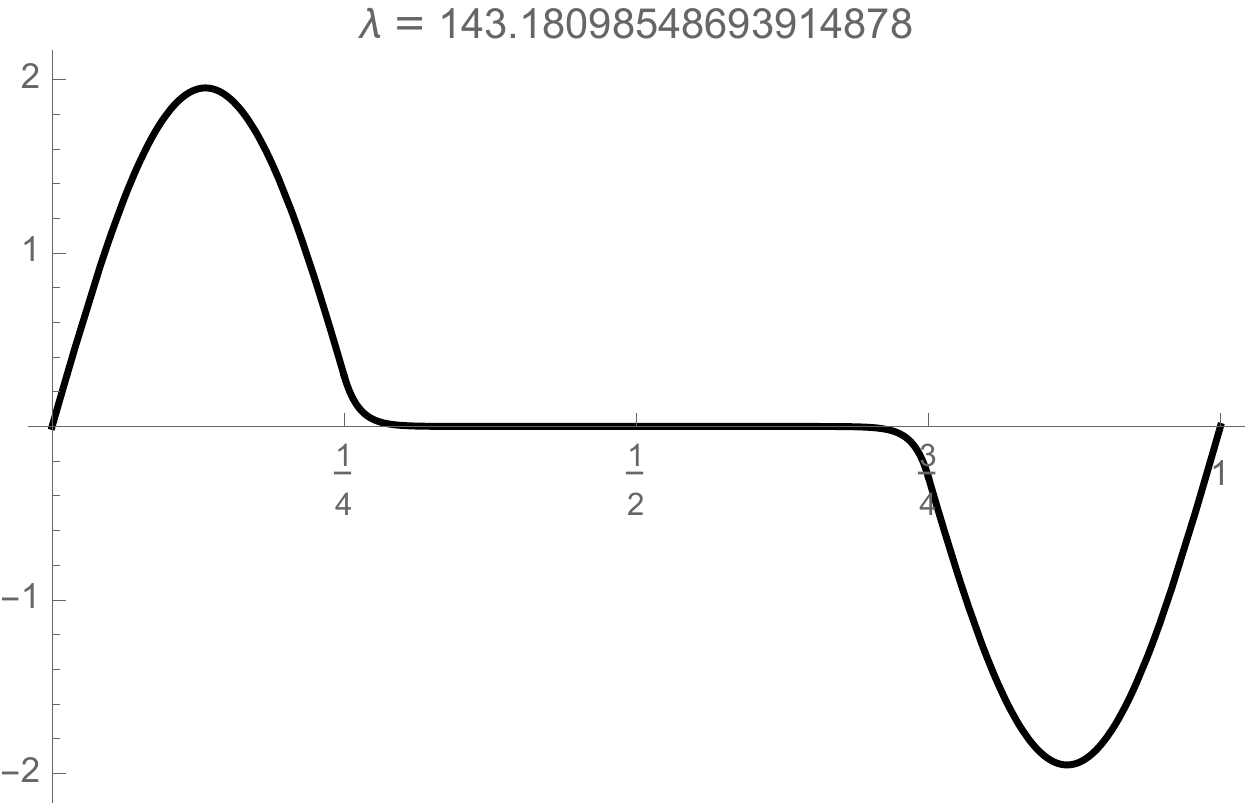}
	}
\\
	\subfloat[The real (left) and imaginary of the first eigenpair
        of $\cL_s$ with $R=(0,1/4)$.  ]
	{
          \includegraphics[width=0.45\textwidth]{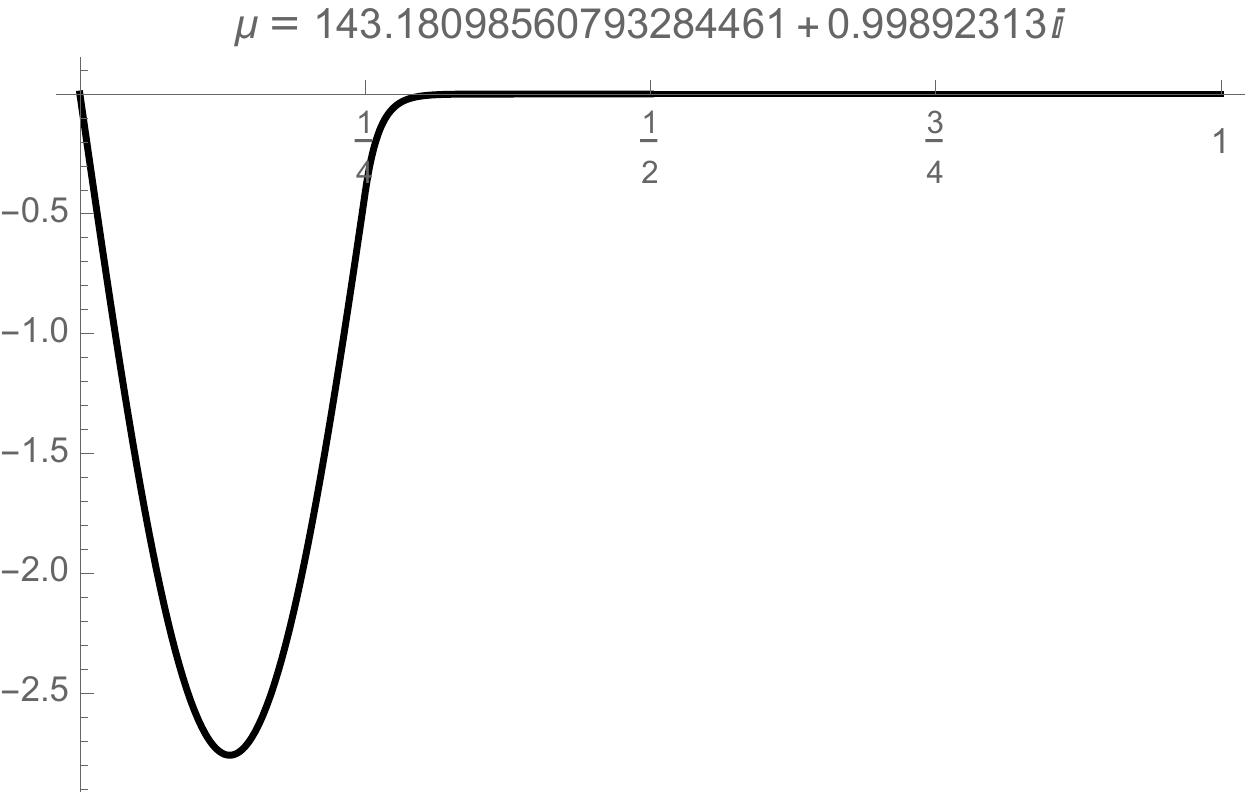}\quad
          \includegraphics[width=0.45\textwidth]{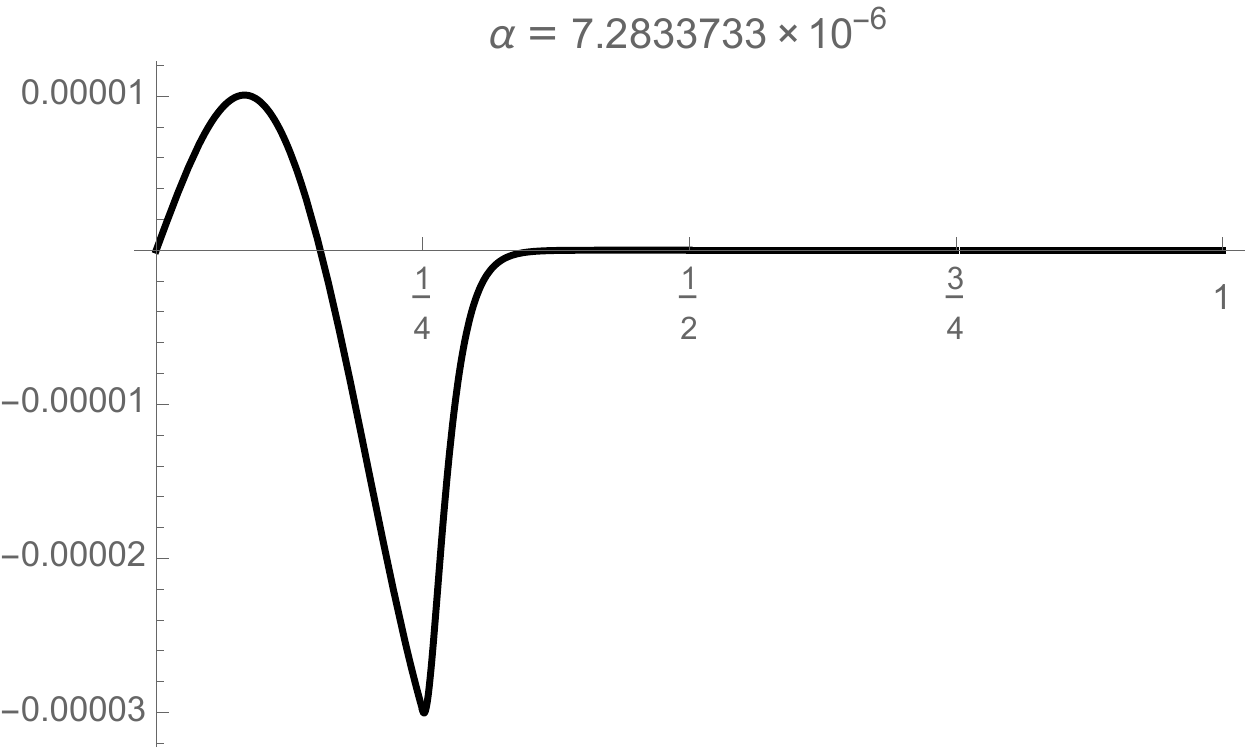}
	}
\\
	\subfloat[The real (left) and imaginary of the first eigenpair
        of $\cL_s$ with $R=(3/4,1)$.]
	{
          \includegraphics[width=0.45\textwidth]{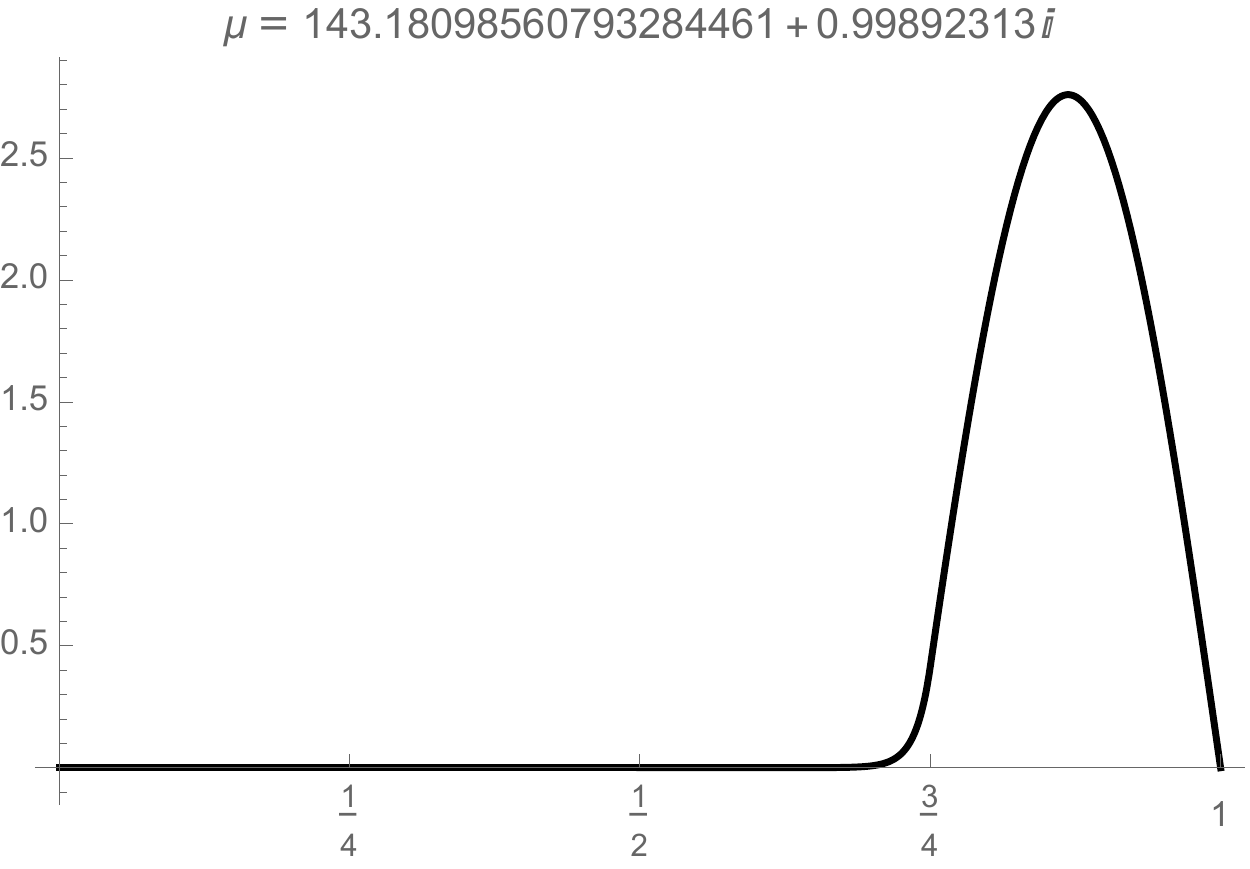}\quad
          \includegraphics[width=0.45\textwidth]{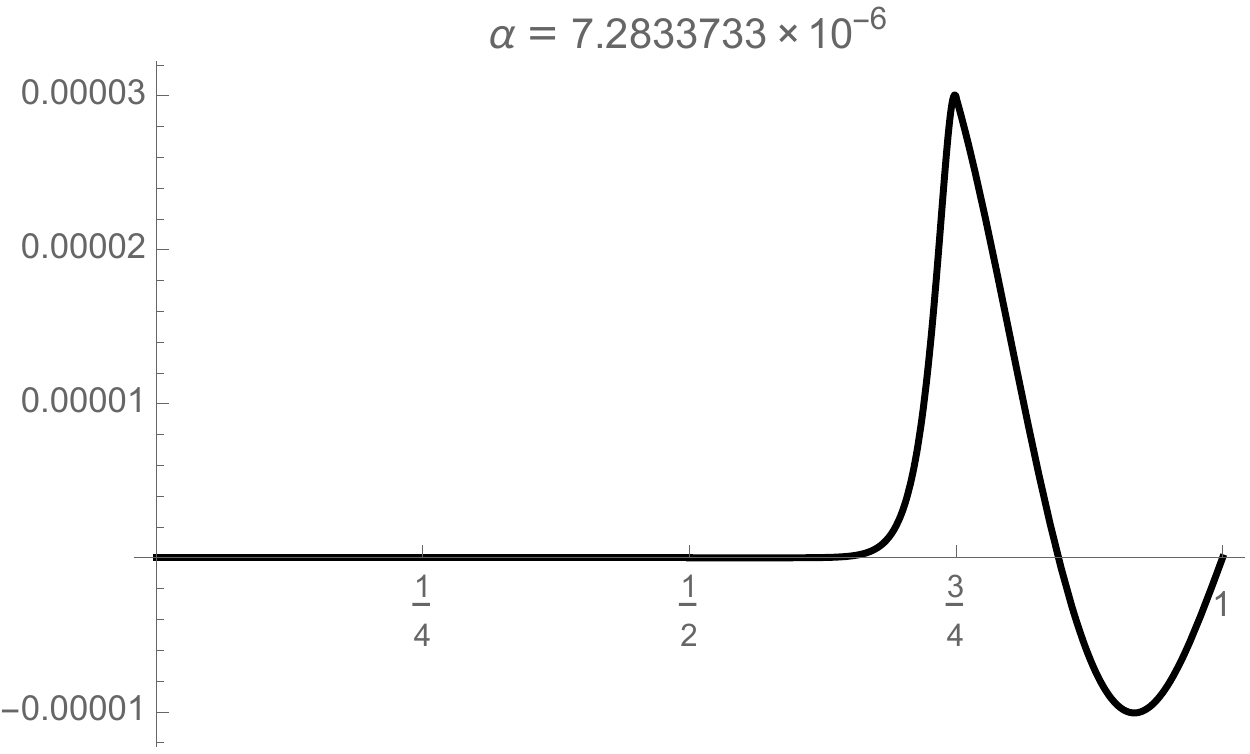}
        }
\caption{\label{FalseEx2} Plots of the first two eigenvectors of
  $\cL$, and of the real an imaginary parts of the first eigenvector
  of $\cL_s$ for different choices of $R$, and $s=1$.  See Example~\ref{FalseLocalizationExample2}.}
\end{figure}
 \begin{figure}
	\centering
	\subfloat[The first two eigenvectors of $\cL$.]
	{
          \includegraphics[width=0.45\textwidth]{Figures/FalseLocalizationX1.pdf}\quad
          \includegraphics[width=0.45\textwidth]{Figures/FalseLocalizationX2.pdf}
	}
\\
	\subfloat[The real (left) and imaginary of the first eigenpair
        of $\cL_s$ with $R=(0,1/4)\cup(3/4,1)$.]
	{
          \includegraphics[width=0.45\textwidth]{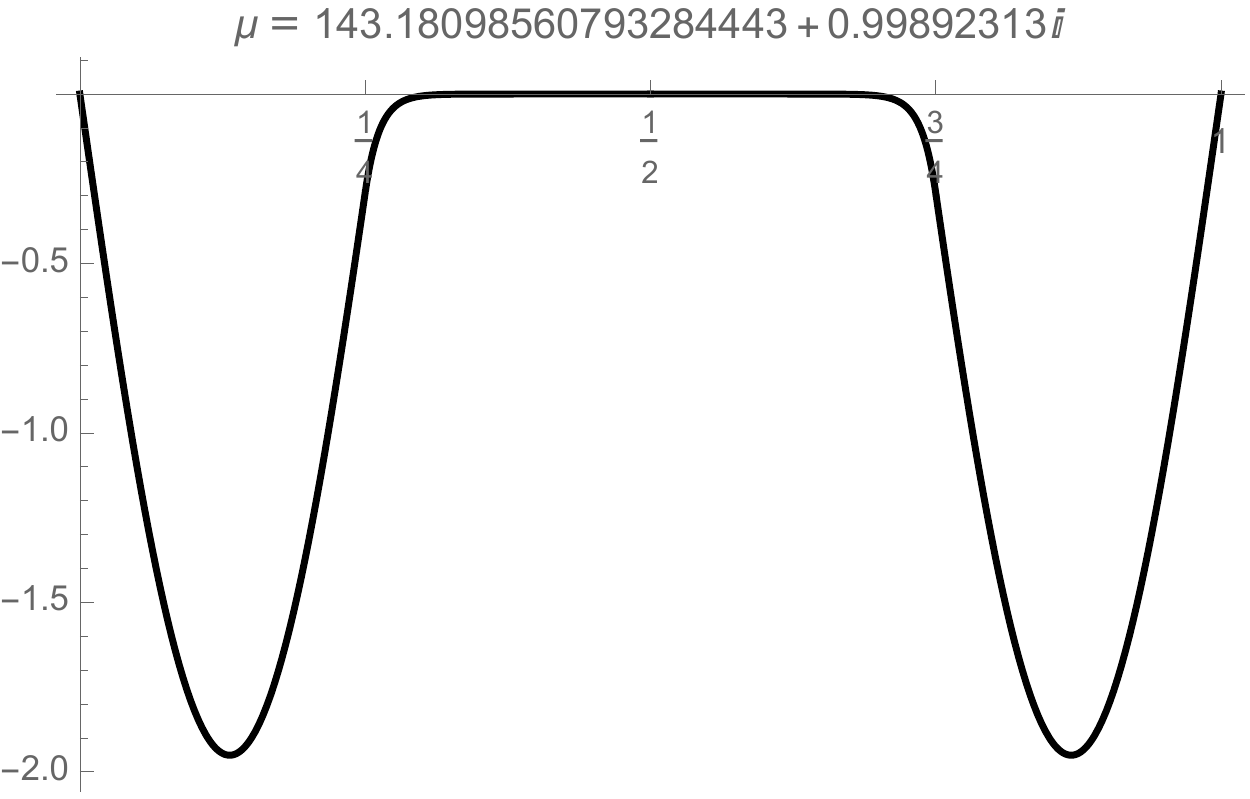}\quad
          \includegraphics[width=0.45\textwidth]{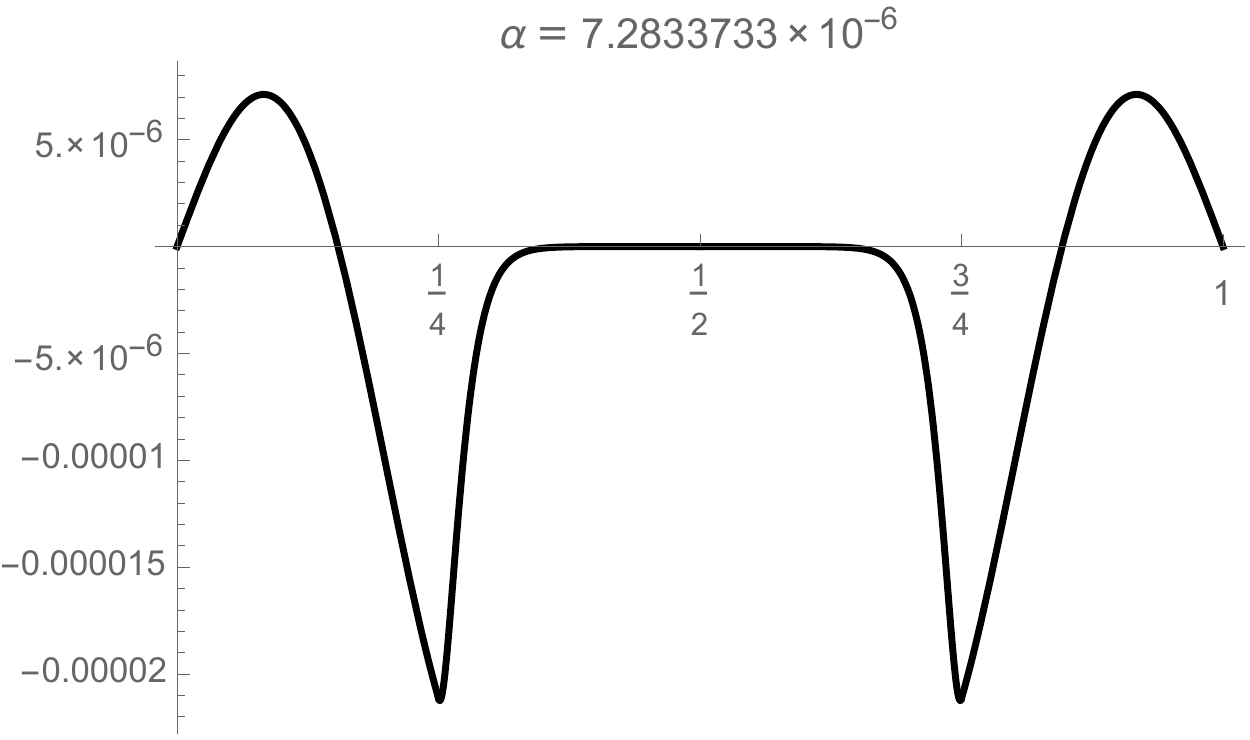}
        }
\\
	\subfloat[The real (left) and imaginary of the second eigenpair
        of $\cL_s$ with $R=(0,1/4)\cup(3/4,1)$.]
	{
          \includegraphics[width=0.45\textwidth]{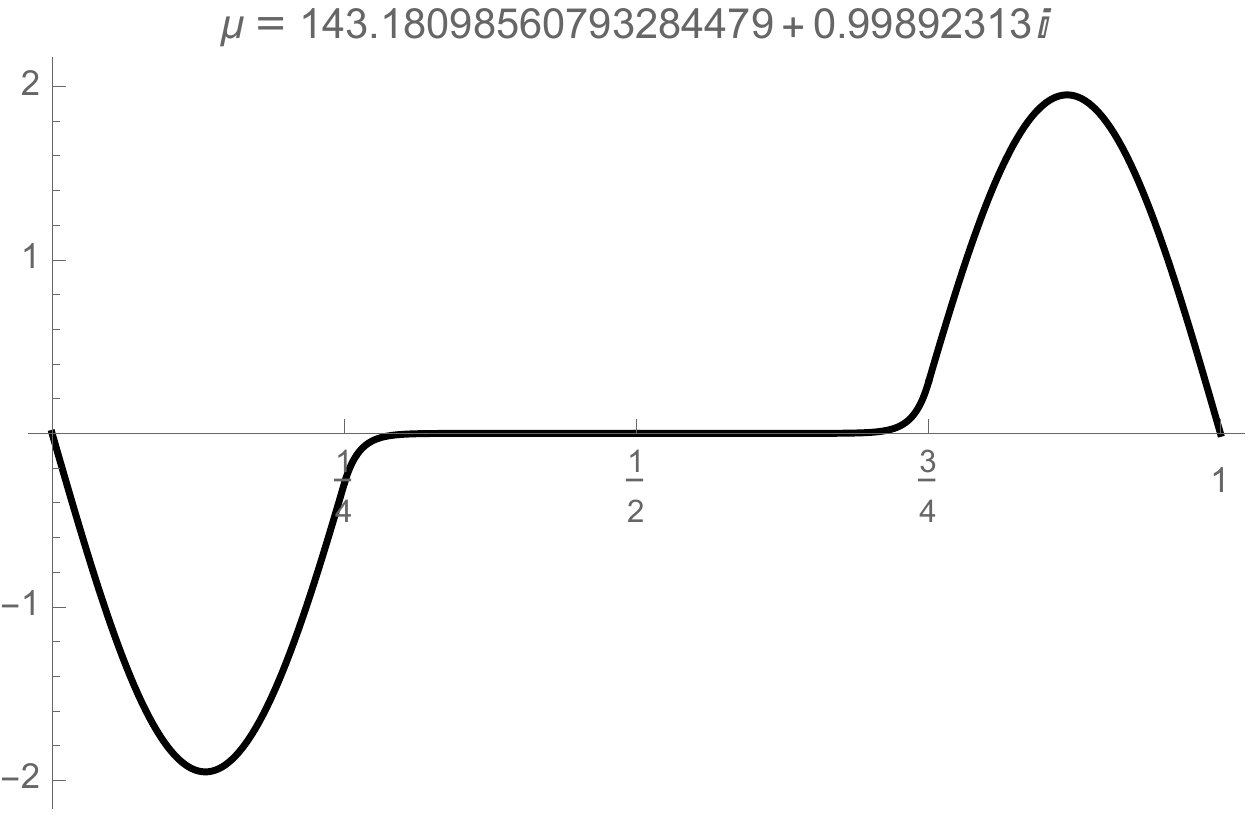}\quad
          \includegraphics[width=0.45\textwidth]{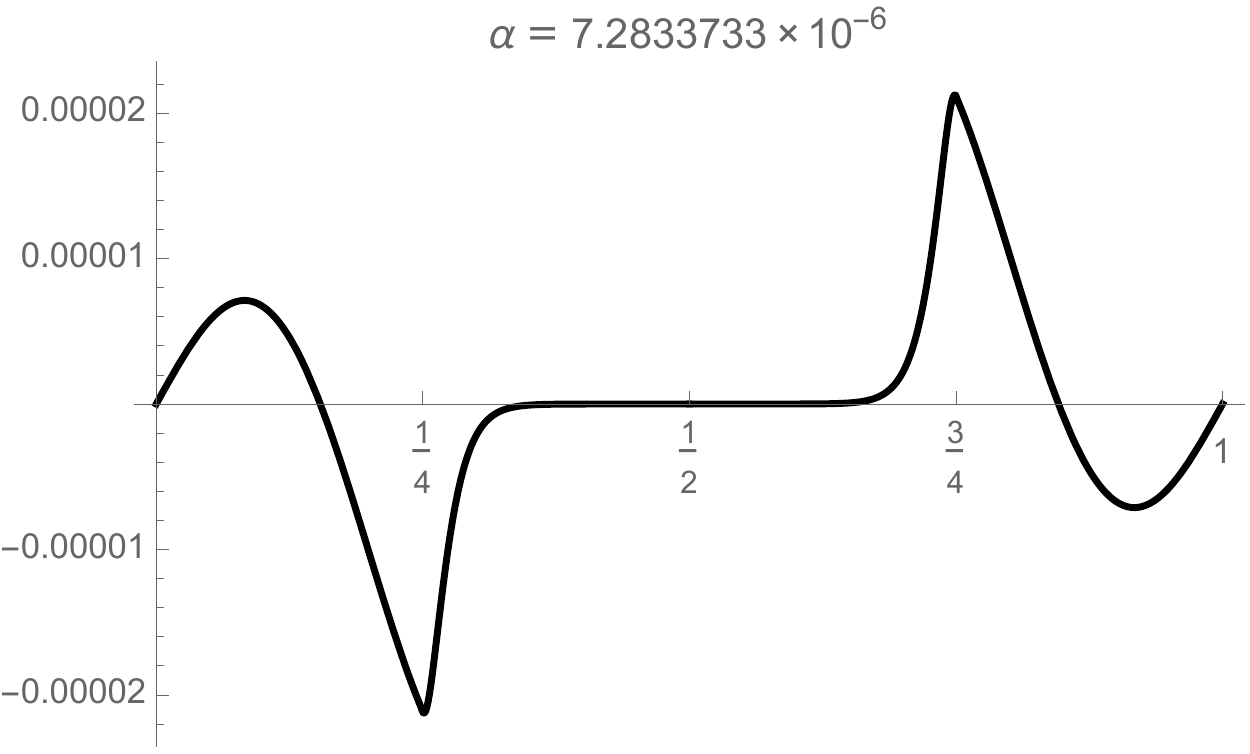}          
	}        
\caption{\label{FalseEx2B} Plots of the first two eigenvectors of
  $\cL$, and of the real an imaginary parts of the first eigenvector
  of $\cL_s$ for different choices of $R$, and $s=1$.  See Example~\ref{FalseLocalizationExample2}.}
\end{figure}

\begin{remark}\label{ShiftInComplement}
  As suggested by~\eqref{AltEigenvalueEstimate}, one could easily
  develop results analogous to those in this section but for which the
  roles of $R$ and $\Omega\setminus R$ are ``reversed''.  While still
  being concerned with localization in $R$, the complex-shifted operator
  would be defined instead as
  $\widetilde\cL_s=\cL+\ii\,s\chi_{\Omega\setminus R}$.  The
  analogue of the statement given at the beginning of this section
  is:
  \begin{quote}
  If $(\lambda,\psi)$ is an eigenpair of $\cL$ with $\psi$ highly
  localized in $R$, then there will be an eigenpair $(\mu,\phi)$ of
  $\widetilde\cL_s$ that is close to $(\lambda,\psi)$.
\end{quote}
The analogues
of~\eqref{EigenvalueCloseness} and~\eqref{EigenvalueCloseness3} in this case are
\begin{align}\label{AltEigenvalueCloseness}
\dist(\lambda\,,\,\Spec(\widetilde\cL_s))&\leq s\delta(\psi,R)
\quad,\quad
 \dist(\Re\mu,\Spec(\cL))\leq s\delta(\phi,R)\tau(\phi,R)~,
\end{align}
for eigenpairs $(\lambda,\psi)$ of $\cL$ and $(\mu,\phi)$ of
$\widetilde\cL_s$.  There are similarly predictable analogues of the
other results in this section.  A careful comparison of an approach
based on $\widetilde\cL_s$ with that based on $\cL_s$ is a topic of future
research.
\end{remark}

\end{example}

\section{An Algorithm Template for~\eqref{KeyTask}}\label{Template}
Guided by the results of Section~\ref{Theory}, we propose the algorithm
  template given in Algorithm~\ref{ELAT} for our fundamental task~\eqref{KeyTask}, which we restate here for convenience,
\begin{equation}
  \tag{T}\label{KeyTask2}
  \parbox{\dimexpr\linewidth-4em}{%
    \strut
    Given a subdomain $R$, an interval $[a,b]$ and a (small) tolerance $\delta^*>0$, find all
eigenpairs $(\lambda,\psi)$ for which $\delta(\psi,R)\leq \delta^*$ and $\lambda\in[a,b]$, or
determine that there are not any.%
    \strut
  }
\end{equation}
Practical realizations are obtained by choices made on lines~\ref{alg1}
and~\ref{alg2} of Algorithm~\ref{ELAT}, and we describe below
reasonable choices for each.
Tables~\ref{1DExampleTab} and~\ref{1DExampleTabB}, in which some 
eigenvalues of $\cL_s$ in $U$ corresponded to an eigenvector of $\cL$
that was not sufficiently localized in $R$, demonstrate why line~\ref{alg3} is needed.

\begin{algorithm}
\caption{Eigenvector Localization Template}\label{ELAT}
\begin{algorithmic}[1]
  \Procedure{Localize}{$a,b,s,\delta^*,R$}
       \State determine all eigenpairs $(\mu,\phi)$ of $\cL_s$ with $\mu\in U(a,b,s,\delta^*)$\label{alg1}
       \Comment{Corollary~\ref{SearchRegionCor}}
       \If{no eigenvalues are found in \ref{alg1}}
       \State exit
       \Comment{There are no eigenpairs $(\lambda,\psi)$ of $\cL$ with
         $\lambda\in [a,b]$ and $\delta(\psi,R)\leq\delta^*$}
       \Else
       \For{each eigenpair $(\mu,\phi)$ found in \ref{alg1}}
       \State $\phi\longleftarrow c\phi$ \label{alg1rot}
       \Comment{Normalize $\phi$, Remark~\ref{EigenvectorRotation}}
       \State post-process $(\Re\mu,\Re\phi)$ to obtain (approximate)
       eigenpair $(\tilde\lambda,\tilde\psi)$ of $\cL$ \label{alg2}
       \If{$\delta(\tilde\psi,R)\leq\delta^*$}\label{alg3}
       \State accept $(\tilde\lambda,\tilde\psi)$
       \EndIf
       \EndFor
       \EndIf
       \State \textbf{return} accepted (approximate) eigenpairs $(\tilde\lambda,\tilde\psi)$
\EndProcedure
\end{algorithmic}
\end{algorithm}

We first consider line~\ref{alg1} of Algorithm~\ref{ELAT}.  There are
several classes of methods that are designed for finding eigenpairs
$(\mu,\phi)$ (or just eigenvalues) of an operator, with $\mu$ in some
user-specified region $\widetilde{U}\subset\CC$, which we will assume
is simply connected and has a (piecewise) smooth boundary
$\gamma=\partial\widetilde U$.  We mention methods that are
based on associated contour integrals, and classify them into four
categories: Sakurai-Sugiura methods (SS, CIRR)
(cf.~\cite{Sakurai2003,Sakurai2007,Imakura2014,Ikegami2010,Yokota2013,Austin2015}),
FEAST methods (cf.~\cite{Polizzi2009,Tang2014,Kestyn2016,Yin2019a,Ye2017,Gopalakrishnan2019,Gopalakrishnan2020,Horning2020}), Beyn methods (cf.~\cite{Beyn2011,Beyn2012,Kleefeld2013,Beyn2014,VanBarel2016}), and Spectral
Indicator Methods (RIM, SIM)
(cf.~\cite{Huang2016,Huang2018,Liu2019,Xiao2020}).  Unlike the other
three approaches, Spectral Indicators Methods do not involve the
approximate solution of eigenvalue problems, and yield only eigenvalue approximations.

As we use the FEAST approach for our experiments in Section~\ref{Experiments}, we
provide a brief high-level description of how it works for a normal (or
selfadjoint) operator $\cA$ having compact resolvent, such as $\cL_s$ or $\cL$.  Although we
are primarily concerned with applying FEAST to the normal operator
$\cL_s$, we describe it first for selfadjoint $\cA$, and then indicate
how it can be made applicable to normal operators.  Suppose that
$f=f(z)$ is a rational function that is bounded on $\Spec(\cA)$.  Then
$\cB\doteq f(\cA)$ is a bounded (normal) operator on
$\Dom(\cB)\doteq\Dom(\cA)$, and if $(\lambda,\psi)$ is an eigenpair of
$\cA$, then $(f(\lambda),\psi)$ is an eigenvector of $\cB$.  We
emphasize that the eigenvectors of $\cA$ and $\cB$ are the same!
Now suppose that the open set $\widetilde{U}$ contains some finite
subset $\Lambda\subset\Spec(\cA)$ and that the contour
$\gamma=\partial\widetilde{U}$ does not intersect $\Spec(\cA)$.  The
rational function $f$ is then chosen as an approximation of the
characteristic function for $\widetilde{U}$, $f(z)\approx
\chi_{\widetilde{U}}(z)=\frac{1}{2\pi\,\ii}\oint_{\gamma}(\xi-z)^{-1}\,d\xi$.
This rational approximation is often obtained from a quadrature
approximation of this Cauchy integral, taking the form
$f(z)=\sum_{k=0}^{n-1}w_k(z_k-z)^{-1}$, but there are other ways of
obtaining such a rational function (cf.~\cite{Guettel2015,VanBarel2016a}).
It follows that $\cB$ approximates (in some sense) the spectral
projector $S$ for $\cA$ associated with $\Lambda$,
i.e. $S=\chi_{\widetilde{U}}(\cA)=\frac{1}{2\pi\,\ii}\oint_{\gamma}(\xi-\cA)^{-1}\,d\xi$.
We have $E\doteq\mbox{Range}(S)=E(\Lambda,\cA)$, which is the target
invariant subspace.  FEAST is based on subspace iteration using the
``filtered operator'' $\cB$:  starting with a random
finite-dimensional subspace
$E_0\subset\Dom(\cA)$ that satisfies $SE_0=E$, the iteration generates
a sequence of subspaces $E_{k+1}=\cB E_k$ that converge to $E$ with
respect to subspace gap.  A Rayleigh-Ritz procedure is used on a
finite rank operator $\cA_k:\,E_k\to E_k$ to obtain approximations
$\Lambda_k$ that converge to $\Lambda$ in the Hausdorff metric, and a
natural by-product of this procedure is that an orthonormal basis of
$E_k$ is obtained.   More specifically,
$\cA_k=P_{k}\cA_{\vert_{E_k}}$, where $P_k$ is the orthogonal
projector onto $E_k$, and $\Lambda_k=\Spec(\cA_k)$.  The rate of
convergence is governed by the ratio
\begin{align}\label{FilterContraction}
\kappa\doteq\frac{\sup_{\lambda\in \Spec(\cA)\setminus\Lambda}|f(z)|}{\inf_{\lambda\in \Lambda}|f(z)|}<1~,
\end{align}
so a good ``filter function'' $f$ for the region $\widetilde{U}$
should decay rapidly (in modulus) away from $\widetilde{U}$, and ideally not vary
too much within $\widetilde{U}$.  For non-selfadjoint (normal) $\cA$,
approximations $E_k$ and $E_k^*$ of the right and left invariant
subspaces $E$ and $E^*$ of $\cA$ are obtained using subspace iteration
with $\cB$ and its adjoint $\cB^*$, with some variations on how to
extract eigenvalue approximations and maintain well-conditioned bases
of  $E_k$ and $E_k^*$ (cf.~\cite{Kestyn2016,Yin2019a}).

It is convenient to use
\begin{align}\label{BunimovichStadium}
  \widetilde{U}=\widetilde{U}(a,b,s,\delta^*)=\{z\in\CC:\,\dist(z,L)\leq
  s\delta^*\}\quad,\quad
  L=[a,b]+\ii\,s~,
\end{align}
for the search region.  Recall that the eigenvalues of $\cL_s$ that
are of interest are in its lower-half, $U$, which is pictured in Figure~\ref{SearchRegionFig}.
The region $\widetilde{U}$, when viewed as a domain in $\RR^2$, is often called a \textit{Bunimovich
  stadium} in the context of quantum billiards, where it serves as a
popular example (cf.~\cite{Burq2005}).  We will refer to $\gamma=\partial\widetilde{U}$
as a \textit{Bunimovich curve}.  A unit-speed parameterization of $\gamma$ that traverses it
counter-clockwise, starting at the point $b+\ii(s-r)$, is given by
$z(t)=x(t)+\ii\,y(t)$, where
\begin{align}\label{Parameterization}
  (x(t),y(t))=\begin{cases}
    r(\sin(\frac{t}{r}), -\cos(\frac{t}{r}))+(b,s)&,\, 0\leq t\leq t_1\\
    (b+t_1-t,s+r)&,\,t_1\leq t\leq t_2\\
    r(\sin(\frac{t+a-b}{r}),-\cos(\frac{t+a-b}{r}))+(a,s)&,\,t_2\leq
    t\leq t_3\\
    (a-t_3+t,s-r)&,\, t_3\leq t\leq P
    \end{cases}~,
\end{align}
and
\begin{align}\label{Parameters}
  r=s\delta^*\;,\; t_1=\pi r\;,\;t_2=t_1+b-a\;,\;t_3=t_2+\pi r\;,\; P= 2\pi r+2(b-a)~.
\end{align}
The parameterization is made $P$-periodic by setting
$z(t+P)=z(t)$.  The $n$-pole rational filter function
associated with $\gamma$ is obtained by applying the trapezoid rule to
$\frac{1}{2\pi\,\ii}\oint_{\gamma}(\xi-z)^{-1}\,d\xi$,
\begin{align}\label{BunimovichFilter}
  f(z)=\sum_{k=0}^{n-1}w_k(z_k-z)^{-1}\quad,\quad z_k=z(k h+\pi
  r/2)\quad,\quad w_k=\frac{h z'(k h+\pi
  r/2)}{2\pi\,\ii}\quad,\quad h=\frac{P}{n}~.
\end{align}
The offset of $\pi r/2$ in the definition of the quadrature points
(poles) $z_k$ and weights $w_k$ provides a more symmetric distribution
of these points.  An example Bunimovich curve, overlaid with the poles
of $f(z)$, is given in Figure~\ref{BunimovichFilterFig}, together with
a contour plot of $|f(z)|$ that illustrates its effectiveness in
distinguishing between points inside $\gamma$ from those outside.
The thick red curve in the contour plot is the portion of $\gamma$ for
which $\Im z\leq s$---recall that eigenvalues $\mu$ of $\cL_s$ satisfy
$\Im\mu <s$.
The thick black curves in the contour plot are the contours
$|f(z)|=2^{j}$, $-8\leq j\leq 2$, with the curve for $|f(z)|=2^{-8}$
being farthest from $\gamma$.  These indicate the desired rapid decay
of $|f(z)|$ away from $\widetilde U$.
\begin{figure}
	\centering
	\subfloat[Bunimovich curve $\gamma=\partial \widetilde U$ and poles (quadrature points) of $f(z)$.]
	{
          \includegraphics[width=0.45\textwidth]{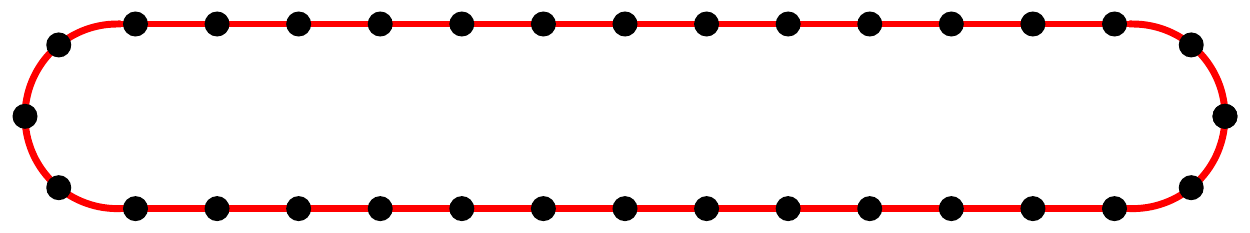}\quad
	}
\quad
	\subfloat[Contour plot of $|f(z)|$ for $a-2r\leq \Re z\leq
        b+2r$ and $s-2r\leq \Im z\leq s$.
         The red curve is part of $\gamma=\partial \widetilde U$.]
	{
          \includegraphics[width=0.45\textwidth]{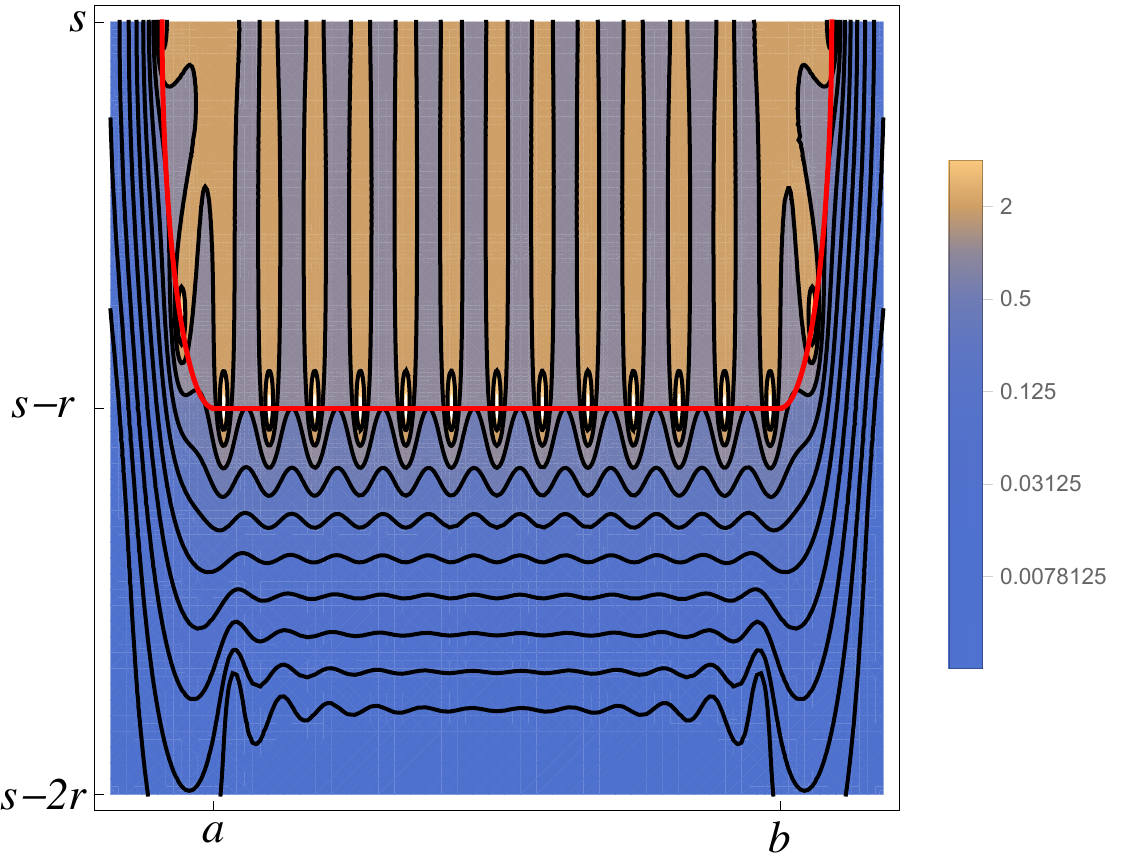}
	}
        \caption{\label{BunimovichFilterFig} Poles and contour plot of
          $|f(z)|$ for the rational filter $f(z)$ associated with
          $\widetilde U=\widetilde U(a,b,s,\delta^*)$ for $(a,b,s,\delta^*)=(-4,18,10,1/5)$; $n=32$ equispaced points are used for the filter.}
      \end{figure}

We now consider the post-processing step on line~\ref{alg2}.   As
suggested by Theorem~\ref{KeyTheorem2} and
Proposition~\ref{ResidualBound}, the pair
$(\mu_1,\phi_1)=(\Re\mu,\Re\phi)$ is often a decent starting point for
finding an eigenpair $(\lambda,\psi)$ of $\cL$ with $\lambda$ ``close'' to
$\mu_1$; recall that $|\mu_1-\lambda|\leq
s\delta(\phi,R)\tau(\phi,R)$.  A reasonable post-processing procedure
consists of a few inverse iterations, see
Algorithm~\ref{PostProcess}.
\begin{algorithm}
\caption{Approximate Eigenpair Post-Processing}\label{PostProcess}
\begin{algorithmic}[1]
  \Procedure{Post-process}{$\mu_1,\phi_1,\mathrm{tol}$}
  \State $\tilde\lambda\longleftarrow\mu_1$
  \State $\tilde\psi\longleftarrow \phi_1/\|\phi_1\|_{L^2(\Omega)}$ 
  \While{$\|\cL
    \tilde\psi-\tilde\lambda\tilde\psi\|_{L^2(\Omega)}>\mathrm{tol}$}\label{ResidualCheck}
  \State $\tilde\psi\longleftarrow (\mu_1-\cL)^{-1}\tilde\psi$\label{InverseIteration}
  \State $\tilde\psi\longleftarrow
  \tilde\psi/\|\tilde\psi\|_{L^2(\Omega)}$
  \State $\tilde\lambda\longleftarrow (\cL\tilde\psi,\tilde\psi)$
  \EndWhile
  \State \textbf{return} post-processed (approximate) eigenpair $(\tilde\lambda,\tilde\psi)$
  \EndProcedure
\end{algorithmic}
\end{algorithm}
One might instead opt for Rayleigh quotient iterations, which replace
$\mu_1$ with the current approximation $\tilde\lambda$ on
line~\ref{InverseIteration} of Algorithm~\ref{PostProcess}.  However, it is expected that few 
iterations will be needed, so requiring the action of only one inverse is perhaps
more attractive.  We note that Proposition~\ref{ResidualBound} allows
for the efficient computation of the initial residual, which may
already be below the prescribed tolerance, resulting in no inverse iterations.  In practice, one might use a
more readily computable proxy for the residual norm.

In Remark~\ref{FalsePositivePP} below, we highlight a potential danger
of relying on the post-processing procedure in
Algorithm~\ref{PostProcess} as stated in situations in which $\mu_1=\Re\mu$ is
close to an eigenvalue of $\cL$, but $\phi_1=\Re\phi$ is not close to
an eigenvector despite the fact that
$\|\cL\phi_1-\mu_1\phi_1\|_{L^2(\Omega)}$ is reasonably small.
Example~\ref{FalseLocalizationExample2} was chosen precisely to
illustrate how such scenarios could arise.  While we do not expect
such situations to be common enough to reject
Algorithm~\ref{PostProcess} as a viable option, it is useful to consider
possible variants that are likely to be more robust in
such situations.  One such variant is to first estimate (as efficiently as
possible) how many
eigenvalues of $\cL$ are ``near'' $\mu_1$ (cf.~\cite{DiNapoli2016,Yin2019b}),
as this has a direct affect
on the convergence rate of inverse iteration.  Recall that
we are guaranteed that there is at least one eigenvalue of $\cL$ that
is within $s\delta(\phi,R)\tau(\phi,R)$ of $\mu_1$, so we might consider a
slightly larger interval around $\mu_1$ for our eigenvalue count
estimate.  If the approach estimates $m$ eigenvalues of $\cL$ near
$\mu_1$, then inverse iteration would be performed using a subspace of
size at least $m$.  After extracting (approximate) eigenpairs, each
would be tested for its localization in $R$, as in lines 9-11 of Algorithm~\ref{ELAT}.
A more in-depth discussion of variants of the post-processing algorithm will
be postponed for subsequent work.

\begin{remark}[Possible False Positives from
  Post-Processing]\label{FalsePositivePP}
  As was demonstrated in Example~\ref{FalseLocalizationExample2}, it
  is possible for  $\|\cL
  \tilde\psi-\tilde\lambda\tilde\psi\|_{L^2(\Omega)}$ to be relatively
  small without $(\tilde\lambda,\tilde\psi)$ being close to an eigenpair of
  $\cL$.  In that example, we would have $\|\cL
  \tilde\psi-\tilde\lambda\tilde\psi\|_{L^2(\Omega)}=7.27553\times
10^{-6}$ for the initial check on line~\ref{ResidualCheck}.  If the
tolerance in Algorithm~\ref{PostProcess} was chosen larger than this, no inverse
iterations would be performed, and the procedure would return its
input, which is not close to an eigenpair of $\cL$ and would falsely
indicate an eigenvector of $\cL$ that is localized in $R=(0,1/4)$.  Setting a
smaller tolerance in this case will force at least one inverse iteration,
but the question of when the tolerance is small enough to be
considered ``safe'' for the types of problems of interest is a subtle
one.  In fact, by increasing the constant $V$ in
Example~\ref{FalseLocalizationExample2}, the initial residual $\|\cL
  \tilde\psi-\tilde\lambda\tilde\psi\|_{L^2(\Omega)}$ can be made
  arbitrarily small, so no tolerance would seem safe.  For
  Example~\ref{FalseLocalizationExample2} with $V=80^2$ (as was used
  in that example),  the form of
  $(\cL-\mu_1)^{-1}\tilde\psi$ is known in advance---linear
  combinations of (regular and/or hyperbolic) sines and cosines having
  known frequencies on each subinterval---so inverse iterations can be
  carried out by solving linear systems that enforce the boundary
  conditions and the continuity of the function and its derivative
  across subintervals.  After performing the first inverse iteration, and
  renormalizing, the resulting function hardly differs from its
  predecessor at all---the maximal pointwise difference between the
  two functions is on the order of $10^{-9}$---which suggests that
  {\bf many} inverse iterations would be required (in essentially
  exact arithmetic!) before the iterates began to reasonably
  approximate the true eigenvector.  The extremely slow convergence of
  inverse iteration in this case is expected, due to the fact that
  there are two eigenvalues of $\cL$ that are extremely close to
  $\mu_1$ (and to each other).  Since the next nearest eigenvalues of
  $\cL$ are much farther away from $\mu_1$, performing inverse
  iteration with a two-dimensional subspace, (re)orthogonalizing its basis
  as needed, will lead to much more rapid convergence, from which
  (approximate) eigenpairs can be easily extracted using a
  Rayleigh-Ritz procedure.  The initial basis might be chosen
  randomly, or one might choose $\tilde\psi=\phi_1/\|\phi_1\|_{L^2(\Omega)}$ as one of the two basis
  functions and the other to be the normalized version of the orthogonal
  complement $u^\perp$ of the
  landscape function $u$ with respect to $\phi_1$, i.e.
  $u^\perp=u-(u,\tilde\psi)\,\tilde\psi$.  In this case, the ``clever''
  choice of initial basis leads to very few inverse iterations before
  approximate eigenpairs very close to those given in
  Figure~\ref{FalseEx2}(A) are obtained by the Rayleigh-Ritz
  procedure.   From these two, one can then deduce that there are
  no eigenvectors of $\cL$ whose eigenvalues are near $\mu_1$ that are
  localized in $R=(0,1/4)$, though one can see that they are localized
  in $R=(0,1/4)\cup(3/4,1)$.
\end{remark}

\begin{remark}\label{SIMVariant}
  As indicated near the beginning of this section, an approach such as SIM,
  which yields only eigenvalue approximations, might be used in line 2
  of Algorithm~\ref{ELAT}.  In this case, line 7 is clearly
  irrelevant, and the post-processing phase would proceed with only
  eigenvalue approximations.  Since inverse iteration is used, and can
  proceed with random initialization, such an approach is feasible.
  The potential reduction in cost by using SIM on line 2 might make up
  for the potential increase in cost of using inverse iteration with
  \textit{random} initialization, as opposed to the (likely) better
  initialization obtained from methods that return eigenpairs in line 2.
\end{remark}

Up to this point, we have not discussed how to choose the parameter
$s$ in $\cL_s$.  Example~\ref{FalseLocalizationExample1} illustrates
that choosing $s$ ``too large'' can introduce false indicators of
localization that would later have to be recognized and
rejected, but how large is ``too large'' in terms of producing false
indicators may be problem-dependent, as can be seen in
Examples~\ref{1DExample} and~\ref{FalseLocalizationExample2}.  In the
first of these examples, choosing $s=1$ or $s=100$ had very little
practical effect on the how well eigenpairs $(\mu,\phi)$ of $\cL_s$
with $\mu\in U(a,b,s,\delta^*)$ served as predictors of eigenpairs
$(\lambda,\phi)$ of $\cL$ for which $\lambda\in[a,b]$ and
$\delta(\phi,R)\leq \delta^*$.
Example~\ref{FalseLocalizationExample2}, which was specifically
designed to yield false indicators of localization even when $s$ is
small, probably should not be weighed so heavily in coming up with
practical guidance about how to choose $s$, but nonetheless
illustrates that the quality of localization indicators coming from
line~\ref{alg1} of Algorithm~\ref{ELAT} can be quite sensitive to the
choice of $s$ for certain problems.  Given the localization tolerance
$\delta^*$, there is some theoretical appeal to choosing $s$ so that
$r=s\delta^*\leq 1$, because it makes it easier for the bounds in
eigenvector results such as~\eqref{EigenvectorCloseness}
and~\eqref{EigenvectorCloseness2} to be meaningful (i.e. smaller than
$1$).  However, Example~\ref{1DExample} again shows that such a
restriction is not necessary.  It may be that a ``good'' choice of
$s$, in relation to $a$, $b$ and $\delta$, is dictated largely by
practical efficiency considerations concerning the method used for
computing eigenpairs $(\mu,\phi)$ of $\cL_s$ with
$\mu\in U(a,b,s,\delta^*)$.  For example, with the FEAST approach used
above, the rational filter $f(z)$ determines how rapidly its
iterations converge.  We recall that the key issue is the contrast
between $|f(\mu)|$ for $\mu\in U=U(a,b,s,\delta^*)$ and
$\mu\in \Spec(\cL_s)\setminus U$; the greater the contrast, the fewer
number of iterations are needed.  For fixed $(a,b,s,\delta^*)$,
increasing the number $n$ of poles in $f$
(see~\eqref{BunimovichFilter}) will reduce the number of iterations
needed, but increase the cost per iteration.  For a fixed $n$, the
aspect ratio of $U$ (equivalently $\widetilde U$) affects the quality of the
filter.  In Figure~\ref{BunimovichFilterFig},  $n=32$ poles produces a
very nice contrast for $(a,b,s,\delta^*)=(-4,18,10,1/5)$; the aspect
ratio $((b-a)+2r)/(2r)$ is $6.5$ in this case.  However,
changing only $s$ from $10$ to $1$ increases the aspect ratio to $56$,
and destroys the quality of the
filter.  We do not picture the poor filter here, but mention that
$|f(z)|\geq 1/2$ for some $z$ with $\Im z=s-2r$; the filter when
$s=10$ satisfied $|f(z)|< 2^{-7}$ (typically even smaller) for all such $z$.  Now, taking
$(a,b,s,\delta^*)=(-4,-9/5,1,1/5)$ restores the aspect ratio of $U$ to
$6.5$, and $|f(z)|$ looks precisely as that pictured in
Figure~\ref{BunimovichFilterFig}.  We note that the only change
between this and the ``poor filter'' situation is that $b$ was changed
to restore the original aspect ratio of  $6.5$.  For a fixed
$U(a,b,s,\delta^*)$ (and $n$), the situation can be improved by
subdividing $[a,b]$, $[a,b]=[a_0,a_1]\cup\cdots\cup [a_{p-1},a_p]$ for
some $p\in\NN$, where $a_j=a+j(b-a)/p$.  Each subregion
$U(a_j,a_{j+1},s,\delta^*)$, which will have a smaller aspect ratio
than $U(a,b,s,\delta^*)$, can be searched independently (in
parallel) for eigenvalues of $\cL_s$.  In light of the discussion
above, it appears that offering practical guidance for choosing $s$
may require significant experimentation, so
we postpone such judgments to later work that is more computationally focused.

\begin{remark}[Alternate Approach Using Bounded Operator]\label{AlternativeApproach}
  As indicated above, when $b-a$ is very large, the search region
  $\widetilde{U}(a,b,s,\delta^*)$ for eigenvalues of $\cL_s$ will
  typically have a very large aspect ratio, which would necessitate
  subdivision
  $\widetilde{U}(a,b,s,\delta^*)=\bigcup\{\widetilde{U}(a_{j},a_{j+1},s,\delta^*):\,0\leq
  j\leq p-1\}$.  Since each subregion
  $\widetilde{U}(a_{j},a_{j+1},s,\delta^*)$ can be explored
  independently and in parallel, large regions can be efficiently
  explored in practice when parallel computing is available.
However, a different approach might
be used that requires only a single search region even if one wants to
test all eigenvectors of $\cL$ for localization in $R$---assuming that
there are only finitely many linearly independent eigenvectors of $\cL$ that are localized in
$R$ to within a given tolerance (see Remark~\ref{InfiniteLocalization}
for a counterexample).  

Suppose that $\bb>0$ is a known or computed lower bound on
$\Spec(\cL)$.  The eigenvectors of $\cM=\bb\,\cL^{-1}$ are precisely
those of $\cL$, and we have $\Spec(\cM)\subset(0,1]$.  If we define
$\cM_s=\cM+\ii\,s\,\chi_R$, then
$\Spec(\cM_s)\subset (0,1]+\ii\,(0,s)$, and the obvious analogues of
the results in Section~\ref{Theory} hold for $\cM$ and $\cM_s$, as
they did for $\cL$ and $\cL_s$.  Now, instead of exploring potentially
many regions $\widetilde{U}(a_{j},a_{j+1},s,\delta^*)$ for eigenvalues
of $\cL_s$, we explore a single region
$\widehat{U}(s,\delta^*)=[0,1]+\ii[s(1-\delta^*),s]$ for eigenvalues
of $\cM_s$.  Of course, this assumes that $\widehat{U}(s,\delta^*)$
contains only finitely many eigenvalues of $\cM_s$ counting multiplicities.  The associated
filtered operator $\cB=f(\cM_s)$ for FEAST iterations has the form
\begin{align}
\cB=\sum_{k=0}^{n-1}w_k(z_k-\cM_s)^{-1}=\cL\sum_{k=0}^{n-1}w_k((z_k-\ii\,s\,\chi_R)\cL-\bb)^{-1}~.
\end{align}
A more thorough investigation of theoretical and practical
considerations related to such an approach is intended for future
work, though we note here that one must still contend with issues of
discretization related to resolving highly oscillatory eigenvectors.
\end{remark}

\begin{remark}\label{InfiniteLocalization}
  Although we expect that, when $R$ is a relatively small subdomain of
  $\Omega$, and $\delta^{*}<1/2$, there will typically be only
  finitely many eigenvectors $\psi$ of $\cL$ that satisfy
  $\delta(\psi,R)\leq \delta^*$, this need not be the case.  For
  example, let $\Omega$ be the unit disk and $\cL=-\Delta$.  The
  eigenvalues are known to be $\lambda_{m,n}=[j_n(m)]^2$, for $m\geq
  1$ and $n\geq 0$, where $j_{\sigma}(m)$ is the $m$th positive root of
  the first-kind Bessel function $J_\sigma$.   The corresponding
  eigenspaces, expressed in polar coordinates, are
  \begin{align*}
    E(\lambda_{m,n},\cL)=\mathrm{span}\{J_n(j_n(m)r)\sin(n\theta)\,,\, J_n(j_m(m)r)\cos(n\theta)\}=\mathrm{span}\{\psi^{(0)}_{m,n}\,,\, \psi^{(1)}_{m,n}\}~.
  \end{align*}
  Given $r^*\in (0,1)$, we consider the annulus $R$ for which
  $r^*<r<1$ ($0\leq\theta<2\pi$).  For $n\geq 1$, we have
  $\|\psi_{m,n}^{(k)}\|^2_{L^2(\Omega)}=
  (\pi/2)[J_{n+1}(j_n(m))]^2$, and
  \begin{align*}
    [\delta(\psi_{m,n}^{(k)},R)]^2=\frac{(r^*)^2([J_{n}(j_n(m) r^*)]^2- J_{n-1}(j_n(m) r^*) J_{n+1}(j_n(m)) r^*)}{[J_{n+1}(j_n(m) r^*)]^2}~.
  \end{align*}
  It can be shown that $\delta(\psi_{1,n}^{(k)},R)\to 0$ as
  $n\to\infty$, regardless of the choice of $r^*$, which is a
  way of quantifying the statement that the eigenvectors
  $\psi_{1,n}^{(k)}$ concentrate near the boundary of $\Omega$ for
  large $n$.  So we see that, even if $r^*$ is near $1$, so $R$ is
  relatively small compared to $\Omega$, and $\delta^*$ is small,
  there will be infinitely many linearly independent eigenvectors $\psi$ of $\cL$
  satisfying $\delta(\psi,R)\leq \delta^*$.
\end{remark}

\section{A Partial Realization of Algorithm~\ref{ELAT}}\label{Experiments}
An implementation of the FEAST algorithm that uses finite element
methods to discretize the associated operators is described
in~\cite{Gopalakrishnan2020} (see also~\cite{Gopalakrishnan2019}).
Corresponding code, Pythonic FEAST~\cite{Gopalakrishnan2017}, builds on the general
purpose finite element software package NGSolve~\cite{Schoeberl2014,Schoeberl2021}, and
provides a convenient user interface in Python.  Recent modifications
to Pythonic FEAST allow for normal operators such as $\cL_s$, which we
use to illustrate a partial realization of Algorithm~\ref{ELAT}.  In
this realization, the
computation (approximation) of eigenpairs of $\cL_s$ whose eigenvalues
are in the search region $U(a,b,s,\delta^*)$, and the renormalization
of the eigenvectors, is performed---up through
line~\ref{alg1rot} of Algorithm~\ref{ELAT}.  These computations
provide likely candidates for associated localized eigenvectors of
$\cL$,  to be obtained through post-processing and then finally
accepted or rejected based on the tolerance $\delta^*$.
The post-processing is not automated here, and we
instead rely on visual comparison and experience to determine the
eigenpairs of $\cL$ that correspond to those computed for $\cL_s$.
The final accept/reject decision for these eigenvectors of $\cL$ is
clear based on their $\delta$-values.

To illustrate the implementation, we have chosen an example for which
localization is due to domain geometry, as opposed to coefficients in
the differential operator.  The operator is $\cL=-\Delta$, with
homogeneous Dirichlet boundary conditions, and the domain $\Omega$
consists of three squares, one $4\times 4$, one
$3\times 3$ and one $2\times 2$, joined by two $2\times 1$ rectangular
``bridges'', as shown in Figure~\ref{BaselineComparison}.
We refer to $\Omega$ as the ``three bulb'' domain, and to each of the
squares as ``bulbs''.  Many simple constructions such as this could be
chosen to yield localization of some eigenvectors (cf.~\cite{Delitsyn2012,Grebenkov2013}), and
this domain was chosen because localization of eigenvectors in each of
the three bulbs occurs (multiple times) early in the spectrum.  The
eigenvalue problems were discretized using quadratic finite elements
on a fixed
(relatively fine) quasi-uniform triangular mesh having maximal edge
length 0.1, resulting a a finite element space of dimension 15493.

To provide a baseline for comparison, we computed the first 71
eigenpairs for this discretization of $\cL$, whose (discrete)
eigenvalues are in the range $(1.22,33.30)$.  Contour plots of the
first sixteen eigenvectors are given in Figure~\ref{BaselineComparison},
and exhibit instances localization in each of the three bulbs.
\begin{figure}
\caption{\label{BaselineComparison} Contour plots of the first sixteen
computed eigenvectors for the three bulb domain.  Computed eigenvalues are
given for each, as well as
localization measures for the bulb in which the
eigenvector is most localized: $\tau_\ell$ for left bulb, $\tau_r$ for
right bulb, and $\tau_m$ for middle bulb.}
\centering
\begin{tabular}{|c|c|c|c|}\hline
$\lambda=1.2297\,,\,\tau_\ell=0.9999$&$\lambda=2.1714\,,\,\tau_r=0.9997$&
$\lambda=3.0528\,,\,\tau_\ell=0.9994$&$\lambda=3.0842\,,\,\tau_\ell=1.0000$\\
\includegraphics[width=0.2\textwidth]{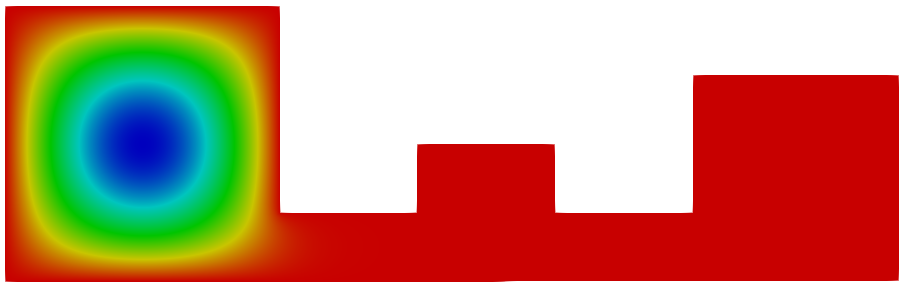}&
\includegraphics[width=0.2\textwidth]{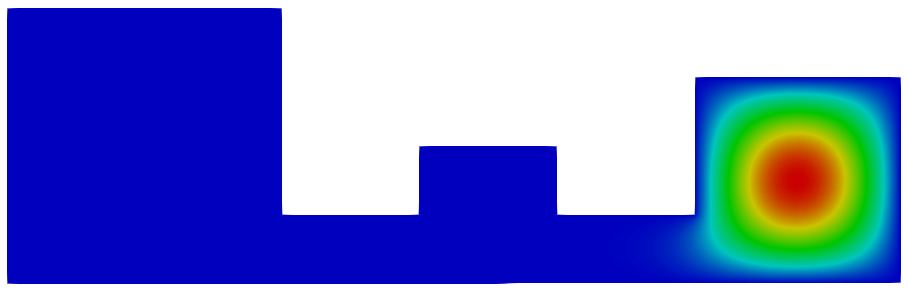}&
\includegraphics[width=0.2\textwidth]{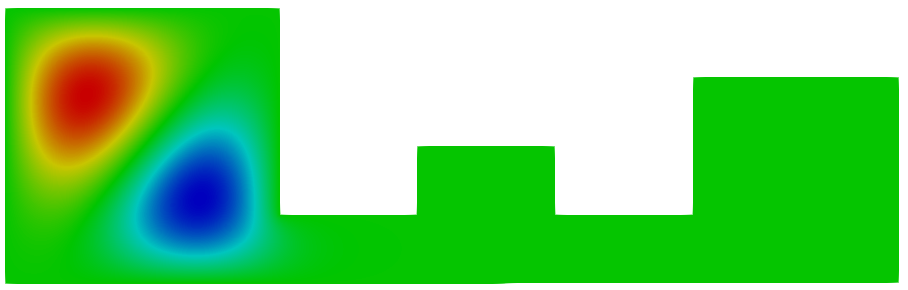}&
\includegraphics[width=0.2\textwidth]{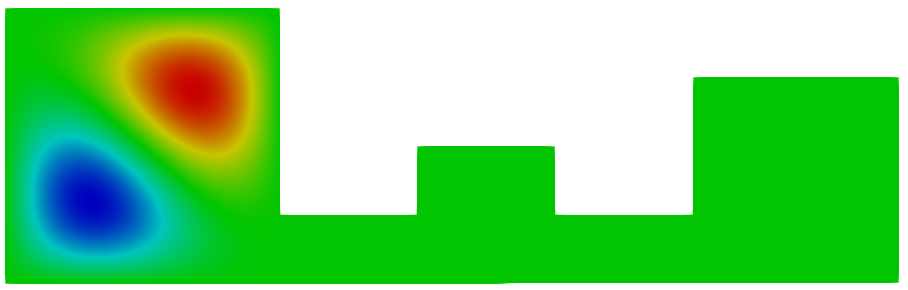}\\\hline
$\lambda=4.5029\,,\,\tau_m=0.9910$&$\lambda=4.8706\,,\,\tau_\ell=0.9983$&
$\lambda=5.3051\,,\,\tau_r=0.9952$&$\lambda=5.4827\,,\,\tau_r=1.0000$\\
\includegraphics[width=0.2\textwidth]{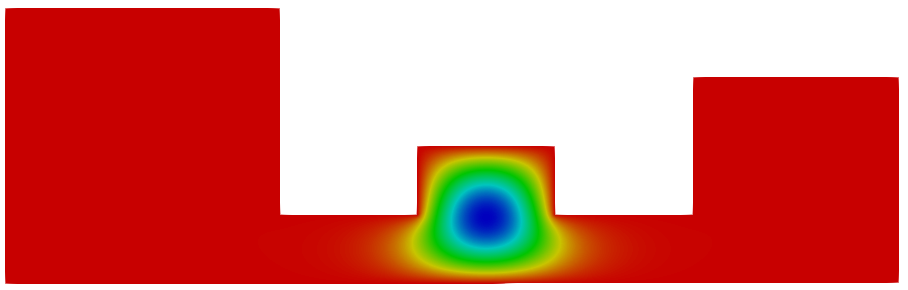}&
\includegraphics[width=0.2\textwidth]{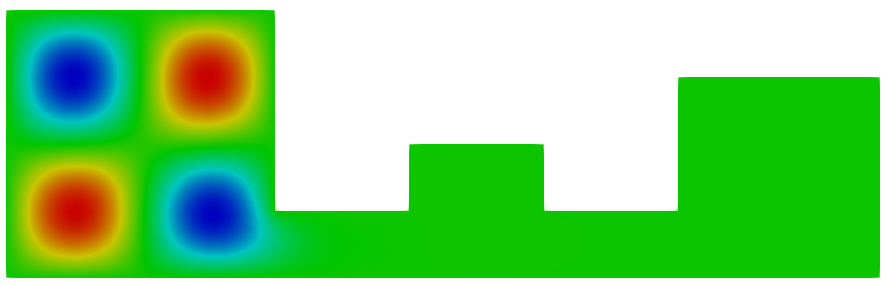}&
\includegraphics[width=0.2\textwidth]{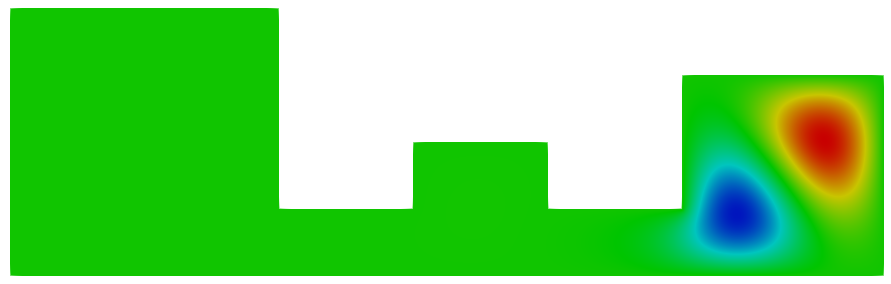}&
\includegraphics[width=0.2\textwidth]{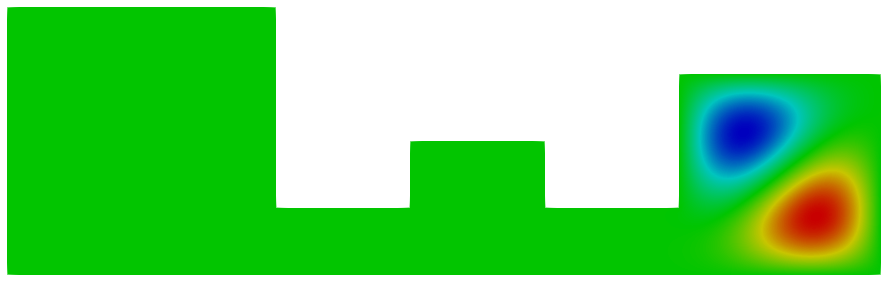}\\\hline
$\lambda=6.0910\,,\,\tau_\ell=0.9974$&$\lambda=6.1682\,,\,\tau_\ell=1.0000$&
$\lambda=7.7167\,,\,\tau_\ell=0.9845$&$\lambda=8.0185\,,\,\tau_\ell=1.0000$\\
\includegraphics[width=0.2\textwidth]{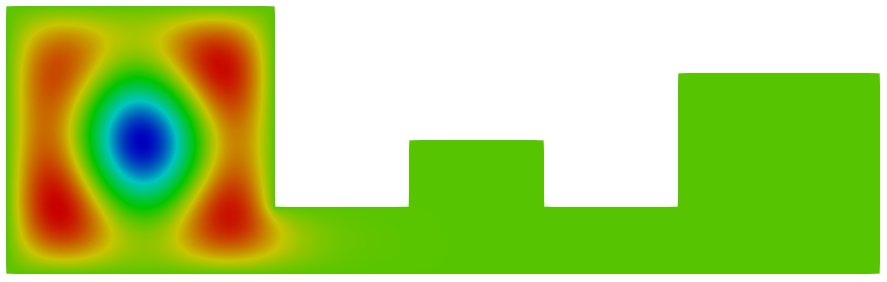}&
\includegraphics[width=0.2\textwidth]{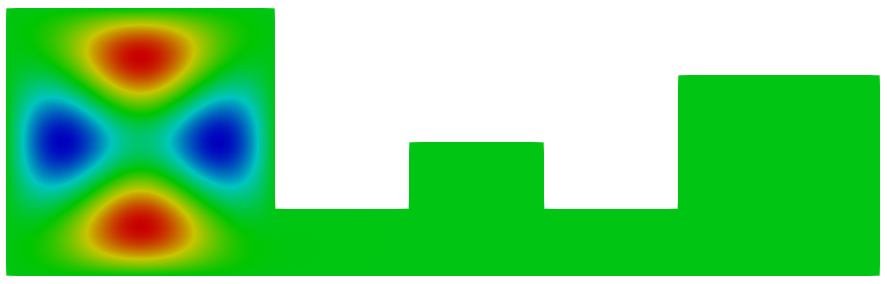}&
\includegraphics[width=0.2\textwidth]{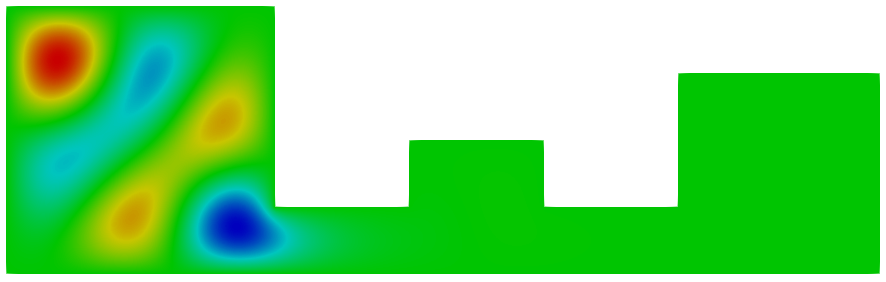}&
\includegraphics[width=0.2\textwidth]{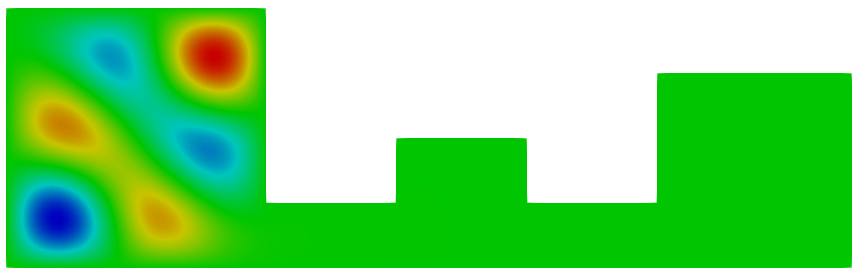}\\\hline
$\lambda=8.3618\,,\,\tau_r=0.9707$&$\lambda=9.3829\,,\,\tau_m=0.7945$&
$\lambda=10.029\,,\,\tau_\ell=0.8359$&$\lambda=10.284\,,\,\tau_r=0.7948$\\
\includegraphics[width=0.2\textwidth]{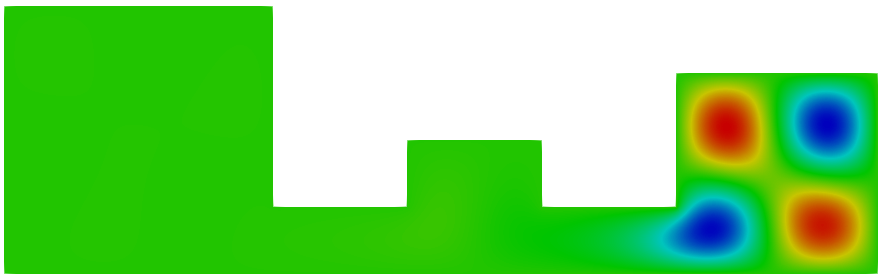}&
\includegraphics[width=0.2\textwidth]{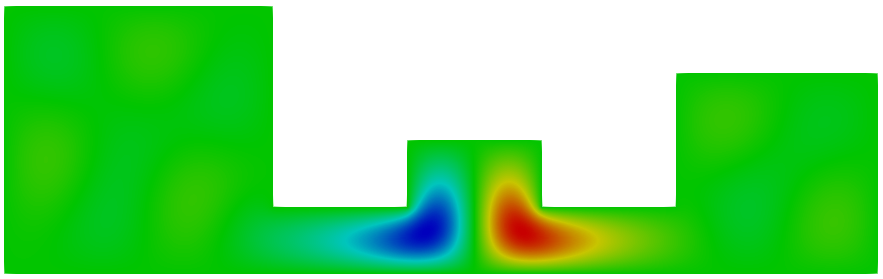}&
\includegraphics[width=0.2\textwidth]{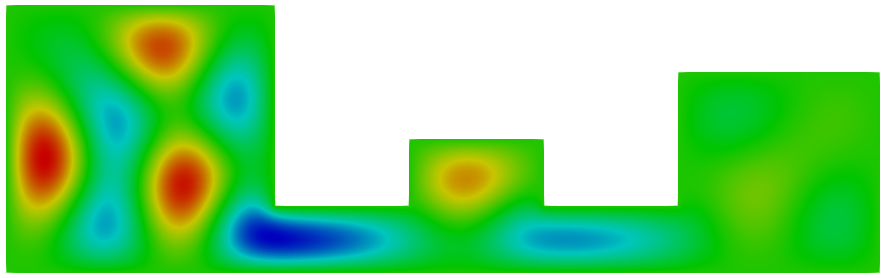}&
\includegraphics[width=0.2\textwidth]{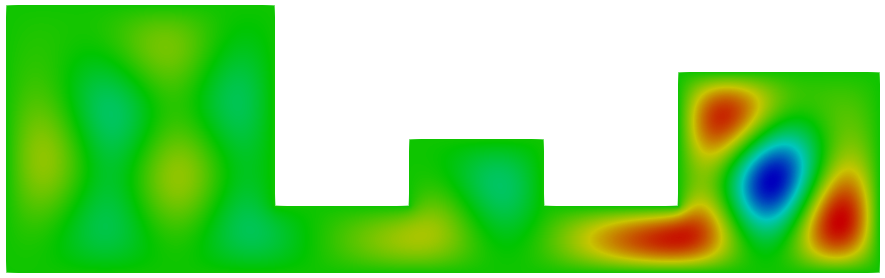}\\\hline
\end{tabular}
\end{figure}
In our localization experiments, we search for eigenvectors that are
localized with tolerance $\delta^*=1/4$, so $\tau^*=\sqrt{15}/4\approx
0.96825$, in each of the three bulbs, for eigenvalues in $[1,33]$.
There are 70 such (discrete) eigenpairs in this portion of the spectrum.
When $R$ is taken to be the left bulb, 20 of these eigenvectors satisfy
the localization tolerance.  When $R$ is taken as the right bulb, 9 of
these eigenvectors satisfy the localization tolerance.  Finally, when
$R$ is taken to be the middle bulb, only one of the these eigenvectors
satisfies the localization tolerance.

When $R$ is taken as the middle bulb, two computed eigenpairs of
$\cL_s$ are found having computed eigenvalues in
$U(1,33,1,\delta^*)$.  This search region was split into smaller
(slightly overlapping)
search regions as discussed in and before
Remark~\ref{AlternativeApproach}, each having aspect ratio
$2/(2r)+1=5$ ($r=s\delta^*=1/4$), that were tested independently using
the Bunimovich filter having $n=32$ poles.   These eigenvalues of $\cL_s$ were
$\mu=4.50447+0.98218\ii$ and $\mu=24.33196+0.96185\ii$.
The corresponding eigenpairs of $\cL$ have $\lambda=4.50292$, for
which the eigenvector is sufficiently localized in $R$, and
$\lambda=24.20972$, for which its eigenvector $\psi$ is not,
$\delta(\psi,R)=0.64489>\delta^*$ ($\tau(\psi,R)=0.76427<\tau^*$).
Contour plots of the real and imaginary parts of the eigenvectors
$\phi$ are given, together with their matched (real) eigenvectors $\psi$, in
Figure~\ref{Fig:MiddleBulb}.
Since the color scheme in each image is relative to the
range of values of the plotted function (blue for the smallest values
and red for the largest),
the ranges of function values are given explicitly in the figures for
added context.
\begin{figure}
\caption{\label{Fig:MiddleBulb} Middle bulb.
Top panel: Contour plots of the eigenvector
$\phi$ of $\cL_s$ corresponding to the true indicator
$\mu=4.50447+0.98218\ii$, and of the matched eigenvector $\psi$ of
$\cL$ (with $\lambda=4.50292$).  
Bottom panel: Contour plots of the eigenvector
$\phi$ of $\cL_s$ corresponding to the false indicator
$\mu=24.33196+0.96185\ii$, and of the matched eigenvector $\psi$ of
$\cL$ (with $\lambda=24.20972$).  The notation ``$f\in[c,d]$'' in the
images below indicates that the range of the function $f$ shown in
a given contour plot is $[c,d]$.}
\centering
\begin{tabular}{|c|c|c|}\hline
$\Re\phi\in[-$0.94,5.0e-4]&$\Im\phi\in[-$1.5e-2,4.4e-3]&$\psi\in[-$0.94,1.3e-3]\\
\includegraphics[width=0.3\textwidth]{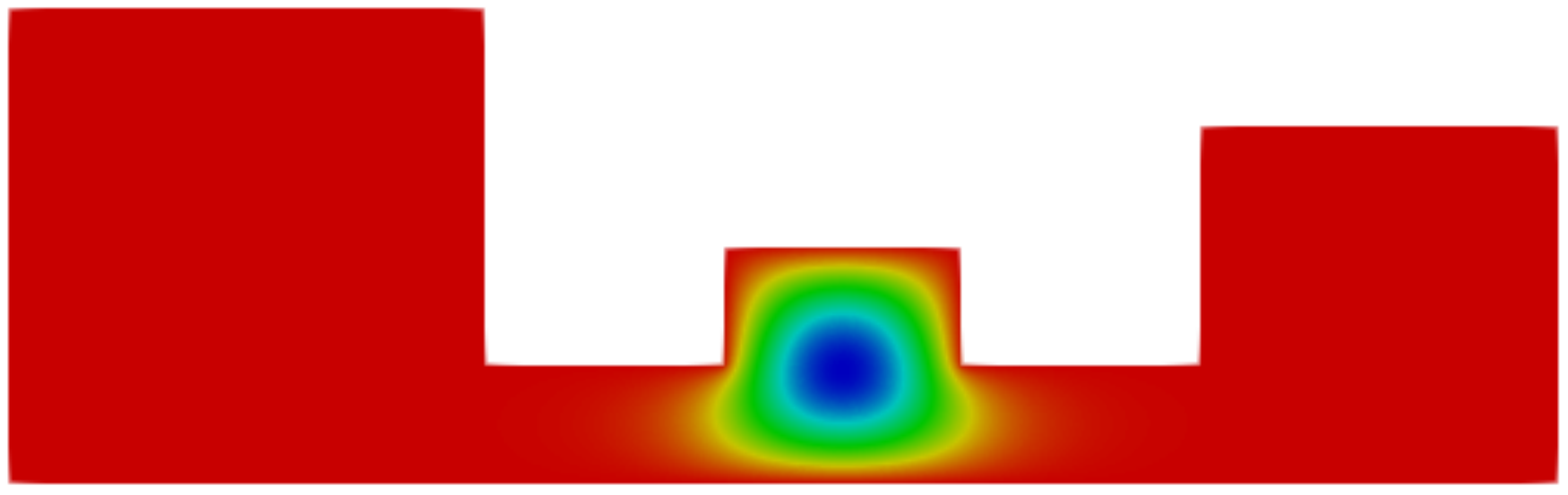}&
\includegraphics[width=0.3\textwidth]{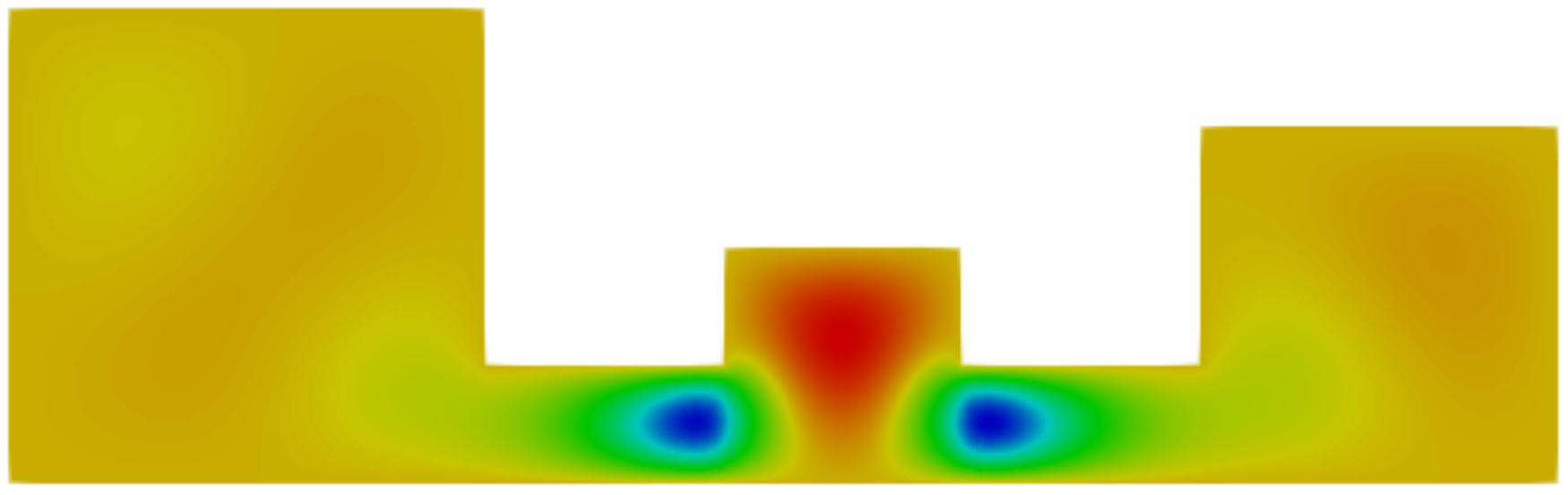}&
\includegraphics[width=0.3\textwidth]{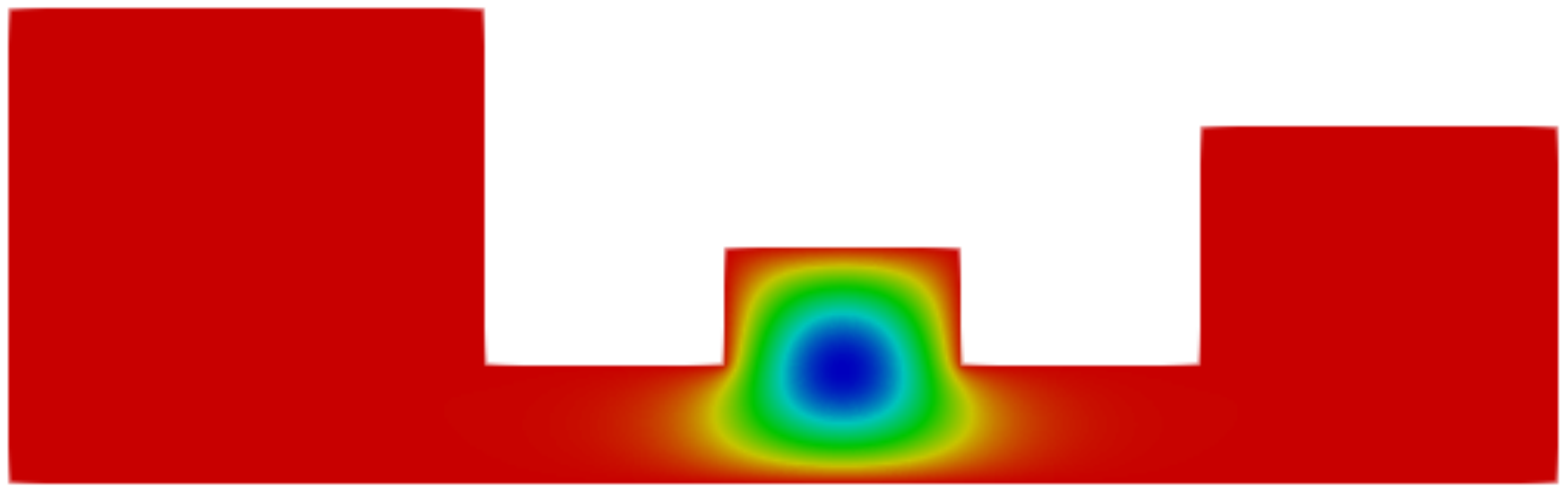}\\\hline
\end{tabular}
\begin{tabular}{|c|c|c|}\hline
$\Re\phi\in[-0.83,1.10]$&$\Im\phi\in[-0.11,0.12]$&$\psi\in[-0.84,0.84]$\\
\includegraphics[width=0.3\textwidth]{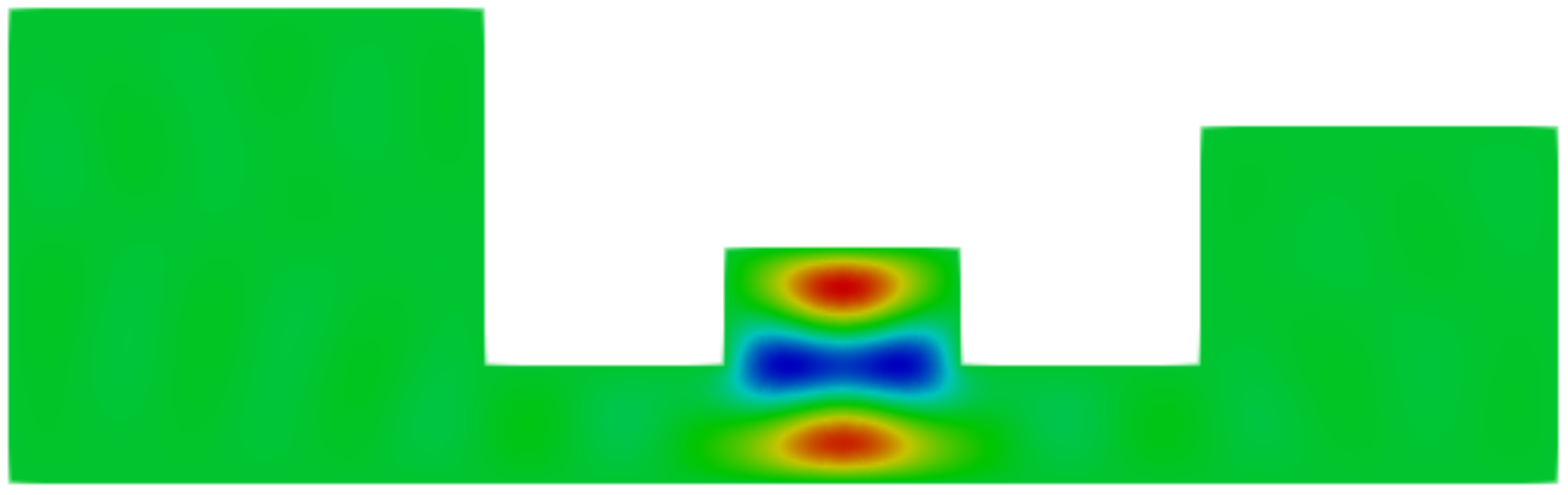}&
\includegraphics[width=0.3\textwidth]{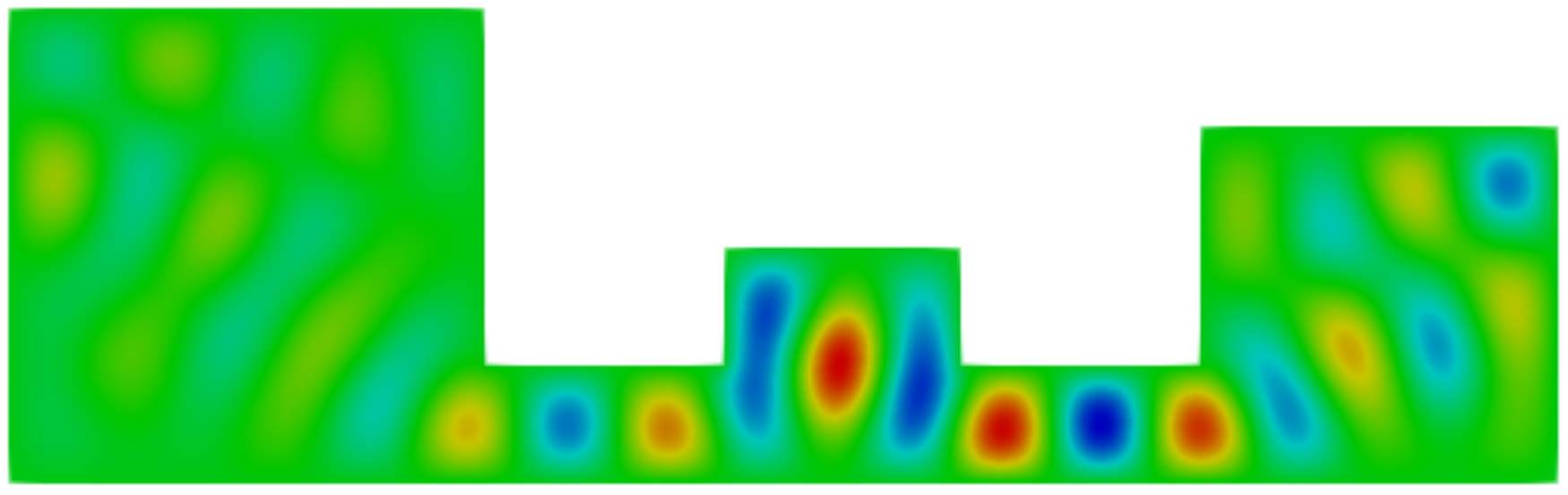}&
\includegraphics[width=0.3\textwidth]{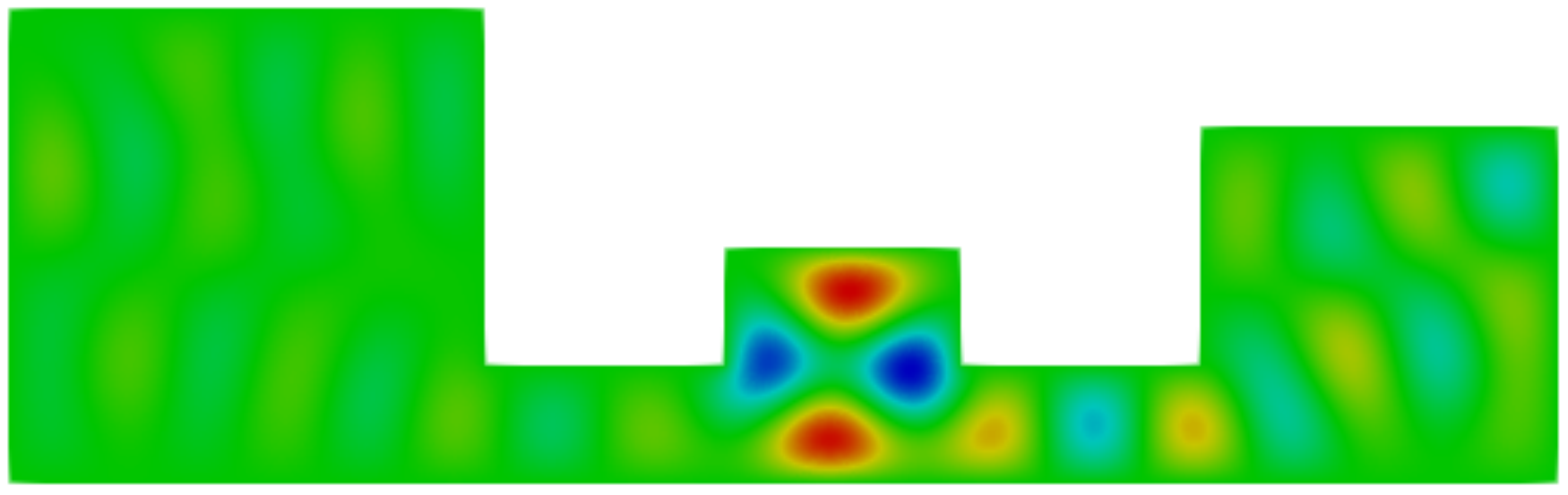}\\\hline
\end{tabular}
\end{figure}
We note that $\max|\Im\phi|$ is significantly smaller than
$\max|\Re\phi|$ for the true indicator---about two orders of
magnitude.  For the false indicator, the difference is about one order
of magnitude.

When $R$ is taken as the right bulb, 14 computed eigenpairs of $\cL_s$
were found whose eigenvalues were in $U(1,33,1,\delta^*)$.  This
search was conducted in the same way for the middle bulb.  Since there
are only 9 eigenvectors of $\cL$ that are localized in $R$ within the
tolerance $\delta^*$, five of the eigenpairs of $\cL_s$ provide
false indicators in this case.   All 14 eigenvalues of $\cL_s$ are given in
Table~\ref{RightBulbTable}, together with the eigenvalues of $\cL$
with which we have matched them and the localization measures
($\delta$-values) of the corresponding eigenvectors of $\cL$.  Those
that fail to make the final cut are emphasized using italics.
\begin{table}
\centering
\begin{tabular}{|cccc|cccc|}\hline
$\Re\mu$&$\Im\mu$&$\lambda$&$\delta(\psi,R)$&$\Re\mu$&$\Im\mu$&$\lambda$&$\delta(\psi,R)$\\\hline
      2.17146& 0.99932&2.17143&0.02610&\textit{18.61842}&\textit{0.99841}&\textit{18.62201}&\textit{0.32601}\\
      5.30605& 0.99047&5.30506&0.09800&\textit{18.98014}&\textit{0.83889}&\textit{18.98051}&\textit{0.48562}\\
      5.48270& 1.00000&5.48270&0.00638&21.86147&0.99711&21.86307&0.07712\\
      8.37589& 0.94815&8.36148&0.24017&27.32128&0.99505&27.32821&0.18688\\
  \textit{10.41327}&\textit{0.82496}&\textit{10.28357}&\textit{0.60658}&\textit{28.12641}&\textit{0.90667}&\textit{28.15248}&\textit{0.33336}\\
        10.96057&0.99984&10.96061&0.02829&28.48136&0.99618&28.48221&0.06920\\
        14.24624&0.99981&14.24624&0.01599&\textit{31.58901}&\textit{0.97649}&\textit{31.58960}&\textit{0.27542}\\\hline
\end{tabular}
\caption{\label{RightBulbTable} Computed eigenvalues $\mu$ of $\cL_s$ in
$U(1,33,1,\delta^*)$ for the right bulb, paired with matched
eigenvalues $\lambda$ of $\cL$ and localization measures $\delta(\psi,R)$ of their
corresponding eigenvectors $\phi$.  The five false indicators are highlighted
in italics.}
\end{table}
Of the five false indicators, a strong case could be made that three
of the corresponding eigenvectors $\psi$ of $\cL$ just barely failed
to make the cut---those for which $0.25<\delta(\psi,R)<0.34$.
In Figure~\ref{RightBulbFP} we provide contour plots, analogous to those in
Figure~\ref{Fig:MiddleBulb}, for the remaining two false indicators and their matched
eigenvectors of $\cL$.
\begin{figure}
\caption{\label{RightBulbFP} Right bulb. Top panel: Contour plots of the eigenvector
$\phi$ of $\cL_s$ corresponding to the false indicator
$\mu=10.41327+0.82496\ii$, and of the matched eigenvector $\psi$ of
$\cL$ (with $\lambda=10.28357$).  
Bottom panel: Contour plots of the eigenvector
$\phi$ of $\cL_s$ corresponding to the false indicator
$\mu=18.98014+0.83889\ii$, and of the matched eigenvector $\psi$ of
$\cL$ (with $\lambda=18.98051$). }
\centering
\begin{tabular}{|c|c|c|}\hline
$\Re\phi\in[-0.63,0.72]$&$\Im\phi\in[-0.28,0.11]$&$\psi\in[-0.61,0.59]$\\
\includegraphics[width=0.3\textwidth]{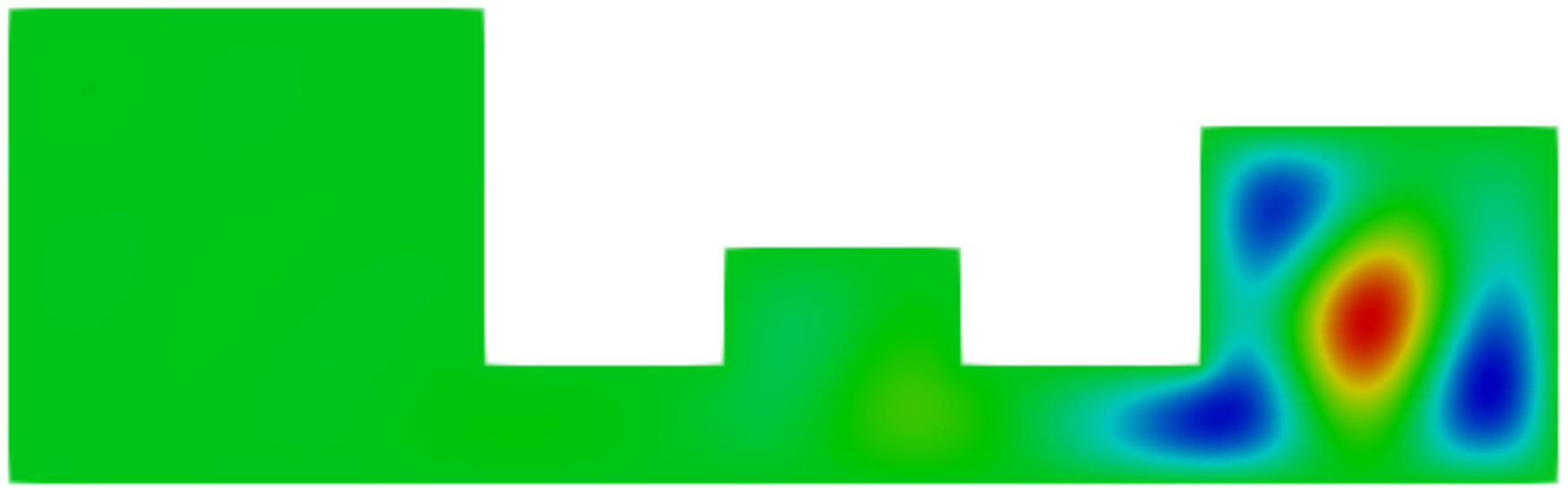}&
\includegraphics[width=0.3\textwidth]{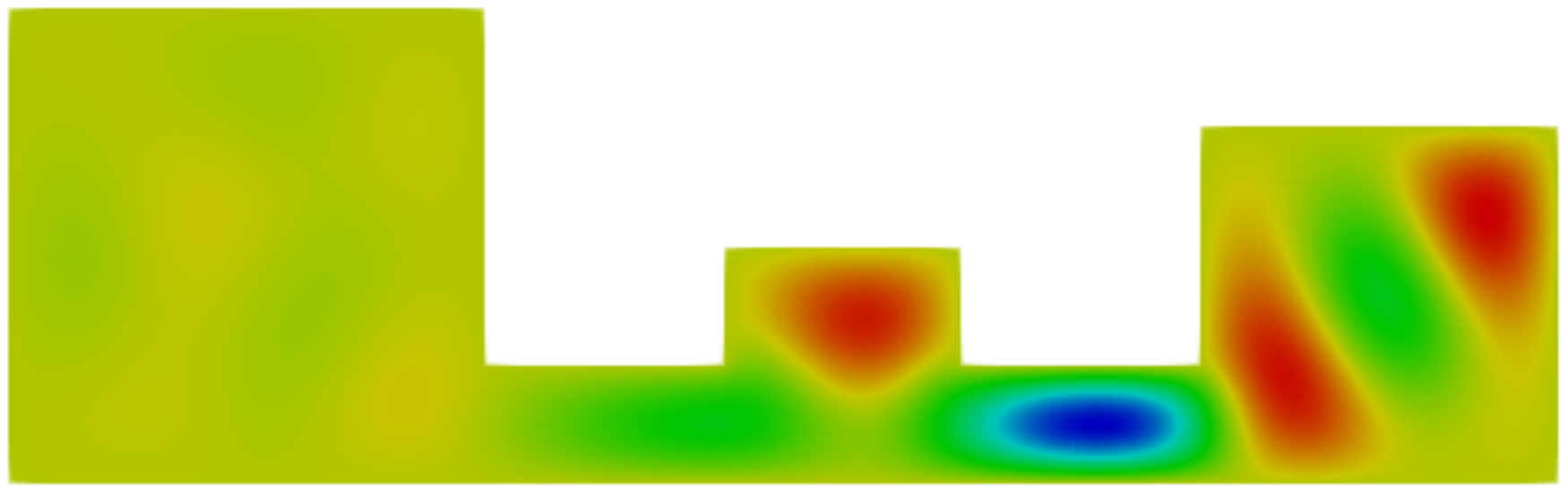}&
\includegraphics[width=0.3\textwidth]{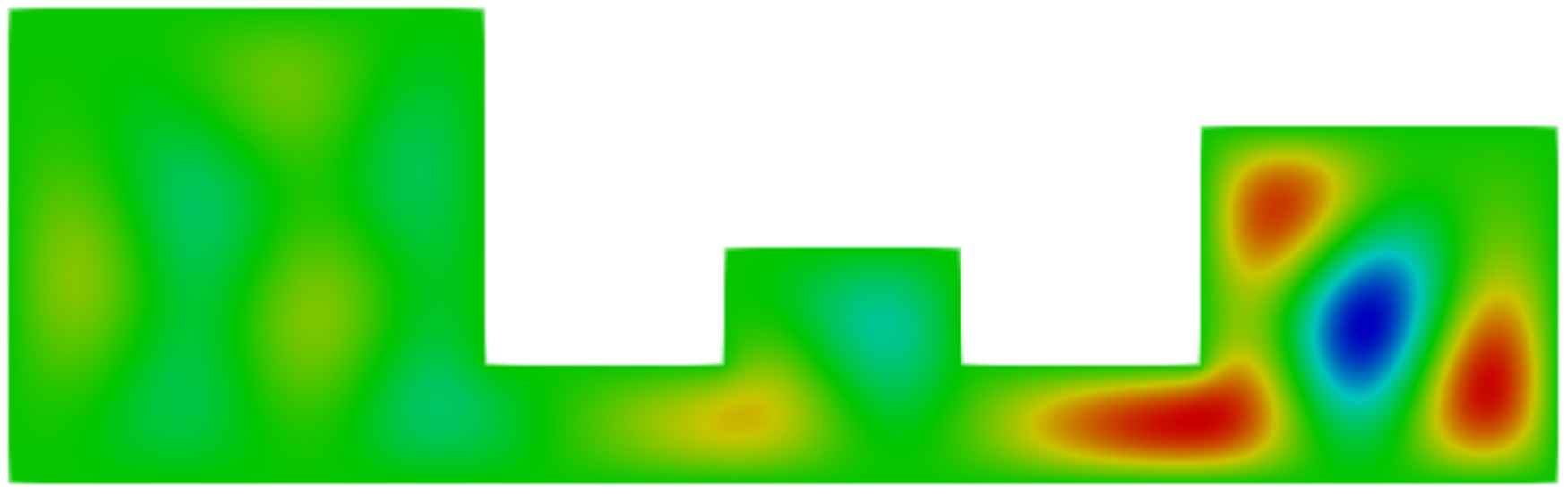}\\\hline
\end{tabular}
\begin{tabular}{|c|c|c|}\hline
$\Re\phi\in[-0.74,0.72]$&$\Im\phi\in[-0.22,0.21]$&$\psi\in[-0.73,0.71]$\\
\includegraphics[width=0.3\textwidth]{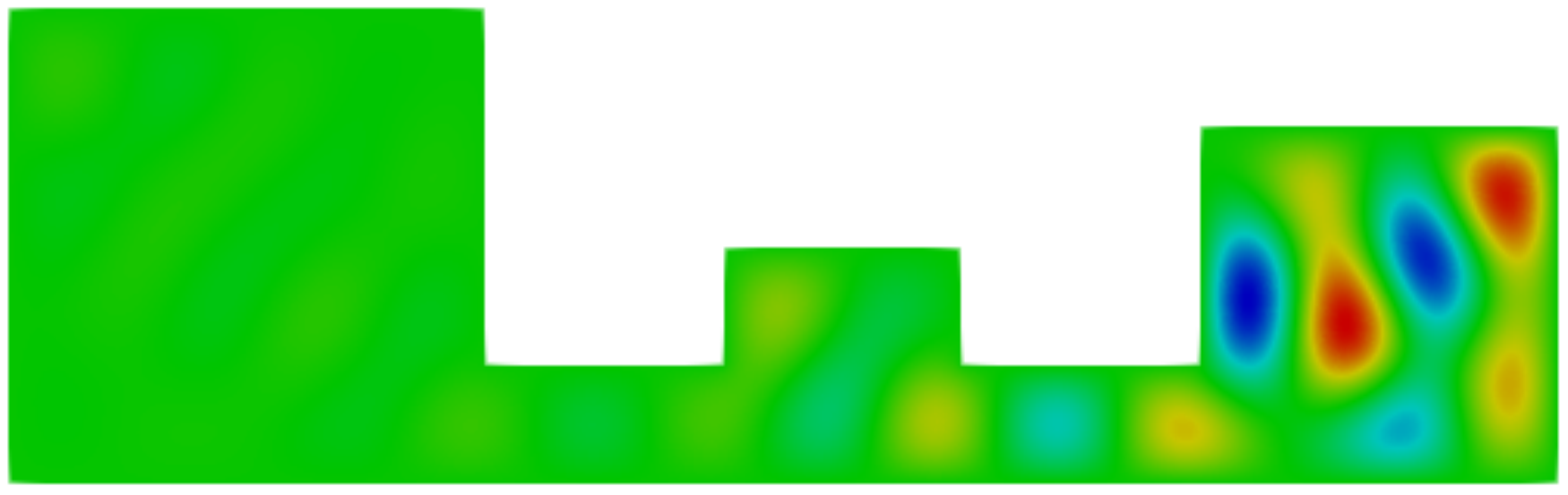}&
\includegraphics[width=0.3\textwidth]{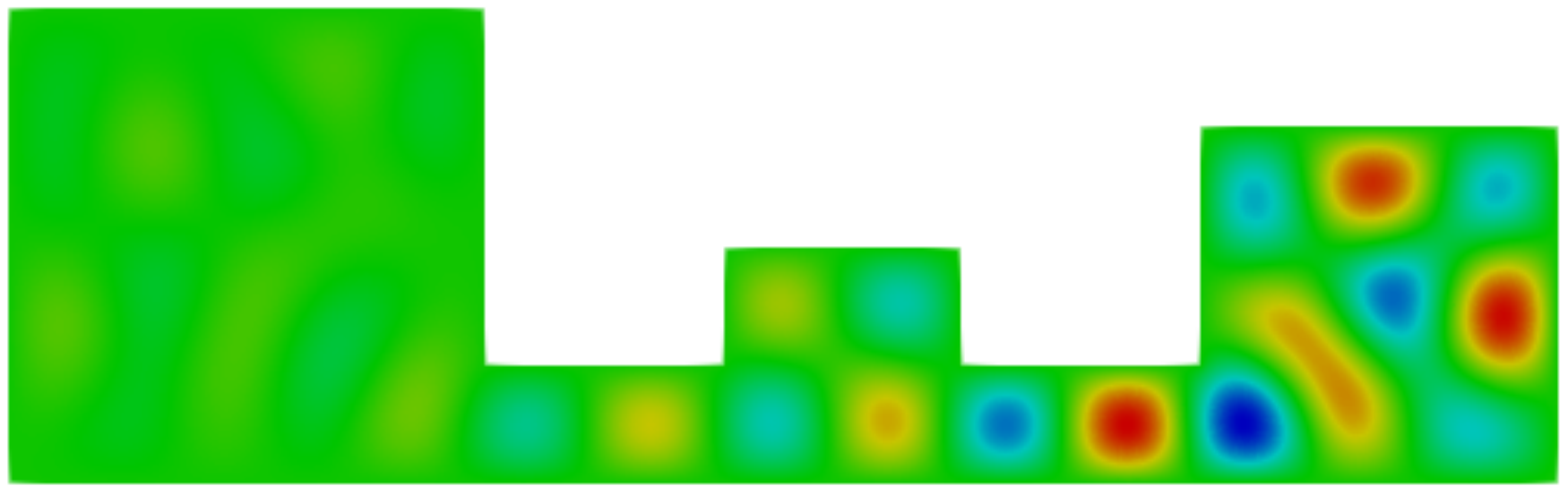}&
\includegraphics[width=0.3\textwidth]{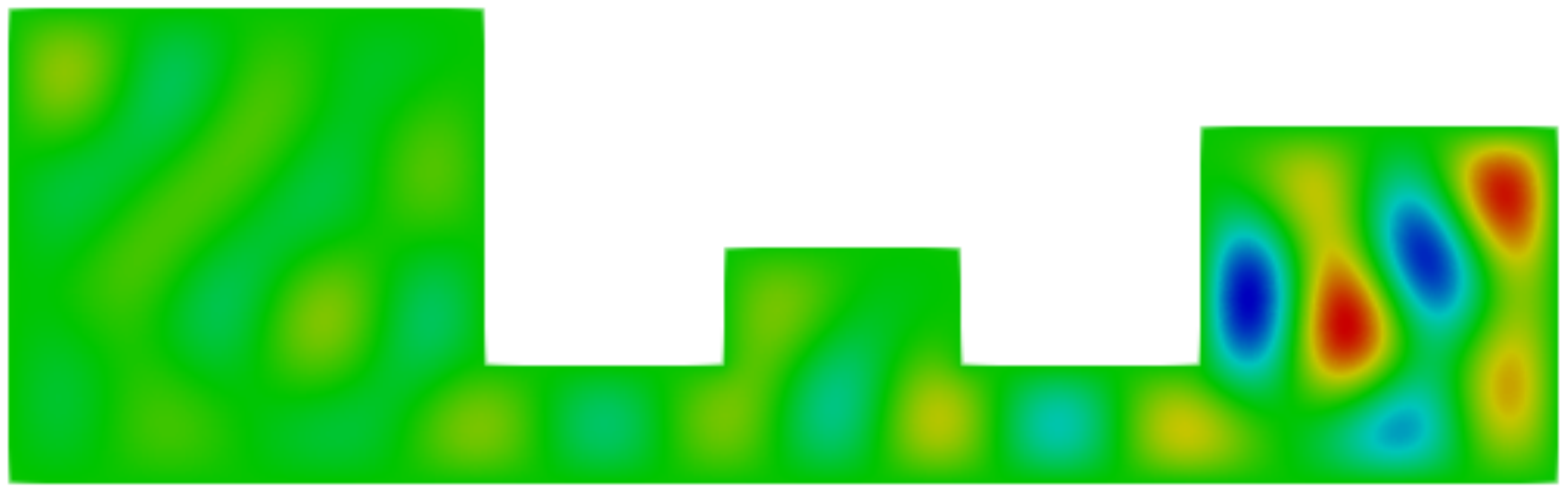}\\\hline
\end{tabular}
\end{figure}

We finally turn to the investigation of localization in the left
bulb.  For this choice of $R$, there are 20 eigenvectors of $\cL$ that
satisfy the localization tolerance.  Taking $U(1,33,1,\delta^*)$ as
before, we compute 29 eigenpairs of $\cL_s$, so there are 9 false
indicators in this case.  Table~\ref{Tab:LeftBulb_s1}, analogous to
Table~\ref{RightBulbTable} for the right bulb, provides the
corresponding numerical data.  Of the 9 false indicators, only three
(perhaps 4) really miss the mark in their predictions of localization
for the matched eigenvectors of $\cL$.
\begin{table}
  \centering
   \begin{tabular}{|cccc|cccc|}\hline
      $\Re(\mu)$&$\Im(\mu)$&$\lambda$&$\delta(\psi,R)$&$\Re(\mu)$&$\Im(\mu)$&$\lambda$&$\delta(\psi,R)$\\\hline
      1.22976&0.99989&1.22975&0.01043&17.86628&0.99922&17.86930&0.03231\\
      3.05287&0.99882&3.05277&0.03393&20.94272&0.99914&20.94341&0.06392\\
      3.08423&1.00000&3.08423&0.00046& \textit{22.62970}& \textit{0.99581}&\textit{22.58585}&\textit{0.44630}\\
      4.87100&0.99661&4.87063&0.05798&22.81320&0.99887&22.81332&0.03824\\
      6.09172&0.99482&6.09100&0.07199&24.62085&0.99792&24.62379&0.13433\\
      6.16815&1.00000&6.16815&0.00195&\textit{24.87619}&\textit{0.94232}&\textit{24.89579}&\textit{0.32321}\\
      7.72286&0.97045&7.71673&0.17517&25.27439&0.88857&25.27430&0.02228\\
      8.01849&0.99999&8.01849&0.00263&27.67153&0.99504&27.67714&0.14656\\
      \textit{10.14045}&\textit{0.84445}&\textit{10.02874}&\textit{0.54897}&30.70848&0.99543&30.70761&0.07065\\
       10.48456& 0.99995&10.48453&0.01301&30.84404&0.99999&30.84395&0.00292\\
       \textit{10.77238}&\textit{0.91878}&\textit{10.84459}&\textit{0.68925}&\textit{31.11016}& \textit{0.82898}&\textit{30.99216}&\textit{0.60879}\\
       \textit{12.33173}&\textit{0.99981}&\textit{12.33501}&\textit{0.29032}&31.92197&0.99340&31.91777&0.18332\\
      15.41686&0.99990&15.41680&0.01730&\textit{32.34692}&\textit{0.93750}&\textit{32.38753}&\textit{0.37163}\\
       \textit{15.95522}&\textit{0.93615}&\textit{16.02489}&\textit{0.26457}&32.65153&0.99340&32.65182&0.10739\\
       \textit{16.03284}&\textit{0.99941}&\textit{16.04093}&\textit{0.28087}&&&&\\\hline
    \end{tabular}  
  \caption{\label{Tab:LeftBulb_s1} Computed eigenvalues $\mu$ of $\cL_s$ in
$U(1,33,1,\delta^*)$ for the left bulb, paired with matched
eigenvalues $\lambda$ of $\cL$ and localization measures $\delta(\psi,R)$ of their
corresponding eigenvectors $\phi$.  The 9 false indicators are highlighted
in italics.}
\end{table}
We also performed the search with $s=1/2$, for eigenpairs of $\cL_s$
with eigenvalues in $U(1,33,1/2,\delta^*)$.  In lieu of the level of
detail provided in Table~\ref{Tab:LeftBulb_s1}, we summarize the
results.  The choice of $s=1/2$ yielded 28 candidates, with one fewer
false indicator, corresponding to $\lambda=30.99216$.
Changing to $s=1/4$ eliminated two more false indicators, those
corresponding to $\lambda=10.02874$ and $\lambda=10.84459$.
Further reducing  to $s=1/8,\, 1/16,\,1/32$ or $1/64$ did not eliminate any of the remaining
false indicators.

\section{Concluding Remarks}\label{Conclusions}
We have provided theory, together with detailed examples illustrating
key results, motivating Algorithm~\ref{ELAT} for
exploring eigenvector localization phenomena.  A partial realization
of Algorithm~\ref{ELAT} was described and tested on a problem exhibiting
multiple instances of localization due to domain geometry early in its
spectrum, providing a ``proof of concept'' for our approach.  What is
missing from this realization is a post-processing phase in which
eigenpairs of $\cL$ are obtained from those of $\cL_s$ automatically,
though some form of inverse iteration was suggested for this.
We have not provided \textit{numerical analysis}, i.e.
theoretical insight into the effects of discretization
errors, for our approach, but intend to pursue that in future work.
Future work will also include extensive testing of a full realization
of Algorithm~\ref{ELAT}, as well as variants discussed in Section~\ref{Template}, on a wide variety of problems, with a view
toward providing guidance on how to set key parameters.

\def\cprime{$'$}


\end{document}